\documentclass[12pt,reqno]{amsart} % \kern+3em

\usepackage{
	amsfonts,
	amssymb,
	amsmath,
	amscd,
	cancel, %\sout{};\cancel{}
	color,
	diagbox,
	float,
	graphicx,
	hyperref,
	multicol,
	multirow,
	pb-diagram,
	soul,
	ulem
}

\newcommand{\aspdf}{1}

\newcommand{\figureone}{
	\ifnum\aspdf=1
		\includegraphics{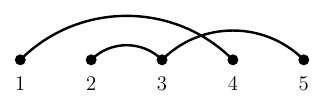}\qquad
		\includegraphics{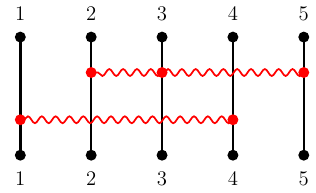}
	\else
		\begin{tikzpicture}[scale=1.2]
			\arcpartition[tkzpic=0]{{1,4},{2,3,5}}
		\end{tikzpicture}
		\begin{tikzpicture}[scale=1.2]
			\permutation[tkzpic=0]{1,2,3,4,5}
			\tie[snake=true,snakelen=4,style=solid]{{1,0.3},{4,0.3}}
			\tie[snake=true,snakelen=4,style=solid]{{2,0.7},{3,0.7},{5,0.7}}
		\end{tikzpicture}
	\fi
}

\newcommand{\figuretwo}{
	\ifnum\aspdf=1
		\includegraphics{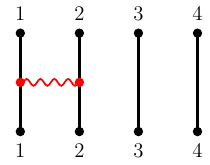}\kern+1.4em
		\includegraphics{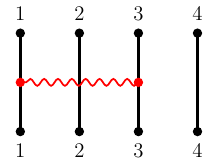}\kern+1.4em
		\includegraphics{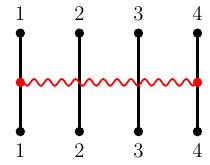}\kern+1.4em
		\includegraphics{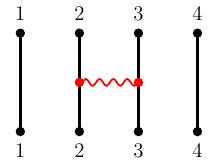}\kern+1.4em
		\includegraphics{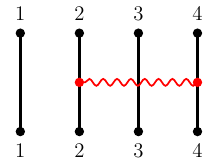}\kern+1.4em
		\includegraphics{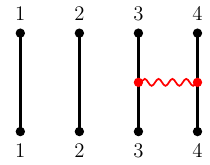}
	\else
		\begin{tikzpicture}
			\permutation[tkzpic=0]{1,2,3,4}
			\tie[snake=true,snakelen=4,style=solid]{1,2}
		\end{tikzpicture}\kern+1.4em
		\begin{tikzpicture}
			\permutation[tkzpic=0]{1,2,3,4}
			\tie[snake=true,snakelen=4,style=solid]{1,3}
		\end{tikzpicture}\kern+1.4em
		\begin{tikzpicture}
			\permutation[tkzpic=0]{1,2,3,4}
			\tie[snake=true,snakelen=4,style=solid]{1,4}
		\end{tikzpicture}\kern+1.4em
		\begin{tikzpicture}
			\permutation[tkzpic=0]{1,2,3,4}
			\tie[snake=true,snakelen=4,style=solid]{2,3}
		\end{tikzpicture}\kern+1.4em
		\begin{tikzpicture}
			\permutation[tkzpic=0]{1,2,3,4}
			\tie[snake=true,snakelen=4,style=solid]{2,4}
		\end{tikzpicture}\kern+1.4em
		\begin{tikzpicture}
			\permutation[tkzpic=0]{1,2,3,4}
			\tie[snake=true,snakelen=4,style=solid]{3,4}
		\end{tikzpicture}
	\fi
}

\newcommand{\figurethr}{
	\ifnum\aspdf=1\[
		\vcdraw{	\includegraphics{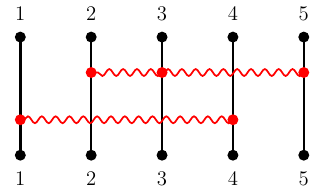}}\kern+0.7em=
		\vcdraw{	\includegraphics{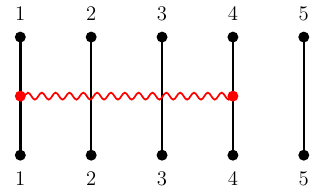}}\kern+0.7em\cdot\kern+0.4em
		\vcdraw{\includegraphics{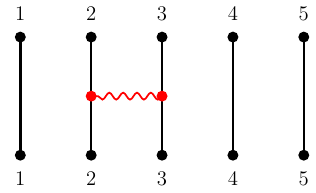}}\kern+0.7em\cdot\kern+0.4em
		\vcdraw{\includegraphics{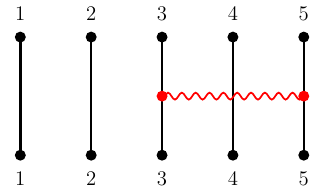}}\]
	\else\[
		\vcdraw{	\begin{tikzpicture}[scale=1.2]
			\permutation[tkzpic=0]{1,2,3,4,5}
			\tie[snake=true,snakelen=4,style=solid]{{1,0.3},{4,0.3}}
			\tie[snake=true,snakelen=4,style=solid]{{2,0.7},{3,0.7},{5,0.7}}
		\end{tikzpicture}}\kern+0.7em=
		\vcdraw{\begin{tikzpicture}[scale=1.2]
			\permutation[tkzpic=0]{1,2,3,4,5}
			\tie[snake=true,snakelen=4,style=solid]{1,4}
		\end{tikzpicture}}\kern+0.7em\cdot\kern+0.4em
		\vcdraw{\begin{tikzpicture}[scale=1.2]
			\permutation[tkzpic=0]{1,2,3,4,5}
			\tie[snake=true,snakelen=4,style=solid]{2,3}
		\end{tikzpicture}}\kern+0.7em\cdot\kern+0.4em
		\vcdraw{\begin{tikzpicture}[scale=1.2]
			\permutation[tkzpic=0]{1,2,3,4,5}
			\tie[snake=true,snakelen=4,style=solid]{3,5}
		\end{tikzpicture}}\]
	\fi
}

\newcommand{\figurefou}{
	\ifnum\aspdf=1\[
		\vcdraw{\includegraphics{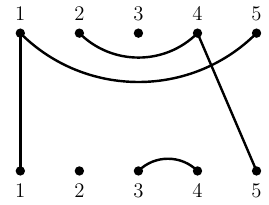}}\kern+0.7em*\kern+0.4em
		\vcdraw{\includegraphics{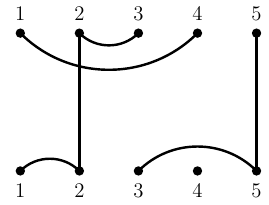}}\kern+0.7em=\kern+0.4em
		\vcdraw{\includegraphics{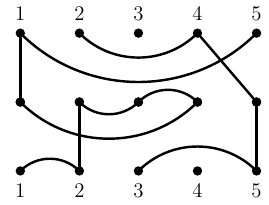}}\kern+0.7em=\kern+0.4em
		\vcdraw{\includegraphics{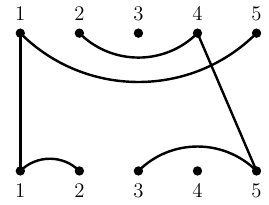}}\]
	\else\[
		\vcdraw{\vpartition[height=1.4,type=2]{{-1,1,5},{2,4,-5},{-3,-4}}}*\kern+0.7em*\kern+0.4em
		\vcdraw{\vpartition[height=1.4,type=2]{{1,4},{3,2,-2,-1},{5,-5,-3}}}\kern+0.7em=\kern+0.4em
		\vcdraw{\begin{tikzpicture}
			\vpartition[floor=1,height=0.7,tkzpic=0,type=1]{{-1,1,5},{2,4,-5},{-3,-4}}
			\vpartition[height=0.7,tkzpic=0,type=-1]{{1,4},{3,2,-2,-1},{5,-5,-3}}
		\end{tikzpicture}}\kern+0.7em=\kern+0.4em
		\vcdraw{\vpartition[height=1.4,type=2]{{5,1,-1,-2},{2,4,-5,-3}}}\]
	\fi
}

\newcommand{\figurefiv}{
	\ifnum\aspdf=1\[
		\vcdraw{\includegraphics{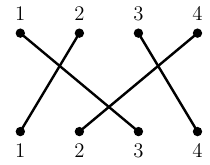}}\kern+0.6em=\kern+0.4em
		\vcdraw{\includegraphics{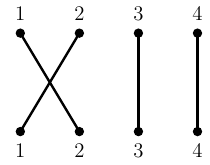}}\kern+0.6em*\kern+0.4em
		\vcdraw{\includegraphics{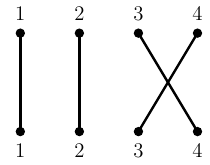}}\kern+0.6em*\kern+0.4em
		\vcdraw{\includegraphics{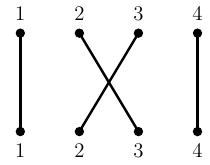}}\]
	\else\[
		\vcdraw{\permutation{3,1,4,2}}=\kern-1em
		\vcdraw{\permutation{2,1,3,4}}*\kern-1em
		\vcdraw{\permutation{1,2,4,3}}*\kern-1em
		\vcdraw{\permutation{1,3,2,4}}\]
	\fi
}

\newcommand{\figuresix}{
	\ifnum\aspdf=1\[
		\vcdraw{\includegraphics{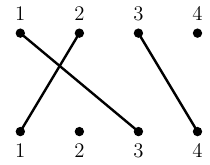}}\kern+0.6em=\kern+0.4em
		\vcdraw{\includegraphics{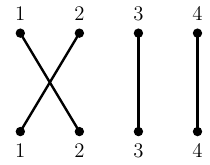}}\kern+0.6em*\kern+0.4em
		\vcdraw{\includegraphics{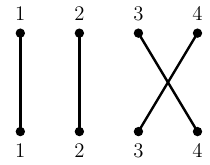}}\kern+0.6em*\kern+0.4em
		\vcdraw{\includegraphics{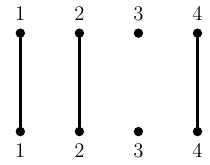}}\kern+0.6em*\kern+0.4em
		\vcdraw{\includegraphics{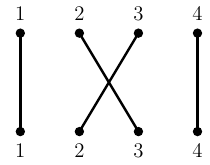}}\]
	\else\[
		\vcdraw{\vpartition[type=2]{{1,-3},{2,-1},{3,-4}}}=\kern-1.2em
		\vcdraw{\permutation{2,1,3,4}}*\kern-1em
		\vcdraw{\permutation{1,2,4,3}}*\kern-0.3em
		\vcdraw{\vpartition[type=2]{{1,-1},{2,-2},{4,-4}}}\,*\kern-1.1em
		\vcdraw{\permutation{1,3,2,4}}\]
	\fi
}

\newcommand{\figuresev}{
	\ifnum\aspdf=1\[
		\vcdraw{\includegraphics{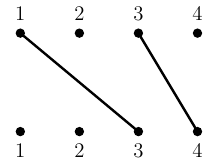}}\kern+0.6em=\kern+0.4em
		\vcdraw{\includegraphics{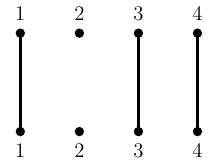}}\kern+0.6em*\kern+0.4em
		\vcdraw{\includegraphics{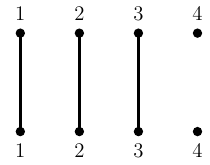}}\kern+0.6em*\kern+0.4em
		\vcdraw{\includegraphics{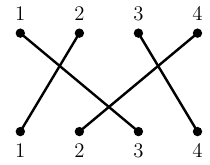}}\kern+0.6em=\kern+0.4em
		\vcdraw{\includegraphics{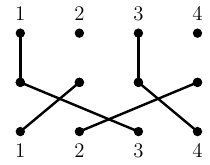}}\]
	\else\[
		\vcdraw{\vpartition[type=2]{{1,-3},{3,-4}}}=\kern-0.5em
		\vcdraw{\vpartition[nstr=4,type=2]{{1,-1},{3,-3},{4,-4}}}\,*\kern-0.5em
		\vcdraw{\vpartition[nstr=4,type=2]{{1,-1},{2,-2},{3,-3}}}\,*\kern-0.5em
		\vcdraw{\vpartition[type=2]{{1,-3},{2,-1},{3,-4},{4,-2}}}=\kern+0.5em
		\vcdraw{\begin{tikzpicture}
			\vpartition[floor=1,height=0.5,nstr=4,tkzpic=0,type=1]{{1,-1},{3,-3}}
			\vpartition[floor=0,height=0.5,type=-1,tkzpic=0]{{1,-3},{2,-1},{3,-4},{4,-2}}
		\end{tikzpicture}}\]
	\fi
}

\newcommand{\figureeig}{
	\ifnum\aspdf=1
		\[\begin{array}{c}
			\includegraphics{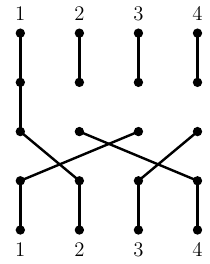}\qquad
			\includegraphics{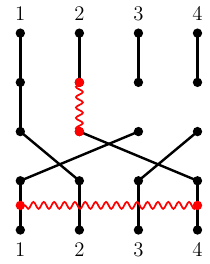}\qquad
			\includegraphics{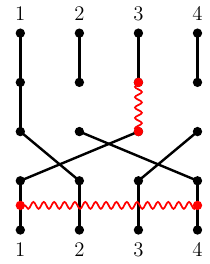}\qquad
			\includegraphics{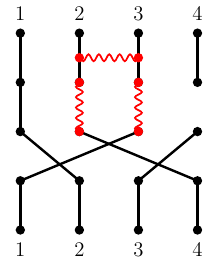}\qquad
			\includegraphics{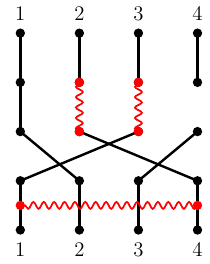}
			\\[0.5cm]
			\includegraphics{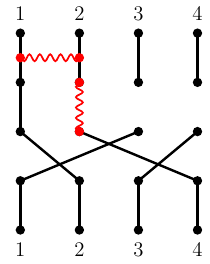}\qquad
			\includegraphics{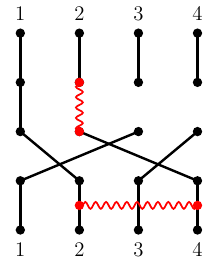}\qquad
			\includegraphics{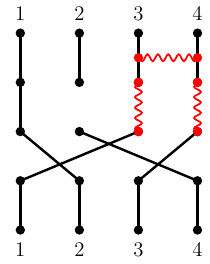}\qquad
			\includegraphics{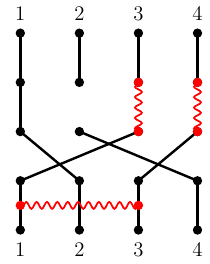}
		\end{array}\]
	\else
		\begin{tikzpicture}
			\permutation[floor=3,height=0.5,tkzpic=0,type=1]{1,2,3,4}
			\vpartition[bulla=0,bullb=0,floor=2,height=0.5,nstr=4,tkzpic=0,type=0]{{1,-1}}
			\permutation[floor=1,height=0.5,tkzpic=0,type=0]{2,4,1,3}
			\permutation[bulla=0,height=0.5,tkzpic=0,type=-1]{1,2,3,4}
			\tie[snake=true,style=solid]{{1,1.75},{2,1.75}}
			\tie[snake=true,style=solid]{{2,1.5},{2,1}}
		\end{tikzpicture}
	\qquad
		\begin{tikzpicture}
			\permutation[floor=3,height=0.5,tkzpic=0,type=1]{1,2,3,4}
			\vpartition[bulla=0,bullb=0,floor=2,height=0.5,nstr=4,tkzpic=0,type=0]{{1,-1}}
			\permutation[floor=1,height=0.5,tkzpic=0,type=0]{2,4,1,3}
			\permutation[bulla=0,height=0.5,tkzpic=0,type=-1]{1,2,3,4}
			\tie[snake=true,style=solid]{{2,0.25},{4,0.25}}
			\tie[snake=true,style=solid]{{2,1.5},{2,1}}
		\end{tikzpicture}
	\qquad
		\begin{tikzpicture}
			\permutation[floor=3,height=0.5,tkzpic=0,type=1]{1,2,3,4}
			\vpartition[bulla=0,bullb=0,floor=2,height=0.5,nstr=4,tkzpic=0,type=0]{{1,-1}}
			\permutation[floor=1,height=0.5,tkzpic=0,type=0]{2,4,1,3}
			\permutation[bulla=0,height=0.5,tkzpic=0,type=-1]{1,2,3,4}
			\tie[snake=true,style=solid]{{3,1.75},{4,1.75}}
			\tie[snake=true,style=solid]{{3,1.5},{3,1}}
			\tie[snake=true,style=solid]{{4,1.5},{4,1}}
		\end{tikzpicture}
	\qquad
		\begin{tikzpicture}
			\permutation[floor=3,height=0.5,tkzpic=0,type=1]{1,2,3,4}
			\vpartition[bulla=0,bullb=0,floor=2,height=0.5,nstr=4,tkzpic=0,type=0]{{1,-1}}
			\permutation[floor=1,height=0.5,tkzpic=0,type=0]{2,4,1,3}
			\permutation[bulla=0,height=0.5,tkzpic=0,type=-1]{1,2,3,4}
			\tie[snake=true,style=solid]{{1,0.25},{3,0.25}}
			\tie[snake=true,style=solid]{{3,1.5},{3,1}}
			\tie[snake=true,style=solid]{{4,1.5},{4,1}}
		\end{tikzpicture}
	\fi
}

\newcommand{\figurenin}{
	\ifnum\aspdf=1
		\includegraphics{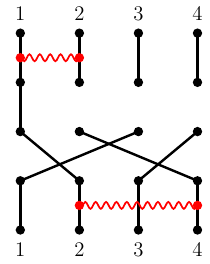}\qquad
		\includegraphics{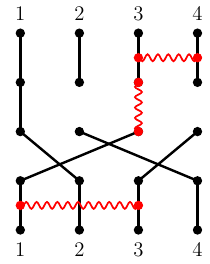}
	\else
		\begin{tikzpicture}
			\permutation[floor=3,height=0.5,tkzpic=0,type=1]{1,2,3,4}
			\vpartition[bulla=0,bullb=0,floor=2,height=0.5,nstr=4,tkzpic=0,type=0]{{1,-1}}
			\permutation[floor=1,height=0.5,tkzpic=0,type=0]{2,4,1,3}
			\permutation[bulla=0,height=0.5,tkzpic=0,type=-1]{1,2,3,4}
			\tie[snake=true,style=solid]{{1,1.75},{2,1.75}}
			\tie[snake=true,style=solid]{{2,0.25},{4,0.25}}
		\end{tikzpicture}
	\qquad
		\begin{tikzpicture}
			\permutation[floor=3,height=0.5,tkzpic=0,type=1]{1,2,3,4}
			\vpartition[bulla=0,bullb=0,floor=2,height=0.5,nstr=4,tkzpic=0,type=0]{{1,-1}}
			\permutation[floor=1,height=0.5,tkzpic=0,type=0]{2,4,1,3}
			\permutation[bulla=0,height=0.5,tkzpic=0,type=-1]{1,2,3,4}
			\tie[snake=true,style=solid]{{3,1.75},{4,1.75}}
			\tie[snake=true,style=solid]{{3,1.5},{3,1}}
			\tie[snake=true,style=solid]{{1,0.25},{3,0.25}}
		\end{tikzpicture}
	\fi
}

\newcommand{\figureten}{
	\ifnum\aspdf=1\[
		I\,\,=\vcdraw{\includegraphics{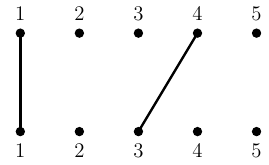}}\,\,,\qquad
		J\,\,=\vcdraw{\includegraphics{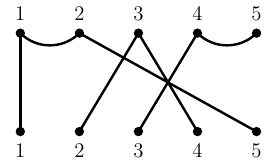}}\,\,,\qquad
		(I,J)\,\,=\vcdraw{\includegraphics{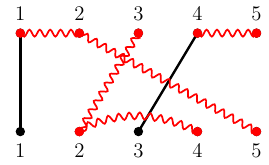}}
	\]\else\[
		I\,\,=\vcdraw{\vpartition[nstr=5,type=2]{{1,-1},{4,-3}}}\qquad
		J\,\,=\vcdraw{\vpartition[nstr=5,type=2]{{-1,1,2,-5},{-2,3,-4},{-3,4,5}}}\qquad
		(I,J)\,\,=\vcdraw{\begin{tikzpicture}
			\vpartition[nstr=5,type=2,tkzpic=0]{{1,-1},{4,-3}}
			\tie[snake=true,style=solid]{{1,1},{2,1},{5,0}}
			\tie[snake=true,style=solid]{{2,0},{3,1}}
			\tie[snake=true,style=solid,bend=25]{{2,0},{4,0}}
			\tie[snake=true,style=solid]{{4,1},{5,1}}
		\end{tikzpicture}}
	\]\fi
}

\newcommand{\figureele}{
	\ifnum\aspdf=1\[
		\begin{array}{c}
			g^*=\!\vcdraw{\includegraphics{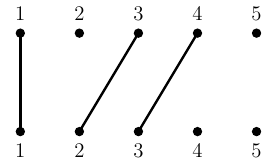}}\,\,,\qquad
			g^*_p=\!\vcdraw{\includegraphics{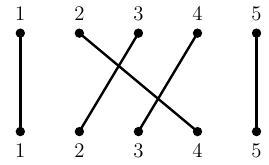}}
			\\[1cm]
			e=\!\vcdraw{\includegraphics{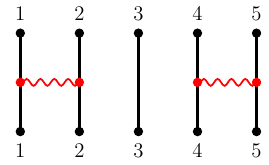}}\,\,,\qquad
			e'=\!\vcdraw{\includegraphics{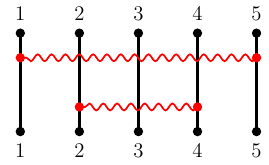}}\,\,,\qquad
			r=\!\vcdraw{\includegraphics{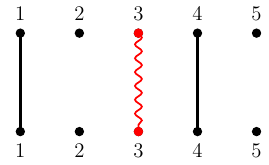}}	
		\end{array}
	\]\else\[
		\begin{array}{c}
			g^*=\!\vcdraw{\begin{tikzpicture}
				\vpartition[nstr=5,tkzpic=0,type=2]{{1,-1},{4,-3},{3,-2}}
			\end{tikzpicture}}\qquad
			g^*_p=\!\vcdraw{\begin{tikzpicture}
				\vpartition[nstr=5,tkzpic=0,type=2]{{1,-1},{4,-3},{3,-2},{2,-4},{5,-5}}
			\end{tikzpicture}}
			\\[1cm]
			e=\!\vcdraw{\begin{tikzpicture}
				\permutation[tkzpic=0]{1,2,3,4,5}
				\tie[snake=true,snakelen=4,style=solid]{1,2}
				\tie[snake=true,snakelen=4,style=solid]{4,5}
			\end{tikzpicture}}\qquad
			e'=\!\vcdraw{\begin{tikzpicture}
				\permutation[tkzpic=0]{1,2,3,4,5}
				\tie[snake=true,snakelen=4,style=solid]{{1,0.75},{5,0.75}}
				\tie[snake=true,snakelen=4,style=solid]{{2,0.25},{4,0.25}}
			\end{tikzpicture}}\qquad
			r=\!\vcdraw{\begin{tikzpicture}
				\strands[nstr=5,tkzpic=0,type=2,tiesnake=true,tiesnakelen=4,tiestyle=solid]{r1-q2-r4}
			\end{tikzpicture}}
			r'=\!\vcdraw{\begin{tikzpicture}
				\strands[nstr=5,tkzpic=0,type=2,tiesnake=true,tiesnakelen=4,tiestyle=solid]{r1-r2-r4}
				\tie[snake=true,snakelen=4,style=solid]{{3,0},{3,1}}
			\end{tikzpicture}}
		\end{array}
	\]\fi
}

\newcommand{\figuretwe}{
	\ifnum\aspdf=1\[
		\!g\,=\vcdraw{\includegraphics{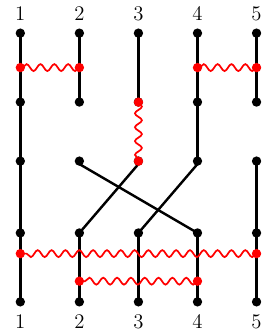}}
	\]\else\[
		g\,=\,\vcdraw{\begin{tikzpicture}
			\permutation[bullb=0,floor=1,height=0.7,tkzpic=0,type=1]{1,2,3,4,5}
			\tie[floor=1,height=0.7,snake=true,snakelen=4,style=solid]{1,2}
			\tie[floor=1,height=0.7,snake=true,snakelen=4,style=solid]{4,5}
			\strands[floor=0.1,height=0.6,nstr=5,rookextend=0,tkzpic=0,type=0,
				tiesnake=true,tiesnakelen=4,tiestyle=solid]{r1-r2-r4}
			\vpartition[bulla=0,bullb=0,floor=-0.9,height=0.7,nstr=5,tkzpic=0,type=0]
				{{1,-1},{4,-3},{3,-2},{2,-4},{5,-5}}
			\tie[floor=0.1,height=0.6,snake=true,snakelen=4,style=solid]{{3,0.065},{3,1.065}}
			\permutation[floor=-1.9,height=0.7,tkzpic=0,type=-1]{1,2,3,4,5}
			\tie[floor=-1.9,height=0.7,snake=true,snakelen=4,style=solid]{{1,0.7},{5,0.7}}
			\tie[floor=-1.9,height=0.7,snake=true,snakelen=4,style=solid]{{2,0.3},{4,0.3}}
		\end{tikzpicture}}
	\]\fi
}

\newcommand{\figureonethr}{
	\ifnum\aspdf=1\[
		\vcdraw{\includegraphics{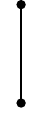}}\!\!\cdots\!\!
		\vcdraw{\includegraphics{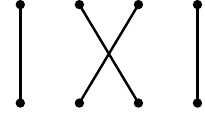}}\cdots\!\!
		\vcdraw{\includegraphics{pics/043.pdf}}\]
	\else\[
		\vcdraw{\vpartition[type=0]{{1,-1},{2,-2}}}
		\kern-0.5em\cdots\kern-1em
		\vcdraw{\vpartition[type=0]{{1,-1},{2,-3},{3,-2},{4,-4}}}
		\kern-0.3em\cdots\kern-1em		
		\vcdraw{\vpartition[type=0]{{1,-1},{2,-2}}}\]
	\fi
}

\newcommand{\figureonefou}{
	\ifnum\aspdf=1\[
		\vcdraw{\includegraphics{pics/043.pdf}}\!\!\cdots\!\!
		\vcdraw{\includegraphics{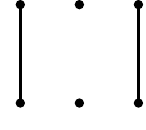}}\cdots\!\!	
		\vcdraw{\includegraphics{pics/043.pdf}}\]
	\else\[
		\vcdraw{\vpartition[type=0]{{1,-1},{2,-2}}}
		\kern-0.3em\cdots\kern-0.8em
		\vcdraw{\vpartition[type=0]{{1,-1},{3,-3}}}
		\kern-0.3em\cdots\kern-1em		
		\vcdraw{\vpartition[type=0]{{1,-1},{2,-2}}}\]
	\fi
}

\newcommand{\figureonefiv}{
	\ifnum\aspdf=1\[
		\includegraphics{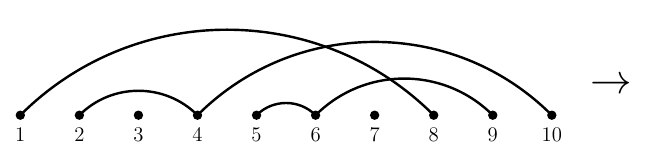}
		\includegraphics{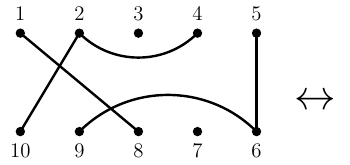}
		\includegraphics{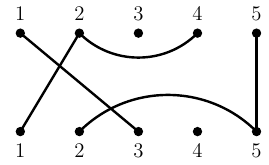}
	\]\else\[
		\begin{tikzpicture}
			\arcpartition[tkzpic=0]{{1,8},{2,4,10},{5,6,9}}
			\node at(6,0.3){$\to$};
		\end{tikzpicture}
		\begin{tikzpicture}
			\vpartition[
				customlabelsabove={1,2,3,4,5},
				customlabelsbelow={10,9,8,7,6},
				tkzpic=0,
				type=6
			]{{-1,2,4},{5,-5,-2},{1,-3}}
			\node at(3,0.3){$\leftrightarrow$};
		\end{tikzpicture}
		\vpartition[type=2]{{-1,2,4},{5,-5,-2},{1,-3}}
	\]\fi
}

\newcommand{\figureonesix}{
	\ifnum\aspdf=1\[\left(\,\,
		\vcdraw{\includegraphics{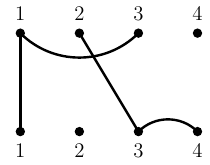}}\,\,\,,\,
		\vcdraw{\includegraphics{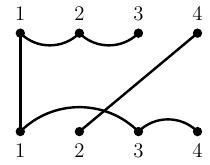}}\,\,\,\right)\,\,=\,\,
		\vcdraw{\includegraphics{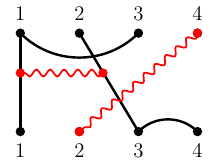}}\]
	\else\[\left(\!\!\!
		\vcdraw{\vpartition[type=2]{{-1,1,3},{2,-3,-4}}},\!\!\!
		\vcdraw{\vpartition[type=2]{{-4,-3,-1,1,2,3},{4,-2}}}\right)=
		\vcdraw{\begin{tikzpicture}
			\vpartition[tkzpic=0,type=2]{{-1,1,3},{2,-3,-4}}
			\tie[height=0.7,snake=true,snakelen=4,style=solid]{{1,0.85},{2.4,0.85}}
			\tie[height=0.7,snake=true,snakelen=4,style=solid]{{2,0},{4,1.43}}
		\end{tikzpicture}}\]
	\fi
}

\newcommand{\figureonesev}{
	\ifnum\aspdf=1\[
		\vcdraw{\includegraphics{pics/043.pdf}}\!\!\cdots\!\!
		\vcdraw{\includegraphics{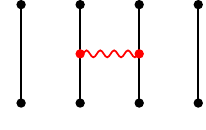}}\cdots\!\!
		\vcdraw{\includegraphics{pics/043.pdf}}\]
	\else\[
		\vcdraw{\vpartition[type=0]{{1,-1}}}
		\kern-0.5em\cdots\kern-1em
		\vcdraw{\begin{tikzpicture}
			\permutation[tkzpic=0,type=0]{1,2,3,4}
			\tie[snake=true,snakelen=4,style=solid]{{2,0.5},{3,0.5}}
		\end{tikzpicture}}
		\kern-0.3em\cdots\kern-1em		
		\vcdraw{\vpartition[type=0]{{1,-1},{2,-2}}}\]
	\fi
}

\newcommand{\figureoneeig}{
	\ifnum\aspdf=1\[
		\vcdraw{\includegraphics{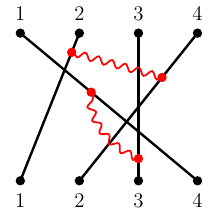}}\,\,\,\,=\,\,\,
		\vcdraw{\includegraphics{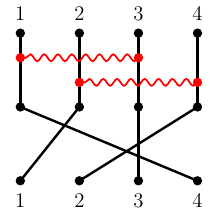}}\,\,\,\,=\,\,\,
		\vcdraw{\includegraphics{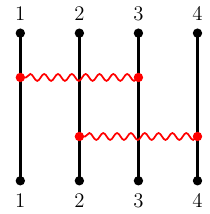}}\,\,\,\,*\,\,\,
		\vcdraw{\includegraphics{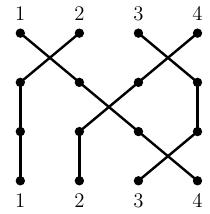}}
	\]\else\[
		\begin{tikzpicture}
			\permutation[height=1.5,tkzpic=0]{4,1,3,2}
			\tie[height=1.5,snake=true,snakelen=4,style=solid]{{1.87,0.87},{3.4,0.7}}
			\tie[bend=-30,height=1.5,snake=true,snakelen=4,style=solid]{{2.2,0.6},{3,0.15}}
		\end{tikzpicture}
		\qquad		
		\begin{tikzpicture}
			\permutation[floor=1,bullb=0,height=0.75,tkzpic=0,type=1]{1,2,3,4}
			\permutation[height=0.75,tkzpic=0,type=-1]{4,1,3,2}
			\tie[floor=0.75,snake=true,snakelen=4,style=solid]{{1,0.5},{3,0.5}}
			\tie[floor=0.75,snake=true,snakelen=4,style=solid]{{2,0.25},{4,0.25}}
		\end{tikzpicture}
		\qquad		
		\begin{tikzpicture}
			\permutation[floor=2,bullb=0,height=0.5,tkzpic=0,type=1]{2,1,4,3}
			\permutation[floor=1,bullb=0,height=0.5,tkzpic=0,type=0]{1,3,2,4}
			\permutation[height=0.5,tkzpic=0,type=-1]{1,2,4,3}
		\end{tikzpicture}
		\qquad		
		\begin{tikzpicture}
			\permutation[height=1.5,tkzpic=0,type=2]{1,2,3,4}
			\tie[height=1.5,snake=true,snakelen=4,style=solid]{1,2}
			\tie[height=1.5,snake=true,snakelen=4,style=solid]{3,4}
		\end{tikzpicture}
		\begin{tikzpicture}
			\permutation[height=1.5,tkzpic=0,type=2]{1,2,3,4}
			\tie[height=1.5,snake=true,snakelen=4,style=solid]{{1,0.7},{3,0.7}}
			\tie[height=1.5,snake=true,snakelen=4,style=solid]{{2,0.3},{4,0.3}}
		\end{tikzpicture}
	\]\fi
}

\newcommand{\figureonenin}{
	\ifnum\aspdf=1\[
		\left(\,\,
		\vcdraw{\includegraphics{pics/043.pdf}}\!\!\cdots\!\!
		\vcdraw{\includegraphics{pics/047.pdf}}\cdots\!\!	
		\vcdraw{\includegraphics{pics/043.pdf}}\,\,\,,\,\,
		\vcdraw{\includegraphics{pics/043.pdf}}\!\!\cdots\!\!
		\vcdraw{\includegraphics{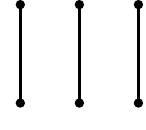}}\cdots\!\!	
		\vcdraw{\includegraphics{pics/043.pdf}}
		\,\,\,\right)\,\,=\,\,
		\vcdraw{\includegraphics{pics/043.pdf}}\!\!\cdots\!\!
		\vcdraw{\includegraphics{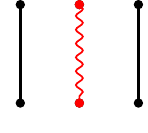}}\cdots\!\!	
		\vcdraw{\includegraphics{pics/043.pdf}}
		\]
		
	\else\[
		\vcdraw{\begin{tikzpicture}
			\vpartition[tkzpic=0,type=0]{{1,-1},{3,-3}}
			\tie[snake=true,snakelen=4,style=solid]{{2,0},{2,1}}
		\end{tikzpicture}}
		\]
	\fi
}

\newcommand{\figuretwenty}{
	\ifnum\aspdf=1\[
		\vcdraw{\includegraphics{pics/043.pdf}}\!\!\cdots\!\!
		\vcdraw{\includegraphics{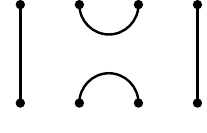}}\cdots\!\!
		\vcdraw{\includegraphics{pics/043.pdf}}\]
	\else\[
		\strands[nstr=4,type=0]{t2}\]
	\fi
}

\newcommand{\figuretwoone}{
	\ifnum\aspdf=1\[
		\vcdraw{\includegraphics{pics/043.pdf}}\!\!\cdots\!\!
		\vcdraw{\includegraphics{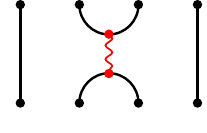}}\cdots\!\!
		\vcdraw{\includegraphics{pics/043.pdf}}\]
	\else\[
		\begin{tikzpicture}
			\strands[nstr=4,type=0,tkzpic=0]{t2}
			\tie[snake=true,snakelen=4,style=solid]{{2.5,0.3},{2.5,0.7}}
		\end{tikzpicture}\]
	\fi
}

\newcommand{\figuretwotwo}{
	\ifnum\aspdf=1\[
		\vcdraw{\includegraphics{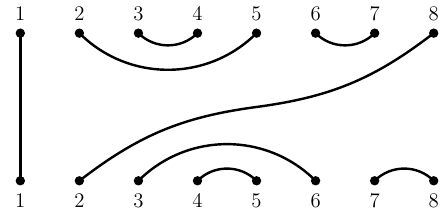}}
		\quad\to\quad
		\vcdraw{\includegraphics{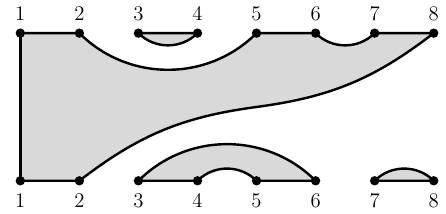}}
		\quad\to\quad
		\vcdraw{\includegraphics{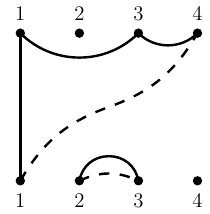}}
	\]
	\else
		\[\begin{tikzpicture}
			\draw[line width=0.7](0.6,0)to[bend left=15](2.4,0.75)to[bend right=15](4.2,1.5);
			\vpartition[height=1.5,tkzpic=0,type=2,font=0.5]{{1,-1},{2,5},{3,4},{6,7},{-3,-6},{-4,-5},{-7,-8}}
		\end{tikzpicture}
		\qquad
		\begin{tikzpicture}
			\path[draw=black,fill=gray!30,line width=0.7](0,0)--(0.6,0)to[bend left=15](2.4,0.75)to[bend right=15](4.2,1.5)--(3.6,1.5)to[bend left=45](3,1.5)--(2.4,1.5)to[bend left=45](0.6,1.5)--(0,1.5)--cycle;
			\path[draw=black,fill=gray!30,line width=0.7](1.2,1.5)to[bend right=45](1.8,1.5)--cycle;
			\path[draw=black,fill=gray!30,line width=0.7](3.6,0)to[bend left=45](4.2,0)--cycle;
			\path[draw=black,fill=gray!30,line width=0.7](1.2,0)to[bend left=45](3,0)--(2.4,0)to[bend right=45](1.8,0)--cycle;
			\vpartition[height=1.5,nstr=8,tkzpic=0,type=2,font=0.5]{{1,-1}}
		\end{tikzpicture}
		\qquad
		\begin{tikzpicture}
			\vpartition[height=1.5,tkzpic=0,type=2,font=0.5]{{-1,1,3,4}}
			\draw[line width=0.7,dashed](0.6,0)to[bend left=25](1.2,0);
			\draw[line width=0.7](0.6,0)to[bend left=40](0.9,0.25)to[bend left=40](1.2,0);
			\draw[line width=0.7,dashed](0,0)to[bend left=20](0.9,0.75)to[bend right=20](1.8,1.5);
		\end{tikzpicture}\]
	\fi
}

\newcommand{\figuretwothr}{
	\ifnum\aspdf=1\[
		\vcdraw{\includegraphics{pics/043.pdf}}\!\!\cdots\!\!
		\vcdraw{\includegraphics{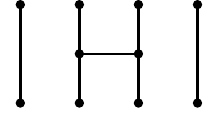}}\cdots\!\!	
		\vcdraw{\includegraphics{pics/043.pdf}}\]
	\else\[
		\vcdraw{\vpartition[type=0]{{1,-1}}}
		\kern-0.3em\cdots\kern-0.8em
		\vcdraw{
			\begin{tikzpicture}
				\vpartition[type=0,tkzpic=0]{{1,-1},{-2,2},{3,-3},{4,-4}}
				\filldraw(0.6,0.5)circle(0.04);
				\filldraw(1.2,0.5)circle(0.04);
				\draw[line width=0.7](0.6,0.5)--(1.2,0.5);
			\end{tikzpicture}
		}
		\kern-0.3em\cdots\kern-1em		
		\vcdraw{\vpartition[type=0]{{1,-1}}}\]
	\fi
}

\newcommand{\figuretwofou}{
	\ifnum\aspdf=1\[
		\vcdraw{\includegraphics{pics/043.pdf}}\!\!\cdots\!\!
		\vcdraw{\includegraphics{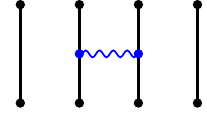}}\cdots\!\!
		\vcdraw{\includegraphics{pics/043.pdf}}\]
	\else\[
		\vcdraw{\vpartition[type=0]{{1,-1}}}
		\kern-0.5em\cdots\kern-1em
		\vcdraw{\begin{tikzpicture}
			\permutation[tkzpic=0,type=0]{1,2,3,4}
			\tie[snake=true,snakelen=4,style=solid,color=blue]{{2,0.5},{3,0.5}}
		\end{tikzpicture}}
		\kern-0.3em\cdots\kern-1em		
		\vcdraw{\vpartition[type=0]{{1,-1},{2,-2}}}\]
	\fi
}

\newcommand{\figuretwofiv}{
	\ifnum\aspdf=1\[
		\vcdraw{\includegraphics{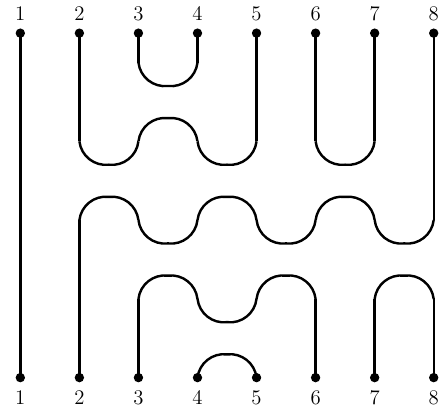}}\quad\to\quad
		\vcdraw{\includegraphics{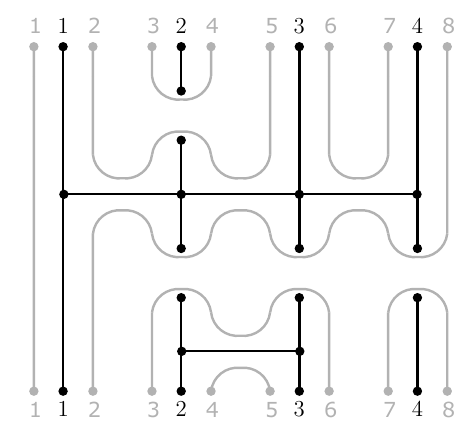}}\quad\to\quad
		\vcdraw{\includegraphics{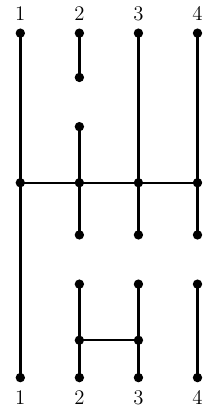}}
		\]
	\else\[
		\begin{tikzpicture}
			\strands[tkzpic=0,height=0.3,font=0.5,type=1,nstr=8,bullb=0,floor=3.2]{i1}
			\strands[tkzpic=0,height=0.8,font=0.5,type=-1,bulla=0]{t3*t2-t4-t6*t3-t5-t7*t4}
		\end{tikzpicture}
		\begin{tikzpicture}
    			\node at(0,0){\includegraphics{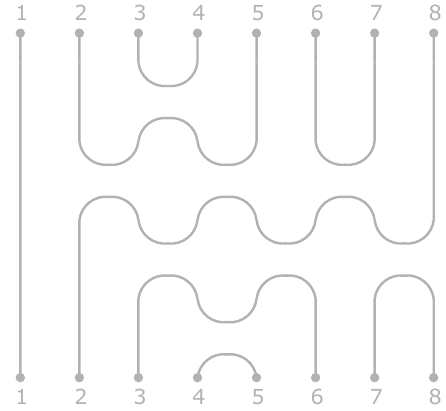}};
    			\node[scale=0.55]at(-1.76,1.963){$1$};
    			\node[scale=0.55]at(-0.56,1.963){$2$};
    			\node[scale=0.55]at(0.64,1.963){$3$};
    			\node[scale=0.55]at(1.84,1.963){$4$};
    			\node[scale=0.55]at(-1.76,-1.935){$1$};
    			\node[scale=0.55]at(-0.56,-1.935){$2$};
    			\node[scale=0.55]at(0.64,-1.935){$3$};
    			\node[scale=0.55]at(1.84,-1.935){$4$};
    			\draw[line width=0.7](-1.76,-1.75)--(-1.76,1.75);
			\filldraw(-1.76,-1.75)circle(0.04);
			\filldraw(-1.76,1.75)circle(0.04);
			\draw[line width=0.7](-0.56,1.75)--(-0.56,1.30);
			\draw[line width=0.7](-0.56,0.8)--(-0.56,-0.3);
			\draw[line width=0.7](-0.56,-0.8)--(-0.56,-1.75);
			\filldraw(-0.56,1.75)circle(0.04);
			\filldraw(-0.56,1.30)circle(0.04);
			\filldraw(-0.56,0.8)circle(0.04);
			\filldraw(-0.56,-0.3)circle(0.04);
			\filldraw(-0.56,-0.8)circle(0.04);
			\filldraw(-0.56,-1.75)circle(0.04);
			\draw[line width=0.7](0.64,1.75)--(0.64,-0.3);
			\draw[line width=0.7](0.64,-0.8)--(0.64,-1.75);
			\filldraw(0.64,1.75)circle(0.04);
			\filldraw(0.64,-0.3)circle(0.04);
			\filldraw(0.64,-0.8)circle(0.04);
			\filldraw(0.64,-1.75)circle(0.04);
			\draw[line width=0.7](1.84,1.75)--(1.84,-0.3);
			\draw[line width=0.7](1.84,-0.8)--(1.84,-1.75);
			\filldraw(1.84,1.75)circle(0.04);
			\filldraw(1.84,-0.3)circle(0.04);
			\filldraw(1.84,-0.8)circle(0.04);
			\filldraw(1.84,-1.75)circle(0.04);
			\tie[
				style=solid,
				color=black
			]{{-1.92,0.25},{4.06,0.25}}
			\tie[
				style=solid,
				color=black
			]{{0.075,-1.345},{2.075,-1.345}}
			\filldraw[black](-0.56,0.25)circle(0.04);
			\filldraw[black](0.64,0.25)circle(0.04);
		\end{tikzpicture}
		\begin{tikzpicture}
			\vpartition[font=0.5,height=3.5,nstr=4,type=2,tkzpic=0]{{1,-1}}
			\draw[line width=0.7](0.6,3.5)--(0.6,3.05);
			\draw[line width=0.7](0.6,2.55)--(0.6,1.45);
			\draw[line width=0.7](0.6,0.95)--(0.6,0);
			\filldraw(0.6,3.05)circle(0.04);
			\filldraw(0.6,2.55)circle(0.04);
			\filldraw[black](0.6,1.98)circle(0.04);
			\filldraw(0.6,1.45)circle(0.04);
			\filldraw(0.6,0.95)circle(0.04);
			\draw[line width=0.7](1.2,3.5)--(1.2,1.45);
			\draw[line width=0.7](1.2,0.95)--(1.2,0);
			\filldraw[black](1.2,1.98)circle(0.04);
			\filldraw(1.2,1.45)circle(0.04);
			\filldraw(1.2,0.95)circle(0.04);
			\draw[line width=0.7](1.8,3.5)--(1.8,1.45);
			\draw[line width=0.7](1.8,0.95)--(1.8,0);
			\filldraw(1.8,1.45)circle(0.04);
			\filldraw(1.8,0.95)circle(0.04);
			\tie[
				style=solid,
				color=black
			]{{1,1.98},{4,1.98}}
			\tie[
				style=solid,
				color=black
			]{{2,0.38},{3,0.38}}
		\end{tikzpicture}
	\]\fi
}

\newcommand{\figuretwosix}{
	\ifnum\aspdf=1\[
		\includegraphics{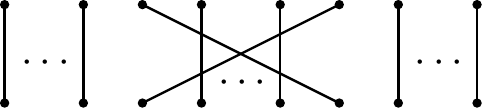}
	\]\else\[
		\begin{tikzpicture}
			\filldraw(-2.4,1)circle(0.04);
			\filldraw(-1.6,1)circle(0.04);
			\filldraw(-1.0,1)circle(0.04);
			\filldraw(-0.4,1)circle(0.04);
			\filldraw(+0.4,1)circle(0.04);
			\filldraw(+1.0,1)circle(0.04);
			\filldraw(+1.6,1)circle(0.04);
			\filldraw(+2.4,1)circle(0.04);
			\filldraw(-2.4,0)circle(0.04);
			\filldraw(-1.6,0)circle(0.04);
			\filldraw(-1.0,0)circle(0.04);
			\filldraw(-0.4,0)circle(0.04);
			\filldraw(+0.4,0)circle(0.04);
			\filldraw(+1.0,0)circle(0.04);
			\filldraw(+1.6,0)circle(0.04);
			\filldraw(+2.4,0)circle(0.04);
			\draw[line width=0.7](-2.4,1)--(-2.4,0);
			\draw[line width=0.7](-1.6,1)--(-1.6,0);
			\draw[line width=0.7](-1.0,1)--(+1.0,0);
			\draw[line width=0.7](-0.4,1)--(-0.4,0);
			\draw[line width=0.7](+0.4,1)--(+0.4,0);
			\draw[line width=0.7](+1.0,1)--(-1.0,0);
			\draw[line width=0.7](+1.6,1)--(+1.6,0);
			\draw[line width=0.7](+2.4,1)--(+2.4,0);
			\node at(-1.95,0.4){$\cdots$};
			\node at(+2.05,0.4){$\cdots$};
			\node at(+0.05,0.2){$\cdots$};
		\end{tikzpicture}
	\]\fi
}

\newcommand{\figuretwosev}{
	\ifnum\aspdf=1\[
		\includegraphics{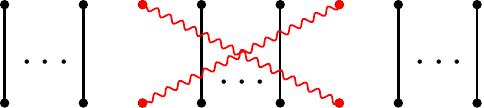}
	\]\else\[
		\begin{tikzpicture}
			\filldraw(-2.4,1)circle(0.04);
			\filldraw(-1.6,1)circle(0.04);
			\filldraw(-1.0,1)circle(0.04);
			\filldraw(-0.4,1)circle(0.04);
			\filldraw(+0.4,1)circle(0.04);
			\filldraw(+1.0,1)circle(0.04);
			\filldraw(+1.6,1)circle(0.04);
			\filldraw(+2.4,1)circle(0.04);
			\filldraw(-2.4,0)circle(0.04);
			\filldraw(-1.6,0)circle(0.04);
			\filldraw(-1.0,0)circle(0.04);
			\filldraw(-0.4,0)circle(0.04);
			\filldraw(+0.4,0)circle(0.04);
			\filldraw(+1.0,0)circle(0.04);
			\filldraw(+1.6,0)circle(0.04);
			\filldraw(+2.4,0)circle(0.04);
			\draw[line width=0.7](-2.4,1)--(-2.4,0);
			\draw[line width=0.7](-1.6,1)--(-1.6,0);
			\draw[line width=0.7](-0.4,1)--(-0.4,0);
			\draw[line width=0.7](+0.4,1)--(+0.4,0);
			\draw[line width=0.7](+1.6,1)--(+1.6,0);
			\draw[line width=0.7](+2.4,1)--(+2.4,0);
			\node at(-1.95,0.4){$\cdots$};
			\node at(+2.05,0.4){$\cdots$};
			\node at(+0.05,0.2){$\cdots$};
			\tie[
				snake=true,
				snakelen=4,
				style=solid,
				color=red
			]{{-0.66,1},{+2.67,0}}
			\tie[
				snake=true,
				snakelen=4,
				style=solid,
				color=red
			]{{-0.66,0},{+2.67,1}}
		\end{tikzpicture}
	\]\fi
}

\newcommand{\figuretwoeig}{
	\ifnum\aspdf=1\[
		\vcdraw{\includegraphics{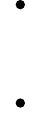}}\,\cdots
		\vcdraw{\includegraphics{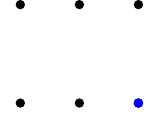}}
		\kern+1.5em\leftrightarrow\kern+1em
		\vcdraw{\includegraphics{pics/070.pdf}}\,\cdots
		\vcdraw{\includegraphics{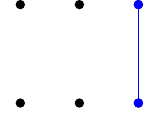}}
	\]\else\[
		\begin{tikzpicture}
			\vpartition[tkzpic=0,type=0]{{1},{-1}}
		\end{tikzpicture}
		\begin{tikzpicture}
			\vpartition[tkzpic=0,type=0]{{1},{-1},{2},{-2},{3},{-3}}
			\tie[color=blue,style=solid]{{3,0},{3,0}}
		\end{tikzpicture}
	\]\fi
}

\newcommand{\figuretwonin}{
	\ifnum\aspdf=1\[
		\includegraphics{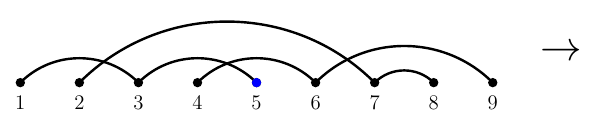}\,\,
		\includegraphics{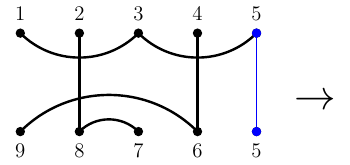}
		\includegraphics{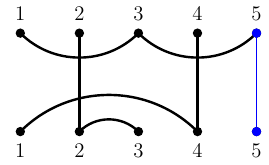}
	\]\else\[
		\begin{tikzpicture}
			\arcpartition[tkzpic=0]{{1,3,5},{2,7,8},{4,6,9}}
			\filldraw[blue](2.4,0)circle(0.04);
			\node at(5.5,0.3){$\to$};
		\end{tikzpicture}
		\begin{tikzpicture}
			\vpartition[
				customlabelsabove={1,2,3,4,5},
				customlabelsbelow={9,8,7,6,5},
				tkzpic=0,
				type=6
			]{{1,3,5},{2,-2,-3},{4,-4,-1}}
			\tie[color=blue,style=solid]{{5,0},{5,1}}
			\node at(3,0.3){$\to$};
		\end{tikzpicture}
		\begin{tikzpicture}
			\vpartition[
				tkzpic=0,
				type=2
			]{{1,3,5},{2,-2,-3},{4,-4,-1}}
			\tie[color=blue,style=solid]{{5,0},{5,1}}
		\end{tikzpicture}
	\]\fi
}

\newcommand{\figurethirty}{
	\ifnum\aspdf=1\[
		I\,\,=\vcdraw{\includegraphics{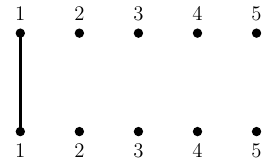}}\,\,,\qquad
		J\,\,=\vcdraw{\includegraphics{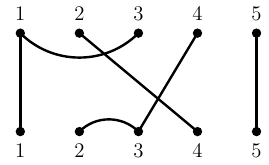}}\,\,,\qquad
		(I,J)\,\,=\vcdraw{\includegraphics{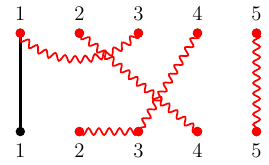}}
	\]\else\[
		\vpartition[nstr=5,type=2]{{1,-1}}
		\vpartition[type=2]{{3,1,-1},{2,-4},{-2,-3,4},{5,-5}}
		\begin{tikzpicture}
			\vpartition[nstr=5,type=2,tkzpic=0]{{1,-1}}
			\tie[snake=true,style=solid,bend=-45]{{1,1},{3,1}}
			\tie[snake=true,style=solid]{{2,1},{4,0}}
			\tie[snake=true,style=solid]{{2,0},{3,0},{4,1}}
			\tie[snake=true,style=solid]{{5,0},{5,1}}
		\end{tikzpicture}
	\]\fi
}

\newcommand{\figurethione}{
	\ifnum\aspdf=1\[\begin{array}{rrr}
		e\,\,=\vcdraw{\includegraphics{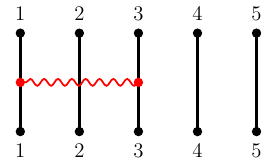}}\qquad\,&
		e'\,\,=\vcdraw{\includegraphics{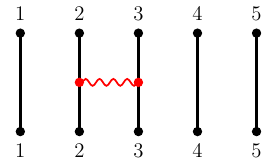}}\qquad\,&
		r\,\,=\vcdraw{\includegraphics{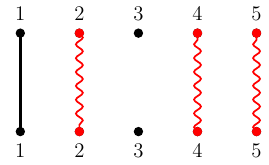}}\\[1cm]	
		g^*\,\,=\vcdraw{\includegraphics{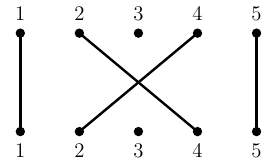}}\qquad\,&
		g^*_p\,\,=\vcdraw{\includegraphics{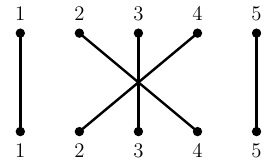}}\qquad\,&
		x_g\,\,=\vcdraw{\includegraphics{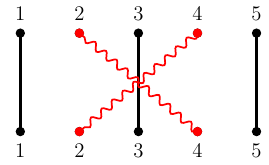}}
	\end{array}\]\else\[
		\vpartition[type=2]{{1,-3},{-1,3},{4,-4}}
		\vpartition[type=2]{{1,-1},{2,-4},{-2,4},{5,-5}}
		\begin{tikzpicture}
			\vpartition[tkzpic=0,type=2]{{1,-1},{2,-2},{3,-3},{4,-4},{5,-5}}
			\tie[
				snake=true,
				snakelen=4,
				style=solid
			]{1,3}
		\end{tikzpicture}
		\begin{tikzpicture}
			\vpartition[tkzpic=0,type=2]{{1,-1},{2,-2},{3,-3},{4,-4},{5,-5}}
			\tie[
				snake=true,
				snakelen=4,
				style=solid
			]{2,3}
		\end{tikzpicture}
		\begin{tikzpicture}
			\vpartition[tkzpic=0,type=2]{{1,-1},{-3,3},{5,-5}}
			\tie[
				snake=true,
				snakelen=4,
				style=solid
			]{{2,0},{4,1}}
			\tie[
				snake=true,
				snakelen=4,
				style=solid
			]{{2,1},{4,0}}
		\end{tikzpicture}
		\begin{tikzpicture}
			\vpartition[tkzpic=0,type=2]{{1,-1},{5,-5}}
			\tie[
				snake=true,
				snakelen=4,
				style=solid
			]{{2,0},{4,1}}
			\tie[
				snake=true,
				snakelen=4,
				style=solid
			]{{2,1},{4,0}}
			\tie[
				snake=true,
				snakelen=4,
				style=solid
			]{{3,1},{3,0}}
		\end{tikzpicture}
		\begin{tikzpicture}
			\vpartition[,nstr=5,tkzpic=0,type=2]{{1,-1}}
			\tie[
				snake=true,
				snakelen=4,
				style=solid
			]{{2,0},{2,1}}
			\tie[
				snake=true,
				snakelen=4,
				style=solid
			]{{4,1},{4,0}}
			\tie[
				snake=true,
				snakelen=4,
				style=solid
			]{{5,1},{5,0}}
		\end{tikzpicture}
	\]\fi
}

\newcommand{\figurethitwo}{
	\ifnum\aspdf=1\[	g\,\,\,\,=\,\,
		\vcdraw{\includegraphics{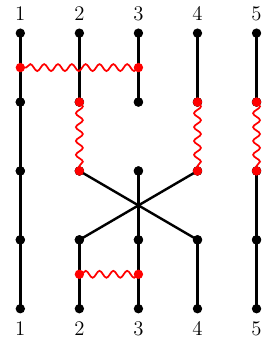}}\,\,\,\,\,=\,\,\,
		\vcdraw{\includegraphics{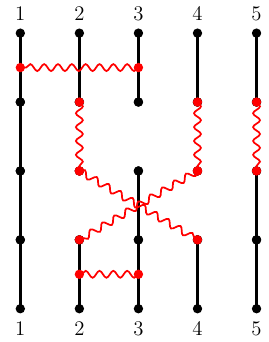}}\,\,\,\,\,=\,\,\,
		\vcdraw{\includegraphics{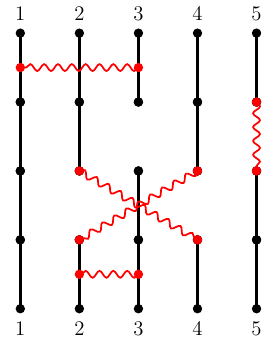}}%\,\,\,\,\,=\,\,\,
	\]\else\[
		\begin{tikzpicture}
			\vpartition[floor=3,tkzpic=0,type=1,height=0.7]{{-1,1},{-2,2},{-3,3},{-4,4},{-5,5}}
			\vpartition[floor=2,tkzpic=0,type=0,height=0.7,nstr=5]{{-1,1}}
			\vpartition[floor=1,tkzpic=0,type=0,height=0.7]{{-1,1},{-3,3},{-5,5},{2,-4},{4,-2}}
			\vpartition[tkzpic=0,type=-1,height=0.7]{{-1,1},{-2,2},{-3,3},{-4,4},{-5,5}}
			\tie[
				floor=0,
				height=0.7,
				snake=true,
				snakelen=4,
				style=solid
			]{2,3}
			\tie[
				floor=2,
				height=0.7,
				snake=true,
				snakelen=4,
				style=solid
			]{{2,0},{2,1}}
			\tie[
				floor=2,
				height=0.7,
				snake=true,
				snakelen=4,
				style=solid
			]{{4,0},{4,1}}
			\tie[
				floor=2,
				height=0.7,
				snake=true,
				snakelen=4,
				style=solid
			]{{5,0},{5,1}}
			\tie[
				floor=3,
				height=0.7,
				snake=true,
				snakelen=4,
				style=solid
			]{1,3}
		\end{tikzpicture}
	\]\fi
}

\newcommand{\figurethithr}{
	\ifnum\aspdf=1\[
		\vcdraw{\includegraphics{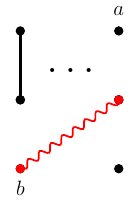}}\,\,\,=\,\,
		\vcdraw{\includegraphics{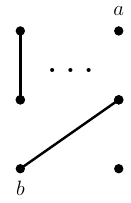}}\kern+3.5em
		\vcdraw{\includegraphics{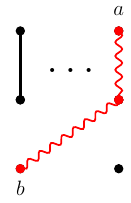}}\,\,\,=\,\,
		\vcdraw{\includegraphics{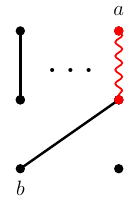}}\,\,\,=\,\,
		\vcdraw{\includegraphics{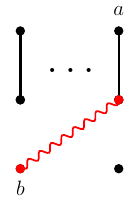}}
	\]\else\[
		\begin{tikzpicture}
			\vpartition[
				customlabelsabove={},
				customlabelsbelow={b,},
				height=0.7,
				nstr=2,
				tkzpic=0,
				type=6,
				width=1
			]{{1}}
			\vpartition[
				customlabelsabove={,a},
				customlabelsbelow={},
				floor=1,
				height=0.7,
				nstr=2,
				tkzpic=0,
				type=6,
				width=1
			]{{1,-1}}
			\tie[
				height=0.7,
				snake=true,
				snakelen=4,
				style=solid,
				width=1
			]{{1,0},{2,1}}
			\node at(0.51,1){$\ldots$};
		\end{tikzpicture}
		\begin{tikzpicture}
			\vpartition[
				customlabelsabove={},
				customlabelsbelow={b,},
				height=0.7,
				nstr=2,
				tkzpic=0,
				type=6,
				width=1
			]{{-1,2}}
			\vpartition[
				customlabelsabove={,a},
				customlabelsbelow={},
				floor=1,
				height=0.7,
				nstr=2,
				tkzpic=0,
				type=6,
				width=1
			]{{1,-1}}
			\node at(0.51,1){$\ldots$};
		\end{tikzpicture}
		\begin{tikzpicture}
			\vpartition[
				customlabelsabove={},
				customlabelsbelow={b,},
				height=0.7,
				nstr=2,
				tkzpic=0,
				type=6,
				width=1
			]{{1}}
			\vpartition[
				customlabelsabove={,a},
				customlabelsbelow={},
				floor=1,
				height=0.7,
				nstr=2,
				tkzpic=0,
				type=6,
				width=1
			]{{1,-1}}
			\tie[
				height=0.7,
				snake=true,
				snakelen=4,
				style=solid,
				width=1
			]{{1,0},{2,1}}
			\tie[
				floor=1,
				height=0.7,
				snake=true,
				snakelen=4,
				style=solid,
				width=1
			]{{2,0},{2,1}}
			\node at(0.51,1){$\ldots$};
		\end{tikzpicture}
		\begin{tikzpicture}
			\vpartition[
				customlabelsabove={},
				customlabelsbelow={b,},
				height=0.7,
				nstr=2,
				tkzpic=0,
				type=6,
				width=1
			]{{-1,2}}
			\vpartition[
				customlabelsabove={,a},
				customlabelsbelow={},
				floor=1,
				height=0.7,
				nstr=2,
				tkzpic=0,
				type=6,
				width=1
			]{{1,-1}}
			\tie[
				floor=1,
				height=0.7,
				snake=true,
				snakelen=4,
				style=solid,
				width=1
			]{{2,0},{2,1}}
			\node at(0.51,1){$\ldots$};
		\end{tikzpicture}
		\begin{tikzpicture}
			\vpartition[
				customlabelsabove={},
				customlabelsbelow={b,},
				height=0.7,
				nstr=2,
				tkzpic=0,
				type=6,
				width=1
			]{{1}}
			\vpartition[
				customlabelsabove={,a},
				customlabelsbelow={},
				floor=1,
				height=0.7,
				nstr=2,
				tkzpic=0,
				type=6,
				width=1
			]{{1,-1},{2,-2}}
			\tie[
				height=0.7,
				snake=true,
				snakelen=4,
				style=solid,
				width=1
			]{{1,0},{2,1}}
			\node at(0.51,1){$\ldots$};
		\end{tikzpicture}
	\]\fi
}

\newcommand{\figurethifou}{
	\ifnum\aspdf=1\[
		\vcdraw{\includegraphics{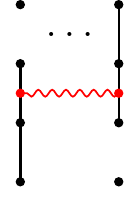}}
		\,\,\,\,\,=\,\,\,
		\vcdraw{\includegraphics{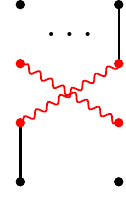}}
	\]\else\[
		\begin{tikzpicture}
			\strands[
				height=0.6,
				tkzpic=0,
				type=0,
				width=1]{r0*i1*r1}
			\tie[
				height=1.8,
				snake=true,
				snakelen=4,
				style=solid,
				width=1,
			]{1,2}
			\filldraw(0,0.6)circle(0.04);
			\filldraw(1,0.6)circle(0.04);
			\filldraw(0,1.2)circle(0.04);
			\filldraw(1,1.2)circle(0.04);
			\node at(0.5,1.5){$\ldots$};
		\end{tikzpicture}
		\begin{tikzpicture}
			\strands[
				height=0.6,
				tkzpic=0,
				type=0,
				width=1]{r0*r0-r1*r1}
			\tie[
				floor=1,
				height=0.6,
				snake=true,
				snakelen=4,
				style=solid,
				width=1,
			]{{1,0},{2,1}}
			\tie[
				floor=1,
				height=0.6,
				snake=true,
				snakelen=4,
				style=solid,
				width=1,
			]{{1,1},{2,0}}
			\node at(0.5,1.5){$\ldots$};
		\end{tikzpicture}
	\]\fi
}

\newcommand{\figurethifiv}[1]{
	\ifnum\aspdf=1\[
		g\,=\vcdraw{\includegraphics{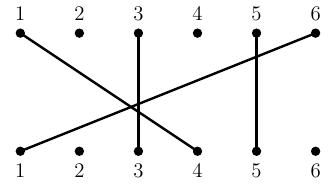}}\qquad
		g_p\,=\,s\,=\,\vcdraw{\includegraphics{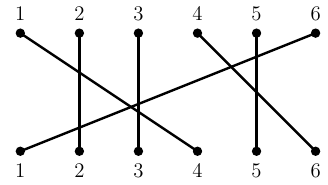}}\qquad
		t\,=\vcdraw{\includegraphics{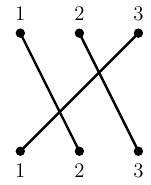}}\,=\,
		\vcdraw{\includegraphics{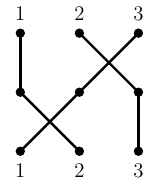}}
	\]#1\[g\,\,=\,\vcdraw{\includegraphics{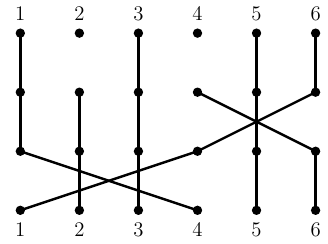}}\]\else\[
		\vpartition[height=1.2,type=2]{{1,-4},{3,-3},{5,-5},{6,-1}}
		\permutation[height=1.2,type=2]{4,2,3,6,5,1}
		\permutation[height=1.2,type=2]{2,3,1}
		\begin{tikzpicture}
			\permutation[bullb=0,floor=1,height=0.6,type=1,tkzpic=0]{1,3,2}
			\permutation[height=0.6,type=-1,tkzpic=0]{2,1,3}
		\end{tikzpicture}
		\begin{tikzpicture}
			\vpartition[floor=2,height=0.6,nstr=6,type=1,tkzpic=0]
			{{1,-1},{3,-3},{5,-5},{6,-6}}
			\permutation[bulla=0,bullb=0,floor=1,height=0.6,type=0,tkzpic=0]{1,2,3,6,5,4}
			\permutation[height=0.6,type=-1,tkzpic=0]{4,2,3,1,5,6}
		\end{tikzpicture}
	\]\fi
}

\newcommand{\figurethisix}{
	\includegraphics{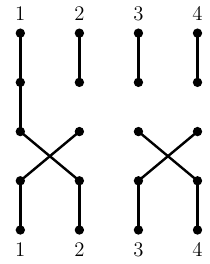}\qquad
	\includegraphics{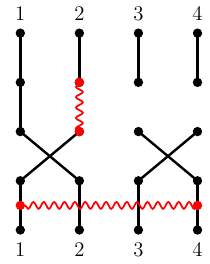}\qquad
	\includegraphics{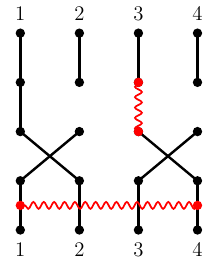}\qquad
	\includegraphics{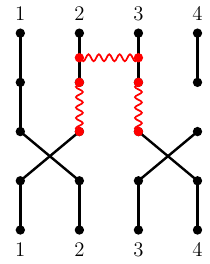}\qquad
	\includegraphics{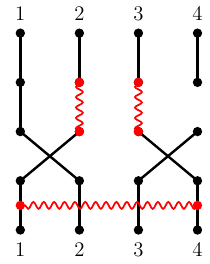}
}

\newcommand{\figurethisev}{
	\ifnum\aspdf=1
		\[\begin{array}{c}
			\includegraphics{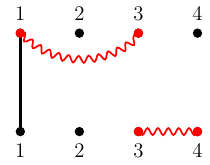}\\[0.3cm]
			\includegraphics{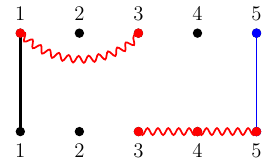}\quad
			\includegraphics{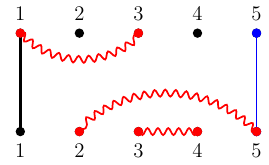}\quad
			\includegraphics{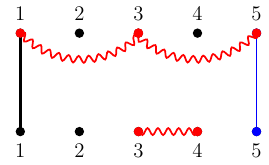}\quad
			\includegraphics{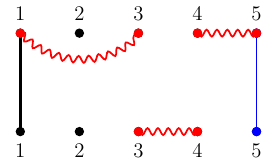}\quad
			\includegraphics{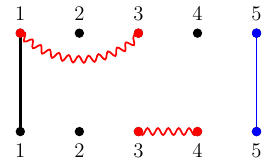}
		\end{array}\]
	\else
		\begin{tikzpicture}%[scale=0.9]
			\vpartition[font=0.5,nstr=4,type=2,tkzpic=0]{{1,-1}}
			\tie[snake=true,snakelen=3,style=solid,bend=45]{{3,1},{1,1}}
			\tie[snake=true,snakelen=3,style=solid]{{3,0},{4,0}}
		\end{tikzpicture}
		\begin{tikzpicture}%[scale=0.9]
			\vpartition[font=0.5,nstr=5,type=2,tkzpic=0]{{1,-1}}
			\tie[color=blue,style=solid]{{5,0},{5,1}}
			\tie[snake=true,snakelen=3,style=solid,bend=45]{{3,1},{1,1}}
			\tie[snake=true,snakelen=3,style=solid]{{3,0},{4,0},{5,0}}
		\end{tikzpicture}
		\begin{tikzpicture}%[scale=0.9]
			\vpartition[font=0.5,nstr=5,type=2,tkzpic=0]{{1,-1}}
			\tie[color=blue,style=solid]{{5,0},{5,1}}
			\tie[snake=true,snakelen=3,style=solid,bend=45]{{3,1},{1,1}}
			\tie[snake=true,snakelen=3,style=solid]{{3,0},{4,0}}
			\tie[snake=true,snakelen=3,style=solid,bend=45]{{2,0},{5,0}}
		\end{tikzpicture}
		\begin{tikzpicture}%[scale=0.9]
			\vpartition[font=0.5,nstr=5,type=2,tkzpic=0]{{1,-1}}
			\tie[color=blue,style=solid]{{5,0},{5,1}}
			\tie[snake=true,snakelen=3,style=solid,bend=45]{{3,1},{1,1}}
			\tie[snake=true,snakelen=3,style=solid]{{3,0},{4,0}}
			\tie[snake=true,snakelen=3,style=solid,bend=-45]{{3,1},{5,1}}
		\end{tikzpicture}
		\begin{tikzpicture}%[scale=0.9]
			\vpartition[font=0.5,nstr=5,type=2,tkzpic=0]{{1,-1}}
			\tie[color=blue,style=solid]{{5,0},{5,1}}
			\tie[snake=true,snakelen=3,style=solid,bend=45]{{3,1},{1,1}}
			\tie[snake=true,snakelen=3,style=solid]{{3,0},{4,0}}
			\tie[snake=true,snakelen=3,style=solid]{{4,1},{5,1}}
		\end{tikzpicture}
		\begin{tikzpicture}%[scale=0.9]
			\vpartition[font=0.5,nstr=5,type=2,tkzpic=0]{{1,-1}}
			\tie[color=blue,style=solid]{{5,0},{5,1}}
			\tie[snake=true,snakelen=3,style=solid,bend=45]{{3,1},{1,1}}
			\tie[snake=true,snakelen=3,style=solid]{{3,0},{4,0}}
		\end{tikzpicture}
	\fi
}

\newcommand{\figurethieig}{
	\ifnum\aspdf=1\[
		\vcdraw{\includegraphics{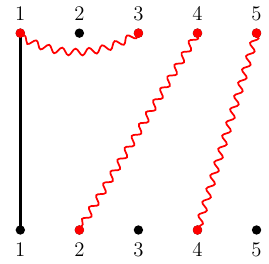}}\quad=\quad
		\vcdraw{\includegraphics{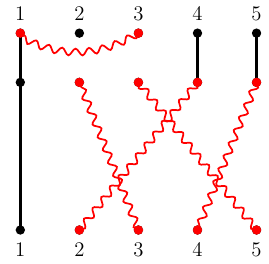}}\quad=\quad
		\vcdraw{\includegraphics{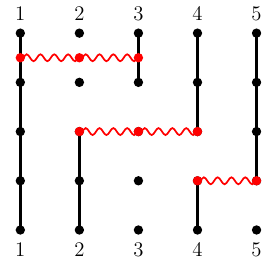}}
	\]\else
		\begin{tikzpicture}
			\vpartition[height=2,nstr=5,tkzpic=0,type=2]{{1,-1}}
			\tie[height=2,snake=true,snakelen=4,style=solid,bend=-30]{{1,1},{3,1}}
			\tie[height=2,snake=true,snakelen=4,style=solid]{{2,0},{4,1}}
			\tie[height=2,snake=true,snakelen=4,style=solid]{{4,0},{5,1}}
		\end{tikzpicture}
		\begin{tikzpicture}
			\vpartition[floor=3,height=0.5,nstr=5,tkzpic=0,type=1]{{1,-1},{4,-4},{5,-5}}
			\vpartition[bulla=0,height=1.5,nstr=5,tkzpic=0,type=-1]{{1,-1}}
			\tie[floor=3,height=0.5,snake=true,snakelen=4,style=solid,bend=-30]{{1,1},{3,1}}
			\tie[height=1.5,snake=true,snakelen=4,style=solid]{{2,1},{3,0}}
			\tie[height=1.5,snake=true,snakelen=4,style=solid]{{2,0},{4,1}}
			\tie[height=1.5,snake=true,snakelen=4,style=solid]{{3,1},{5,0}}
			\tie[height=1.5,snake=true,snakelen=4,style=solid]{{4,0},{5,1}}
		\end{tikzpicture}
		\begin{tikzpicture}
			\vpartition[floor=3,height=0.5,nstr=5,tkzpic=0,type=1]{{1,-1},{3,-3},{4,-4},{5,-5}}
			\vpartition[floor=2,bulla=0,height=0.5,nstr=5,tkzpic=0,type=0]{{1,-1},{4,-4},{5,-5}}
			\vpartition[floor=1,bulla=0,height=0.5,nstr=5,tkzpic=0,type=0]{{1,-1},{2,-2},{5,-5}}
			\vpartition[bulla=0,height=0.5,nstr=5,tkzpic=0,type=-1]{{1,-1},{2,-2},{4,-4}}
			\tie[floor=2,height=0.5,snake=true,snakelen=4,style=solid]{{1,1.5},{2,1.5},{3,1.5}}
			\tie[height=0.5,snake=true,snakelen=4,style=solid]{{2,2},{3,2},{4,2}}
			\tie[height=0.5,snake=true,snakelen=4,style=solid]{{4,1},{5,1}}
		\end{tikzpicture}
	\fi
}

\newcommand{\figurethinin}{
	\ifnum\aspdf=1\[
		\vcdraw{\includegraphics{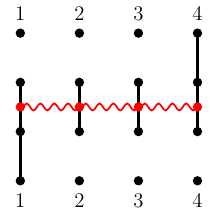}}\qquad
		\vcdraw{\includegraphics{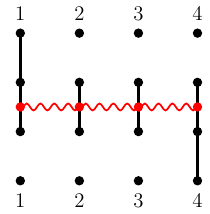}}\qquad
		\vcdraw{\includegraphics{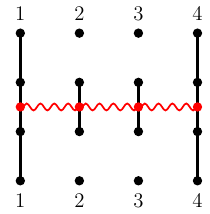}}
	\]\else
		\begin{tikzpicture}
			\vpartition[floor=2,height=0.5,nstr=4,tkzpic=0,type=1]{{4,-4}}
			\vpartition[bulla=0,bullb=0,floor=1,height=0.5,nstr=4,tkzpic=0,type=0]{{1,-1},{2,-2},{3,-3},{4,-4}}
			\vpartition[height=0.5,nstr=4,tkzpic=0,type=-1]{{1,-1}}
			\tie[floor=1,height=0.5,snake=true,snakelen=4,style=solid]{1,2,3,4}
		\end{tikzpicture}
		\begin{tikzpicture}
			\vpartition[floor=2,height=0.5,nstr=4,tkzpic=0,type=1]{{1,-1}}
			\vpartition[bulla=0,bullb=0,floor=1,height=0.5,nstr=4,tkzpic=0,type=0]{{1,-1},{2,-2},{3,-3},{4,-4}}
			\vpartition[height=0.5,nstr=4,tkzpic=0,type=-1]{{4,-4}}
			\tie[floor=1,height=0.5,snake=true,snakelen=4,style=solid]{1,2,3,4}
		\end{tikzpicture}
		\begin{tikzpicture}
			\vpartition[floor=2,height=0.5,nstr=4,tkzpic=0,type=1]{{1,-1},{4,-4}}
			\vpartition[bulla=0,bullb=0,floor=1,height=0.5,nstr=4,tkzpic=0,type=0]{{1,-1},{2,-2},{3,-3},{4,-4}}
			\vpartition[height=0.5,nstr=4,tkzpic=0,type=-1]{{1,-1},{4,-4}}
			\tie[floor=1,height=0.5,snake=true,snakelen=4,style=solid]{1,2,3,4}
		\end{tikzpicture}
	\fi
}

\newcommand{\figureforty}{
	\ifnum\aspdf=1\[\begin{array}{c}
		\vcdraw{\includegraphics{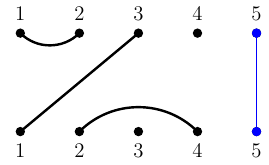}}\qquad
		\vcdraw{\includegraphics{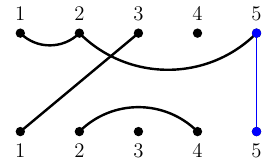}}\qquad
		\vcdraw{\includegraphics{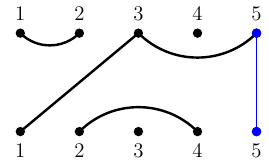}}\\[0.9cm]
		\vcdraw{\includegraphics{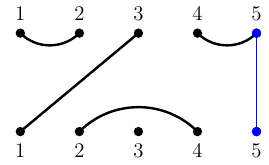}}\qquad
		\vcdraw{\includegraphics{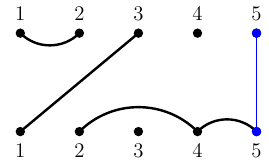}}\qquad
		\vcdraw{\includegraphics{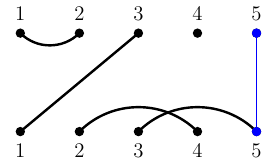}}
	\end{array}\]\else
		\vpartition[height=1,type=2]{{1,2},{3,-1},{-2,-4}}
		\begin{tikzpicture}
			\vpartition[height=1,tkzpic=2,type=2,nstr=5]{{1,2},{3,-1},{-2,-4}}
			\tie[color=blue,style=solid]{{5,0},{5,1}}
		\end{tikzpicture}
		\begin{tikzpicture}
			\vpartition[height=1,tkzpic=2,type=2,nstr=5]{{1,2,5},{3,-1},{-2,-4}}
			\tie[color=blue,style=solid]{{5,0},{5,1}}
		\end{tikzpicture}
		\begin{tikzpicture}
			\vpartition[height=1,tkzpic=2,type=2,nstr=5]{{1,2},{5,3,-1},{-2,-4}}
			\tie[color=blue,style=solid]{{5,0},{5,1}}
		\end{tikzpicture}
		\begin{tikzpicture}
			\vpartition[height=1,tkzpic=2,type=2,nstr=5]{{1,2},{3,-1},{4,5},{-2,-4}}
			\tie[color=blue,style=solid]{{5,0},{5,1}}
		\end{tikzpicture}
		\begin{tikzpicture}
			\vpartition[height=1,tkzpic=2,type=2,nstr=5]{{1,2},{3,-1},{-2,-4,-5}}
			\tie[color=blue,style=solid]{{5,0},{5,1}}
		\end{tikzpicture}
		\begin{tikzpicture}
			\vpartition[height=1,tkzpic=2,type=2,nstr=5]{{1,2},{3,-1},{-2,-4},{-3,-5}}
			\tie[color=blue,style=solid]{{5,0},{5,1}}
		\end{tikzpicture}
	\fi
}

\newcommand{\figureforone}{
	\ifnum\aspdf=1
		\[\begin{array}{c}
			\includegraphics{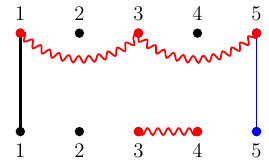}\\[0.3cm]
			\includegraphics{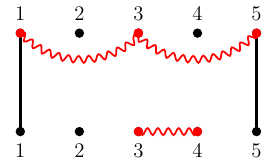}\quad
			\includegraphics{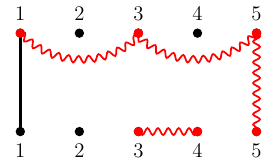}\quad
			\includegraphics{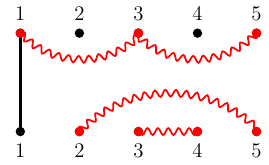}\quad
			\includegraphics{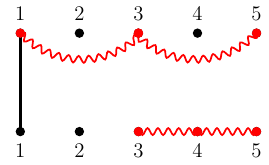}\quad
			\includegraphics{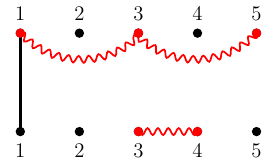}
		\end{array}\]
	\else
		\begin{tikzpicture}
			\vpartition[font=0.5,nstr=5,type=2,tkzpic=0]{{1,-1}}
			\tie[color=blue,style=solid]{{5,0},{5,1}}
			\tie[snake=true,snakelen=3,style=solid,bend=45]{{3,1},{1,1}}
			\tie[snake=true,snakelen=3,style=solid]{{3,0},{4,0}}
			\tie[snake=true,snakelen=3,style=solid,bend=-45]{{3,1},{5,1}}
		\end{tikzpicture}
		\begin{tikzpicture}
			\vpartition[font=0.5,nstr=5,type=2,tkzpic=0]{{1,-1},{5,-5}}
			\tie[snake=true,snakelen=3,style=solid,bend=45]{{3,1},{1,1}}
			\tie[snake=true,snakelen=3,style=solid]{{3,0},{4,0}}
			\tie[snake=true,snakelen=3,style=solid,bend=-45]{{3,1},{5,1}}
		\end{tikzpicture}
		\begin{tikzpicture}
			\vpartition[font=0.5,nstr=5,type=2,tkzpic=0]{{1,-1}}
			\tie[snake=true,snakelen=3,style=solid,bend=45]{{3,1},{1,1}}
			\tie[snake=true,snakelen=3,style=solid]{{3,0},{4,0}}
			\tie[snake=true,snakelen=3,style=solid,bend=-45]{{3,1},{5,1}}
			\tie[snake=true,snakelen=3,style=solid]{{5,0},{5,1}}
		\end{tikzpicture}
		\begin{tikzpicture}
			\vpartition[font=0.5,nstr=5,type=2,tkzpic=0]{{1,-1}}
			\tie[snake=true,snakelen=3,style=solid,bend=45]{{3,1},{1,1}}
			\tie[snake=true,snakelen=3,style=solid]{{3,0},{4,0}}
			\tie[snake=true,snakelen=3,style=solid,bend=-45]{{3,1},{5,1}}
			\tie[snake=true,snakelen=3,style=solid,bend=-45]{{5,0},{2,0}}
		\end{tikzpicture}
		\begin{tikzpicture}
			\vpartition[font=0.5,nstr=5,type=2,tkzpic=0]{{1,-1}}
			\tie[snake=true,snakelen=3,style=solid,bend=45]{{3,1},{1,1}}
			\tie[snake=true,snakelen=3,style=solid]{{3,0},{4,0},{5,0}}
			\tie[snake=true,snakelen=3,style=solid,bend=-45]{{3,1},{5,1}}
		\end{tikzpicture}
		\begin{tikzpicture}
			\vpartition[font=0.5,nstr=5,type=2,tkzpic=0]{{1,-1}}
			\tie[snake=true,snakelen=3,style=solid,bend=45]{{3,1},{1,1}}
			\tie[snake=true,snakelen=3,style=solid]{{3,0},{4,0}}
			\tie[snake=true,snakelen=3,style=solid,bend=-45]{{3,1},{5,1}}
		\end{tikzpicture}
	\fi
}

\newcommand{\figurefortwo}{
	\ifnum\aspdf=1\[
		\vcdraw{\includegraphics{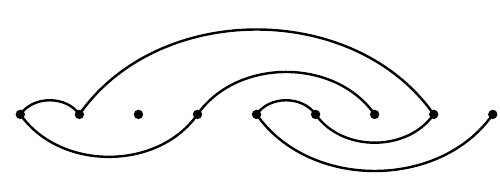}}
		\quad\,\,=\quad
		\vcdraw{\includegraphics{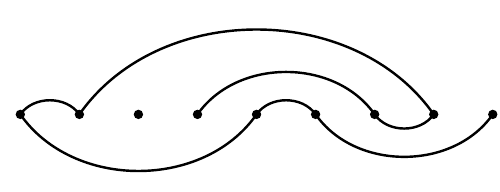}}
	\]\else
		\begin{tikzpicture}
			\arcpartition[bend=55,num=9,tkzpic=0,type=0]{{1,2,8},{4,7},{5,6}}
			\arcpartition[bend=-55,bull=0,tkzpic=0,type=0]{{1,4},{5,9},{6,8}}
		\end{tikzpicture}
		\begin{tikzpicture}
			\arcpartition[bend=55,num=9,tkzpic=0,type=0]{{1,2,8},{4,7},{5,6}}
			\arcpartition[bend=-55,bull=0,tkzpic=0,type=0]{{1,5},{6,9},{7,8}}
		\end{tikzpicture}
	\fi
}

\newcommand{\figureforthr}{
	\ifnum\aspdf=1\[\left(
		\vcdraw{\includegraphics{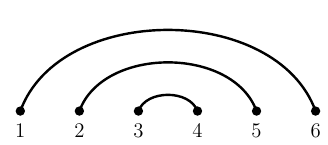}}\,\,\,,
		\vcdraw{\includegraphics{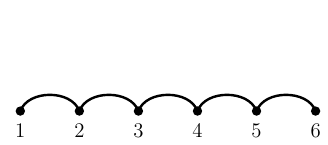}}\,\right)\quad=\quad
		\vcdraw{\includegraphics{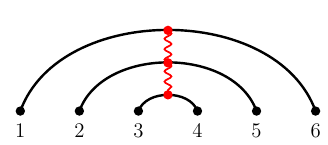}}
	\]\else\[
		\vcdraw{\begin{tikzpicture}
			\arcpartition[norma=0.9,bend=70,tkzpic=0,type=1]{{1,6},{2,5},{3,4}}
		\end{tikzpicture}}
		\vcdraw{\begin{tikzpicture}
			\arcpartition[norma=0.9,bend=70,tkzpic=0,type=1]{{1,2,3,4,5,6}}
		\end{tikzpicture}}
		\vcdraw{\begin{tikzpicture}
			\arcpartition[norma=0.9,bend=70,tkzpic=0,type=1]{{1,6},{2,5},{3,4}}
			\tie[snake=true,snakelen=3,style=solid]{{3.5,0.165},{3.5,0.49},{3.5,0.82}}
		\end{tikzpicture}}
		\]
	\fi
}

\newcommand{\figureforfor}{
	\ifnum\aspdf=1\[
		\vcdraw{\includegraphics{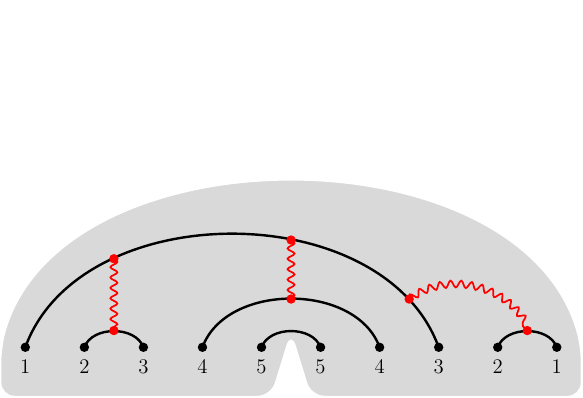}}\kern+0.7em\to\kern+0.1em
		\vcdraw{\includegraphics{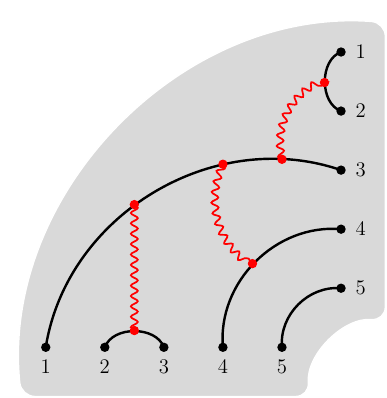}}\quad\!\to\!\quad
		\vcdraw{\includegraphics{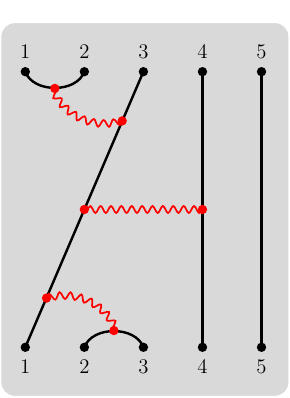}}
	\]\else\[
		\vcdraw{\begin{tikzpicture}
			\path[rounded corners,draw=white,fill=gray!30]
				(5.65,0)to[bend right=75]
				(-0.25,0)to
				(-0.25,-0.5)to
				(2.5,-0.5)to			
				(2.7,0.15)to
				(2.9,-0.5)to
				(5.65,-0.5)--cycle;
			\arcpartition[
				customlabels={1,2,3,4,5,5,4,3,2,1},
				bend=70,
				norma=3.3,
				tkzpic=0,
				type=-6
			]{{1,8},{9,10},{2,3},{4,7},{5,6}}
			\tie[snake=true,snakelen=3,style=solid]{{5.5,0.49},{5.5,1.09}}
			\tie[snake=true,snakelen=3,style=solid]{{2.5,0.17},{2.5,0.9}}
			\tie[bend=-50,snake=true,snakelen=3,style=solid]{{9.5,0.17},{7.5,0.49}}
		\end{tikzpicture}}
		\vcdraw{\begin{tikzpicture}
			\path[rounded corners,draw=white,fill=gray!30]
				(3.45,3.3)to[bend right=50]
				(-0.25,-0.5)to
				(2.68,-0.5)to[bend left=50]
				(3.45,0.27)--cycle;
			\arcpartition[
				customlabels={1,2,3,4,5},
				bend=70,
				norma=3.3,
				num=5,
				tkzpic=0,
				type=-6
			]{{2,3}}
			\filldraw(3,0.6)circle(0.04);\node[scale=0.5]at(3.2,0.6){5};
			\filldraw(3,1.2)circle(0.04);\node[scale=0.5]at(3.2,1.2){4};
			\filldraw(3,1.8)circle(0.04);\node[scale=0.5]at(3.2,1.8){3};
			\filldraw(3,2.4)circle(0.04);\node[scale=0.5]at(3.2,2.4){2};
			\filldraw(3,3.0)circle(0.04);\node[scale=0.5]at(3.2,3.0){1};
			\draw[bend right=-70,line width=0.7](3,2.4)to(3,3.0);	
			\draw[bend right=-50,line width=0.7](2.4,0)to(3,0.6);	
			\draw[bend right=-50,line width=0.7](1.8,0)to(3,1.2);	
			\draw[bend right=-50,line width=0.7](0.0,0)to(3,1.8);	
			\tie[snake=true,snakelen=3,style=solid]{{2.5,0.17},{2.5,1.45}}
			\tie[bend=40,snake=true,snakelen=3,style=solid]{{4.5,0.85},{4,1.86}}
			\tie[bend=-40,snake=true,snakelen=3,style=solid]{{5.72,2.69},{5,1.91}}
		\end{tikzpicture}}
		\vcdraw{\begin{tikzpicture}
			\path[rounded corners,draw=white,fill=gray!30]
				(-0.25,3.3)to
				(-0.25,-0.5)to
				(2.68,-0.5)to
				(2.68,3.3)--cycle;
			\vvpartition[bend=70,height=2.8,norma=3.3,type=2]{{1,2},{3,-1},{-2,-3},{4,-4},{5,-5}}
			\tie[bend=-40,snake=true,snakelen=3,style=solid]{{1.5,2.63},{2.64,2.3}}
			\tie[bend=-40,snake=true,snakelen=3,style=solid]{{2.5,0.17},{1.36,0.5}}
			\tie[snake=true,snakelen=3,style=solid]{{2,1.4},{4,1.4}}
		\end{tikzpicture}}
		\]
	\fi
}

\newcommand{\figureforfiv}{
	\ifnum\aspdf=1\[
		\vcdraw{\includegraphics{pics/043.pdf}}\!\!\cdots\!\!
		\vcdraw{\includegraphics{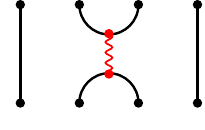}}\cdots\!\!
		\vcdraw{\includegraphics{pics/043.pdf}}\]
	\else\[
		\begin{tikzpicture}
			\strands[nstr=4,type=0,tkzpic=0]{t2}
			\tie[snake=true,snakelen=3,style=solid]{{2.5,0.295},{2.5,0.705}}
		\end{tikzpicture}
	\]\fi
}

\newcommand{\figureforsix}{
	\ifnum\aspdf=1\[
		\includegraphics{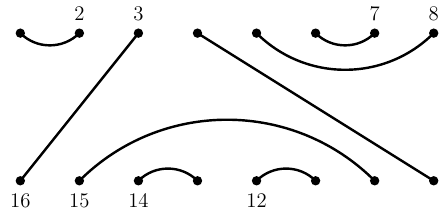}
	\]\else\[
		\vpartition[
			height=1.5,
			type=6,
			customlabelsabove={,2,3,,,,7,8},
			customlabelsbelow={16,15,14,,12,,,}]
		{{1,2},{3,-1},{4,-8},{5,8},{6,7},{-2,-7},{-3,-4},{-5,-6}}
	\]\fi
}

\newcommand{\figureforsev}{
	\ifnum\aspdf=1\[
		\includegraphics{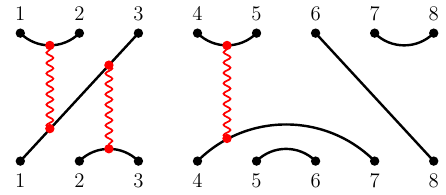}\kern+4em
		\includegraphics{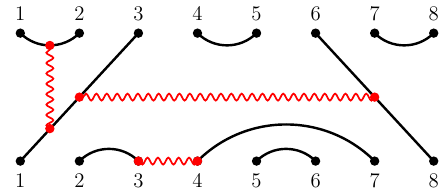}
	\]\else\[
		\begin{tikzpicture}
			\vvpartition[height=1.3,type=2]{{1,2},{3,-1},{4,5},{6,-8},{7,8},{-2,-3},{-4,-7},{-5,-6}}
			\tie[snake=true,snakelen=3,style=solid]{{1.5,1.175},{1.5,0.33}}
			\tie[snake=true,snakelen=3,style=solid]{{2.5,0.125},{2.5,0.975}}
			\tie[snake=true,snakelen=3,style=solid]{{4.5,1.175},{4.5,0.23}}
		\end{tikzpicture}
		\begin{tikzpicture}
			\vvpartition[height=1.3,type=2]{{1,2},{3,-1},{4,5},{6,-8},{7,8},{-2,-3},{-4,-7},{-5,-6}}
			\tie[snake=true,snakelen=3,style=solid]{{1.5,1.175},{1.5,0.33}}
			\tie[snake=true,snakelen=3,style=solid]{{3,0},{4,0}}
			\tie[snake=true,snakelen=3,style=solid]{{2,0.65},{7,0.65}}
		\end{tikzpicture}
	\]\fi
}

\newcommand{\figureforeig}{
	\ifnum\aspdf=1\[
		\vcdraw{\includegraphics{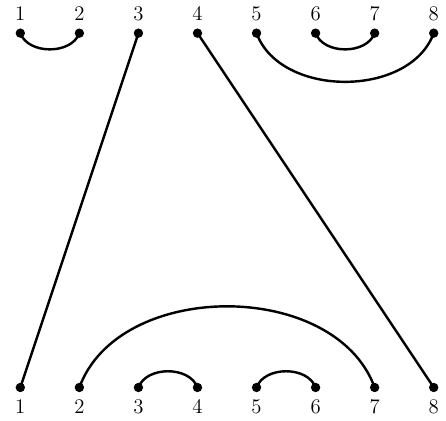}}\quad=\quad
		\vcdraw{\includegraphics{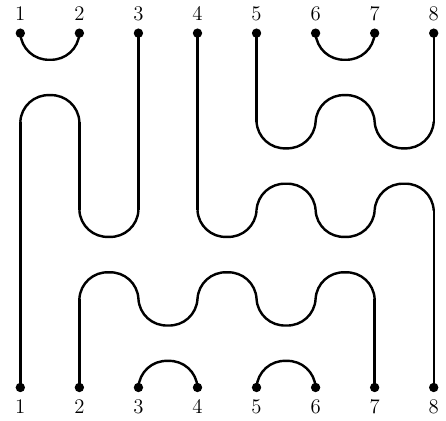}}
	\]\else\[
		\vpartition[bend=70,height=3.6,type=2]{{1,2},{3,-1},{4,-8},{5,8},{6,7},{-2,-7},{-3,-4},{-5,-6}}
		\strands[height=0.9,type=2]{t1-t6*t5-t7*t2-t4-t6*t3-t5}
	\]\fi
}

\newcommand{\figurefornin}{
	\ifnum\aspdf=1\[
		\vcdraw{\includegraphics{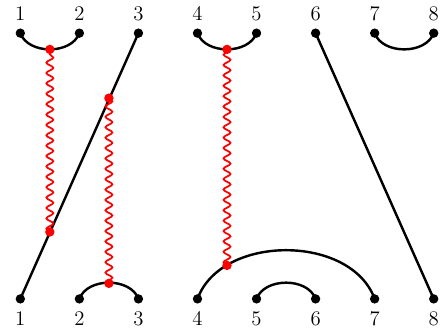}}\quad=\quad
		\vcdraw{\includegraphics{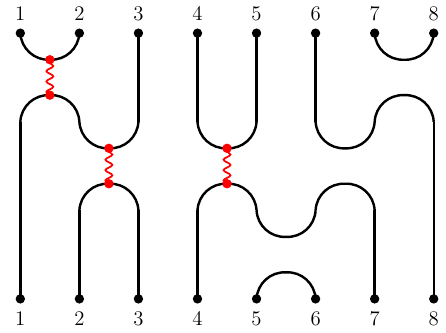}}
	\]\else\[
		\begin{tikzpicture}
			\vvpartition[bend=70,height=2.7,type=2]{{1,2},{3,-1},{4,5},{6,-8},{7,8},{-2,-3},{-4,-7},{-5,-6}}
			\tie[snake=true,snakelen=3,style=solid]{{1.5,2.535},{1.5,0.68}}
			\tie[snake=true,snakelen=3,style=solid]{{2.5,0.158},{2.5,2.04}}
			\tie[snake=true,snakelen=3,style=solid]{{4.5,2.535},{4.5,0.34}}
		\end{tikzpicture}
		\begin{tikzpicture}
			\strands[height=0.9,tkzpic=0,type=2]{t1-t7*t2-t4-t6*t5}
			\tie[snake=true,snakelen=3,style=solid]{{1.5,2.43},{1.5,2.07}}
			\tie[snake=true,snakelen=3,style=solid]{{2.5,1.53},{2.5,1.17}}
			\tie[snake=true,snakelen=3,style=solid]{{4.5,1.53},{4.5,1.17}}
		\end{tikzpicture}
	\]\fi
}

\newcommand{\figurefifty}{
	\ifnum\aspdf=1\[
		\begin{array}{cccc}		
			&\vcdraw{\includegraphics{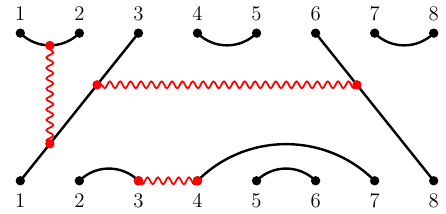}}
			&\to
			&\vcdraw{\includegraphics{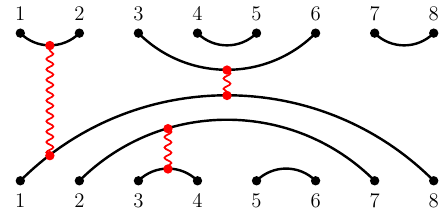}}\\[1.5cm]
			\to
			&\vcdraw{\includegraphics{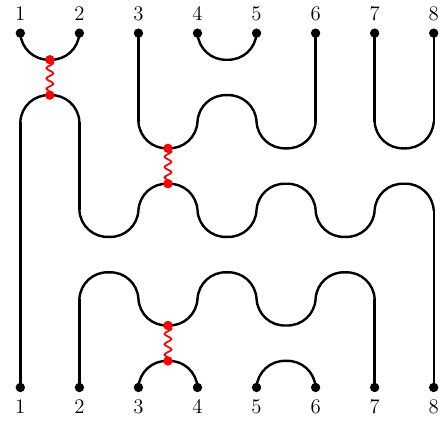}}
			&\to
			&\vcdraw{\includegraphics{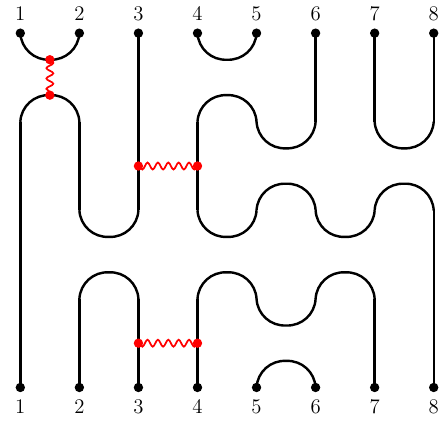}}
		\end{array}
	\]\else\[
		\begin{tikzpicture}
			\vvpartition[bend=45,height=1.5,type=2]{{1,2},{3,-1},{4,5},{6,-8},{7,8},{-2,-3},{-4,-7},{-5,-6}}
			\tie[height=1.5,snake=true,snakelen=3,style=solid]{{1.5,0.916},{1.5,0.25}}
			\tie[height=1.5,snake=true,snakelen=3,style=solid]{{3,0},{4,0}}
			\tie[height=1.5,snake=true,snakelen=3,style=solid]{{2.3,0.65},{6.7,0.65}}
		\end{tikzpicture}
		\begin{tikzpicture}
			\vvpartition[bend=45,height=1.5,type=2]{{1,2},{3,6},{4,5},{7,8},{-1,-8},{-2,-7},{-3,-4},{-5,-6}}
			\tie[height=1.5,snake=true,snakelen=3,style=solid]{{1.5,0.916},{1.5,0.17}}
			\tie[height=1.5,snake=true,snakelen=3,style=solid]{{3.5,0.08},{3.5,0.353}}
			\tie[height=1.5,snake=true,snakelen=3,style=solid]{{4.5,0.578},{4.5,0.75}}
		\end{tikzpicture}
		\begin{tikzpicture}
			\strands[height=0.9,tkzpic=0,type=2]{t1-t4*t3-t5-t7*t2-t4-t6*t3-t5}
			\tie[snake=true,snakelen=3,style=solid]{{1.5,2.97},{1.5,3.33}}
			\tie[snake=true,snakelen=3,style=solid]{{3.5,0.27},{3.5,0.63}}
			\tie[snake=true,snakelen=3,style=solid]{{3.5,2.07},{3.5,2.43}}
		\end{tikzpicture}
		\begin{tikzpicture}
			\strands[height=0.9,tkzpic=0,type=2]{t1-t4*t5-t7*t2-t4-t6*t5}
			\tie[snake=true,snakelen=3,style=solid]{{1.5,2.97},{1.5,3.33}}
			\tie[snake=true,snakelen=3,style=solid]{{3,0.45},{4,0.45}}
			\tie[snake=true,snakelen=3,style=solid]{{3,2.25},{4,2.25}}
		\end{tikzpicture}
	\]\fi
}

\newcommand{\figurefifone}{
	\ifnum\aspdf=1\[
		\vcdraw{\includegraphics{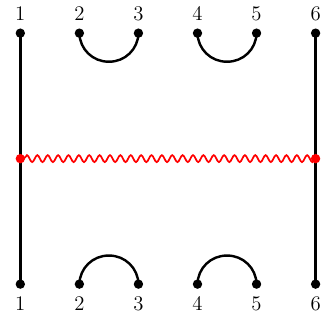}}\to
		\vcdraw{\includegraphics{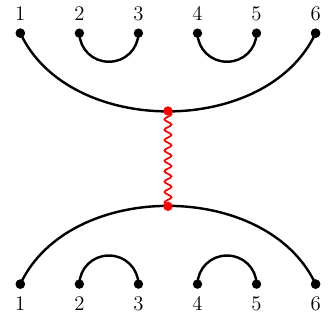}}\to
		\vcdraw{\includegraphics{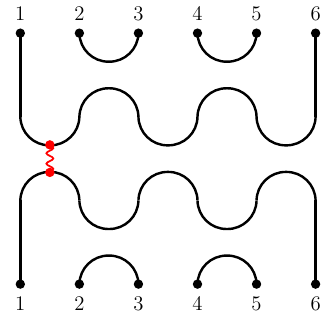}}\to
		\vcdraw{\includegraphics{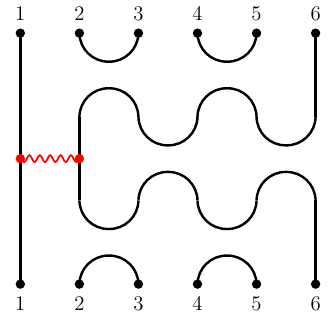}}
	\]\else\[
		\begin{tikzpicture}
			\strands[height=2.55,nstr=6,tkzpic=0,type=2]{t2-t4}
			\tie[height=2.55,snake=true,snakelen=3,style=solid]{1,6}
		\end{tikzpicture}
		\begin{tikzpicture}
			\strands[height=2.55,nstr=6,tkzpic=0,type=2]{r0-t2-t4-r5}
			\tie[color=black,height=2.55,style=solid,bend=-65,tiewidth=0.7]{{1,1},{6,1}}
			\tie[color=black,height=2.55,style=solid,bend=+65,tiewidth=0.7]{{1,0},{6,0}}
			\tie[height=2.55,snake=true,snakelen=3,style=solid]{{3.5,0.31},{3.5,0.69}}
		\end{tikzpicture}
		\begin{tikzpicture}
			\strands[height=0.85,tkzpic=0,type=2]{t2-t4*t1-t3-t5*t2-t4}
			\tie[snake=true,snakelen=3,style=solid]{{1.5,1.135},{1.5,1.415}}
		\end{tikzpicture}
		\begin{tikzpicture}
			\strands[height=0.85,tkzpic=0,type=2]{t2-t4*t3-t5*t2-t4}
			\tie[height=2.55,snake=true,snakelen=3,style=solid]{{1,0.5},{2,0.5}}
		\end{tikzpicture}
	\]\fi
}

\newtheorem{corollary}{Corollary}
\newtheorem{lemma}{Lemma}
\newtheorem{proposition}{Proposition}
\newtheorem{theorem}{Theorem}

\theoremstyle{definition}
\newtheorem{definition}{Definition}

\newtheorem{notation}{Notation}
\newtheorem{remark}{Remark}

\newcommand{\N}{\mathbb{N}}

\newcommand{\R}{\mathcal{R}}
\newcommand{\T}{\mathcal{T}}
\newcommand{\Sym}{\mathcal{S}}
\newcommand{\IS}{\mathcal{IS}}

\newcommand{\CC}{\mathfrak{C}}
\newcommand{\PP}{\mathcal{PP}}
\newcommand{\I}{\mathcal{I}}
\newcommand{\J}{\mathcal{J}}
\newcommand{\PR}{\mathcal{PR}}
\newcommand{\tTL}{\mathrm{tTL}}

\renewcommand{\P}{\mathcal{P}}

\newcommand{\e}{\epsilon}

\newcommand{\vcdraw}[1]{\vcenter{\hbox{#1}}}

\begin{document}

\title{Ramified inverse and planar monoids}

\author[F. Aicardi]{Francesca Aicardi}
\address{Sistiana Mare 56, 34011 Trieste, Italy.}
\email{francescaicardi22@gmail.com}

\author[D. Arcis]{Diego Arcis}
\address{Departamento de Matem\'aticas, Universidad de La Serena, Cisternas 1200, 1700000 La Serena, Chile.}
\email{diego.arcis@userena.cl}

\author[J. Juyumaya]{Jes\'us Juyumaya}
\address{Instituto de Matem\'aticas, Universidad de Valpara\'{\i}so\\Gran Breta\~na 1111, 2340000 Valpara\'{\i}so, Chile.}
\email{juyumaya@gmail.com}

\date{}
\keywords{}
\subjclass{20F05,20M18,20M20,05A18,20F36,57M27,20C08}
\thanks{}

\maketitle

\begin{abstract}
Ramified monoids are a class of monoids introduced by the authors. The main motivation for considering these monoids comes from knot theory, see \cite{AiJuJKTR2016,AiJuMathZ2018,AiJu2021}. Thus, in \cite{AiArJu2023} we have studied the ramified monoids of the symmeytric group and of the Brauer monoid, among others. This paper study the ramified of the inverse symmetric monoid, which plays a notable role in knot theory as well, see \cite{BiStRaReJKTR2012}. Here is also introduced the notion of planar ramified monoid. In particular, we give presentations for some planar ramified monoids arising from noncrossing set partitions.
\end{abstract}

\maketitle

\begin{center}
\begin{minipage}{10cm}
\tableofcontents
\end{minipage}
\end{center}

\section*{Introduction}

In \cite[Definition 10]{AiArJu2023}, it  was defined the ramified monoid  attached to every  submonoid  of the partition monoid \cite{Jo1993, HaRa2005}. The initial motivation to define ramified monoids arises from the tied symmetric monoid  which is the Coxeter version of the tied braid monoid, see \cite{AiJuJKTR2016,AiArJu2023}. In the context of knot theory, these tied monoids  are to the so-called bt-algebra \cite{AiJuJKTR2016, MaIMRN2018, PoWaProc2018, RyHaJAC2011}, as the symmetric group  and braid group are to the Iwahori--Hecke algebra. It should be noted that these  monoids are constructed as semidirect products. This construction was carried out for other Coxeter or Artin-type monoids, thus obtaining several families of tied monoids, see \cite{ArJu2021} for details. However, for other monoids of interest in knot theory, such as the Brauer, Jones and inverse  symmetric monoids, among others, it is not possible to attach a tied monoid by applying the techniques used in \cite{ArJu2021}. The ramified monoid concept comes to solve this problem, that is, the ramification yields  a tied version of the Brauer monoid, which in turn yields the tBMW algebra \cite{AiJuMathZ2018} and also recover the tied symmetric monoid.

This article concerns with the construction of presentations of the ramified (or tied) inverse symmetric monoid and some   ramified of planar related monoids. Inverse monoids were introduced, independently, by  V.V. Wagner \cite{Wa1952} and G. B. Preston \cite{PrJLMS1954}. The inverse symmetric monoid can be considered as a generalization of the symmetric group and plays an important role in the inverse monoid theory. For instance, the Wagner--Preston theorem says: every inverse monoid embeds in an appropriate symmetric inverse monoid. This theorem is the classical Cayley's theorem for monoids. See \cite{LiMSM-AMS1996} for a survey on inverse symmetric monoids.

On the other hand, famous deformations of inverse symmetric monoids are the Rook algebras introduced by L. Solomon \cite{SoGD1990}. In \cite{BiStRaReJKTR2012}, the Jones and Alexander polynomial were obtained from representations of the Rook algebra. To be more precise, these representations are obtained from the planar Rook algebra, see  \cite{BiStRaReJKTR2012} for details, cf. \cite{HaRa2004}. This relation between Rook algebra and knot theory is another of the reasons that motivated this work, since it is plausible to think of the existence of a ramified Rook algebra built from the presentation of the ramified inverse symmetric monoids constructed here.

The main objectives of the article are to build presentations of the ramified monoid $\R(IS_n)$, of the inverse symmetric monoid on $n$ points $IS_n$, and of some monoids related to it. Observe that the results in this article should be obtained for the inverse braid monoid, see \cite{EaLaAim2004,EastAiM2007}.

The article is organized as follows. In Section \ref{Preliminaries} we recall some facts about the monoid of set partition $P_n$ and the partition monoid $\CC_n$. In particular, we discuss the diagrammatic realization of these monoids as well as the relation between them. In Section \ref{SnISn}, we recall some details of the inverse symmetric monoid $IS_n$. In particular, we take a close look at the realization of it as the submonoid $\IS_n$ of $\CC_n$, that is, as the one formed by the set partitions of $[2n]$ that have only lines or points as blocks. By using this diagrammatic realization a normal form is determined for the elements of $IS_n$ (Remark \ref{normalFormISn}), which will be used later. In Section \ref{RamTiedMon}, we recall what a ramified partition is and the ramified monoid of a submonoid of $\CC_n$. Also, we include a couple of general properties of ramified monoids, see Proposition \ref{facRam} and its corollary; this section conclude by recalling that the tied symmetric monoid (see \cite[Subsection 5.1.1]{AiArJu2023}) is the ramified monoid of the symmetric group $S_n$. Section \ref{RamSymInvMon} begins by calculating the cardinality of the ramified monoid $\R(IS_n)$ of $IS_n$ which is obtained directly by using Proposition \ref{facRam}. We continue showing the initial motivation of this article, which is to find a presentation of the ramified monoid $\R(IS_n)$, see Theorem \ref{PresRISn}. To prove this theorem we provide a normal form for the elements of $\R(\IS_n)$, see Remark \ref{lemNormRISn} for details. In Section \ref{PlanMons}, we introduce the concept of planar ramified monoid $\PR(M)$, for every planar submonoid $M$ of $\CC_n$. Note that this extends the definition of planar monoid given by Jones in \cite{Jo1983}. In this section we give also presentations for the monoids $\R(\I_n)$ (Theorem \ref{presRIn}), $\PR(\I_n)$ (Theorem \ref{presPRIn}) and $\PR(\J_n)$ (Theorem \ref{PRJnisomTTL}). Observe that the monoid $\I_n$ is one of the factors that appears in the decomposition $IS_n=\I_nS_n$, see Eq. (\ref{InPres}). Also note that thanks to Proposition \ref{facRam} we have $\R(ISn)=\R(\I_n)S_n$.

{\it Notation:} in this manuscript, for integers $a,b$, we will denote by $[a,b]$ the interval of integers $k$ satisfying $a\leq k\leq b$. If $a=1$, this interval will be denoted simply by $[b]$ instead. Further, for a set $X$, we will denote by $X^*$ the free monoid generated by it.

\section{Preliminaries}\label{Preliminaries}

This section recalls two classical monoids used throughout the article, that is, the monoid of set partitions and the partition monoid. We discuss the diagrammatic realization of these monoids as well as the relation between them.

\subsection{Set partitions}

A {\it set partition} of a set $A$ is a finite collection $I$ of nonempty sets $I_1,\ldots,I_k$, called {\it blocks}, such that $I_1\cup\cdots\cup I_k=A$ and $I_i\cap I_j=\emptyset$ for all $i,j\in[k]$ with $i\neq j$. We denote $I$ by $(I_1,\ldots,I_k)$ if $A$ is linearly ordered and $\min(I_1)<\cdots<\min(I_k)$. The collection of set partitions of $A$ is denoted by $P(A)$, and for a positive integer $n$ we shall write $P_n$ instead of $P([n])$. The number of set partitions of $[n]$ is the $n$th {\it Bell number} $b_n:=|P_n|$, see \cite[A000110]{OEIS}.

Diagrammatically, set partitions are usually represented by {\it arc diagrams}, however, here we represent a partition of  $[n]$ by a {\it diagram of ties}, i.e., by $n$ parallel lines, playing the role of the elements of $[n]$, which may be connected by some red springs \cite{AiJuJKTR2016}, called {\it ties}, if they belong to the same block. Note that, due to transitivity, not every pair of lines belonging the same block need to be connected by a tie. See Figure \ref{figureone}.
\begin{figure}[H]
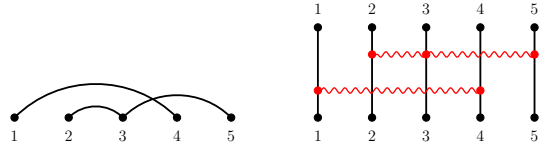

\figureone
\caption{Diagrams of the set partition $(\{1,4\},\{2,3,5\})$ of $[5]$.}
\label{figureone}
\end{figure}

For $B\subseteq A$ and $I\in P(A)$, we denote by $I\cap B$ the set partition of $B$ obtained by removing the elements of $A\setminus B$ from the blocks of $I$, i.e. $I\cap B=\{K\cap B\mid K\in A\}\setminus\emptyset$.

\subsection{Monoids of set partitions}\label{SetPartSub}
There is a partial order $\preceq$ given by refinement, which gives to $P(A)$ a structure of poset, i.e. $I\preceq J$ if each block of $J$ is a union of blocks of $I$. The collection $P(A)$ together with the product $IJ=\sup(I,J)$ is an idempotent commutative monoid with identity $1_A=\{\{a\}\mid a\in A\}$. Denote $1_{[n]}$ by $1_n$, or simply $1$. 

\begin{remark}\label{PA=Pn}
We have $P(A)\simeq P_n$ for all set $A$ with $|A|=n$. The monoid $P_n$ is called the {\it monoid of set partitions}.
\end{remark}

For every nonempty subset $B$ of $A$, we denote by $e_B$ the set partition of $A$ in which $B$ is its unique nontrivial block, that is, $e_B=\{B\}\cup\{\{a\}\mid a\in A\setminus B\}$. We set $e_{i,j}=e_{\{i,j\}}$. These set partitions are the generators of $P_n$. See Figure \ref{figuretwo}.
\begin{figure}[H]
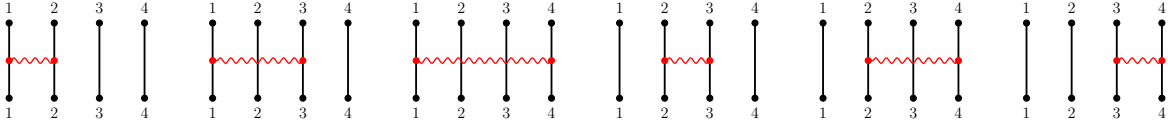

\figuretwo
\caption{Generators of $P_4$.}
\label{figuretwo}
\end{figure}

\begin{theorem}[FitzGerald {\cite[Theorem 2]{Fi03}}]
The monoid $P_n$ can be presented by generators $e_{i,j}$, with $i,j\in[n]$ and $i<j$, subject to the following relations:
\begin{align}
e_{i,j}^2&=e_{i,j}\quad\text{for all }i<j,\label{ee1}\\
e_{i,j}e_{r,s}&=e_{r,s}e_{i,j}\quad\text{for all }i<j\text{ and }r<s,\label{ee2}\\
e_{i,j}e_{i,k}&=e_{i,j}e_{j,k}=e_{i,k}e_{j,k}\quad\text{for all }i<j<k.\label{ee3}
\end{align}
\end{theorem}

\begin{proposition}[Normal form {\cite[Proposition 3.3]{ArJu2021}}]\label{nfPn}
Every set partition $I=(I_1,\ldots,I_k)$ of $[n]$ has a unique decomposition\[I=e_{I_1}\cdots e_{I_k}\quad\text{with}\quad e_{I_i}=e_{i_1,i_2}e_{i_2,i_3}\cdots e_{i_{p-1},i_p}\quad\text{where}\quad I_i=\{i_1<\cdots<i_p\}.\] 
\end{proposition}
See Figure \ref{figurethr} for an example.

\begin{figure}[H]
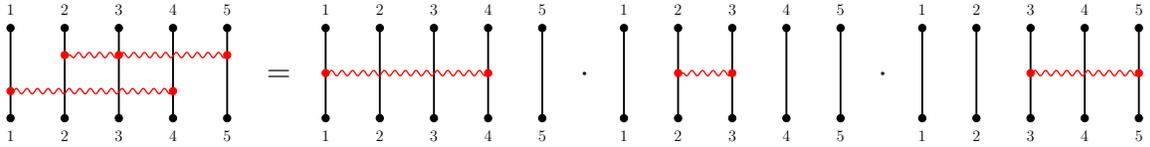

\figurethr
\caption{Normal decomposition of $(\{1,4\},\{2,3,5\})$.}
\label{figurethr}
\end{figure}

\begin{notation}\label{notInj}
By abuse of notation, for sets $A,B$ and $(I,J)\in P(A)\times P(B)$, we simply denote by $IJ$ the product $(I\cup1_{B\setminus A})(1_{A\setminus B}\cup J)$ in $P(A\cup B)$.
\end{notation}

\subsection{Partition monoid}\label{MonPartSub}
Here, every set partition of $[2n]$ will be represented by means of a {\it linear graph}, i.e., by $n$ aligned dots above, playing the role of the elements of $[n]$, and $n$ aligned dots below, inversely sorted, playing the role of the elements of $[n+1,2n]$, which may be connected by some lines when they belong to the same block. To our purpose it is convenient to relabel the dots below by replacing $n+k$ by $n+1-k$. See Figure \ref{figureonefiv}. As with diagram of ties, due to transitivity, not every pair of points in the same block need to be connected. In this context, the blocks containing only one element (trivial blocks) are called {\it points} and the blocks $\{i,n+j\}$ with $i,j\in[n]$ are called {\it lines}. The blocks that contain elements of $[n]$ and $[2n]\setminus[n]$ are called {\it generalized lines} and the blocks containing only two elements are usually called {\it arcs}; more specifically, arcs either contained in $[n]$ or disjoint with $[n]$ are called {\it brackets}, otherwise they are called {\it lines}.
\begin{figure}[H]
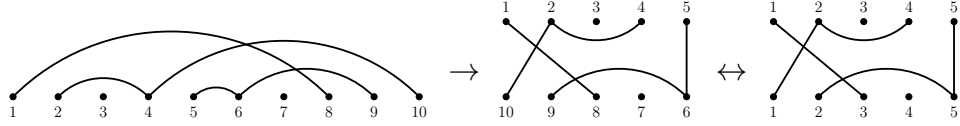

\figureonefiv
\caption{Linear graph of $(\{1,8\},\{2,4,10\},\{3\},\{5,6,9\},\{7\})$.}
\label{figureonefiv}
\end{figure}

Given $I,J\in P([2n])$, we use now the so--called {\it concatenation product} $I*J$ of $I$ with $J$, which is defined as $I_XJ^X$, where $I_X$ (resp. $J^X$) is obtained by replacing $n+k$ (resp. $k$) by $x_k$ from the blocks of $I$ (resp. $J$) for all $k\in[n]$, and $X$ is the auxiliary set $\{x_1,\ldots,x_n\}$ such that $X\cap[2n]$ is empty, see \cite[(17)]{AiArJu2023} and \cite{Xi1999} for details. Figure \ref{figurefou} shows the concatenation product in terms of diagrams. The set $P([2n])$ furnished with the {\it concatenation product $*$} is a noncommutative monoid with identity $1=(\{1,2n\},\ldots,\{n,n+1\})$, called the {\it partition monoid}, which is usually denoted by $\CC_n$. 
\begin{figure}[H]
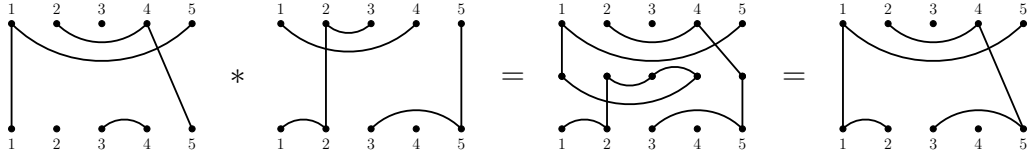

\figurefou
\caption{Concatenation product of $(\{1,5,10\},\{2,4,6\},\{3\},\{7,8\},\{9\})$ with $(\{1,4\},\{2,3,9,10\},\{5,6,8\},\{7\})$.}
\label{figurefou}
\end{figure}

\subsubsection{}\label{MonPartSubSub}
The collection of set partitions of $[2n-1]$ can be regarded as the subcollection of all set partitions of $[2n]$ that contain the singleton block $\{n+1\}$. This subcollection becomes a subsemigroup of $\CC_n$ and will be denoted by $\CC_n^{\bullet}$. In order to get a monoid structure on $\CC_n^{\bullet}$, we will represent this subcollection by identifying the $n$th dots above and below with a blue line as in the diagram given in Figure \ref{figuretwoeig}. Under this identification, $\CC_n^{\bullet}$ can be regarded as a submonoid of $\CC_n$. See Figure \ref{figuretwonin} for an example.

\begin{figure}[H]
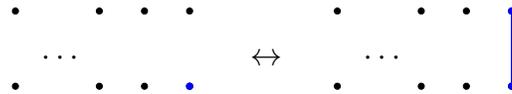

\figuretwoeig
\caption{Diagrammatic representation of $[2n-1]$.}
\label{figuretwoeig}
\end{figure}
\begin{figure}[H]
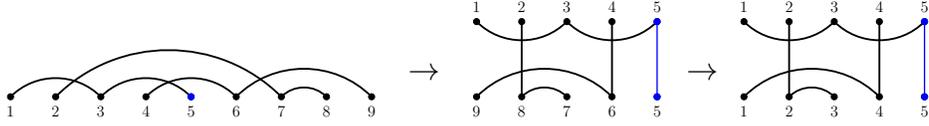

\figuretwonin
\caption{Set partition $(\{1,3,5\},\{2,7,8\},\{4,6,9\})$ of $[9]$ in $\CC_5^{\bullet}$.}
\label{figuretwonin}
\end{figure}

\begin{definition}\label{MBullet}
For a submonoid $M$ of $\CC_n$, we define $M^{\bullet}$ as the monoid $M\cap\CC_n^{\bullet}$.
\end{definition}

\section{The monoids $S_n$ and $IS_n$}\label{SnISn}

Here we recall the definition and the main properties of the inverse symmetric monoid $IS_n$. This  monoid is an extension of the symmetric group $S_n$, and in diagrammatic terms it contains elements of the partition monoid $\CC_n$ formed only by lines and points. Further, a normal form is shown for the elements of $IS_n$.

\subsection{Symmetric group}\label{SymGroupSub}
Recall that the symmetric group $S_n$ can be presented by generators $s_1,\ldots,s_{n-1}$ satisfying the following relations:
\begin{equation}\label{s}
s_i^2=1,\quad
s_is_j=s_js_i \text{ for }\vert i-j\vert>1,\quad
s_is_js_i=s_js_is_j\quad\text{for }\vert i-j\vert=1.
\end{equation}

The group $S_n$ can be realised as the submonoid $\Sym_n$ of $\CC_n$ formed by the partitions whose blocks are lines. Moreover, it coincides with the group of units of $\CC_n$, see \cite[Lemma 3.3]{KuMaCEJM2006}. Thus, for each $i\in[n-1]$, the generator $s_i$ can be realised as the set partition $(\{i,2n-i\},\{i+1,2n-i+1\}\}\cup\{\{k,2n-k+1\}\mid k\in[n]\setminus\{i,i+1\})$ represented in Figure \ref{figureonethr}. See Figure \ref{figurefiv} for an example.
\begin{figure}[H]
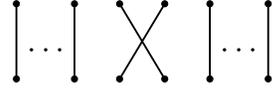

\figureonethr
\caption{Generator $s_i$.}
\label{figureonethr}
\end{figure}
\begin{figure}[H]
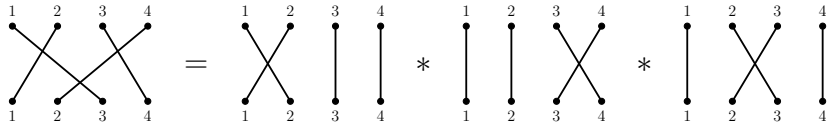

\figurefiv
\caption{Decomposition of $(\{1,6\},\{2,8\},\{3,5\},\{4,7\})$ in generators $s_i$'s.}
\label{figurefiv}
\end{figure}

For indexes $i,j$ with $i<j$, denote by $s_{i,j}\in\Sym_n$ the permutation exchanging $i$ with $j$, that is, $s_{i,j}=s_i\cdots s_{j-1}\cdots s_i$, which is represented as the set partition in Figure \ref{figuretwosix}. It is well known that $\Sym_n$ can be presented by generators $s_{i,j}$ with $i<j$, subject to the relations:\begin{equation}\label{genSnRels}s_{i,j}^2=s_{i,j},\quad s_{i,j}s_{j,k}=s_{j,k}s_{i,k}=s_{i,k}s_{i,j},\quad s_{i,j}s_{a,b}=s_{a,b}s_{i,j},\qquad i<j<k,\,\,\,a<b,\end{equation}where the continuous intervals $[i,j]$ and $[a,b]$ are either disjoint or nested.

\begin{figure}[H]
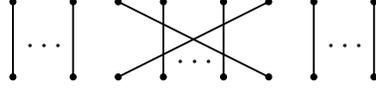

\figuretwosix
\caption{Generator $s_{i,j}$.}
\label{figuretwosix}
\end{figure}

\subsection{Symmetric inverse monoid}\label{SymInvMonSub}
A monoid $M$ is called {\it inverse} if its idempotents form a commutative submonoid and $M$ is {\it regular}, that is, for every $a\in M$ there is $b\in M$ such that $aba=a$ and $bab=b$. The prototype of inverse monoid is the so--called {\it symmetric inverse monoid}, which is a natural generalization of the symmetric group, formed by all partial transformations on a set of cardinality $n$ with multiplication given by the composition of functions. This monoid is denoted by $IS_n$ and was firstly studied in \cite{MuPCP1957}. Elements of $IS_n$ are also called {\it partial permutations}. It was shown in \cite[Remark 4.13]{KuMaCEJM2006} that $IS_n$ has a presentation with generators $s_1,\ldots,s_{n-1}$ satisfying (\ref{s}), and generators $r_1,\ldots,r_n$ subject to the following relations:
\begin{align}
r_i^2=r_i,&\quad
r_ir_j=r_jr_i\quad\text{for all }i,j,\label{r1}\\
s_ir_j&=r_{s_i(j)}s_i\quad\text{for all }i,j,\label{r2}\\
r_is_ir_i&=r_ir_{i+1}\quad\text{for }1\leq i\leq n-1.\label{r3}
\end{align}
Thus, $IS_n$ is the quotient of $\I_n\rtimes S_n$ by the congruence generated by relation (\ref{r3}), where $\I_n$ is the free idempotent commutative monoid of rank $n$, that is\begin{equation}\label{InPres}\I_n=\langle r_1,\ldots,r_n\mid r_i^2=r_i,\,r_ir_j=r_jr_i\rangle	\simeq I_1^{\times n}.\end{equation}

\begin{remark}
By using Tietze's transformations, we can show that $IS_n$ is presented by generators $s_{i,j}$ with $i<j$, and $r_1,\ldots,r_n$, subject to (\ref{genSnRels}), (\ref{r1}) and the following relations:\begin{equation}\label{gr23}s_{i,j}r_k=r_{s_{i,j}(k)}s_{i,j},\qquad r_is_{i,j}r_i=r_ir_j.\end{equation}
In \cite{PoUZL1961}, L. M. Popova shown that $IS_n$ can be presented with generators $s_1,\ldots,s_{n-1}$, $r$ subject to (\ref{s}) and the relations: 
\[r^2=r,\quad rs_{n-1}r=s_{n-1}rs_{n-1}r=rs_{n-1}rs_{n-1}\quad\text{and}\quad rs_i=s_ir\quad\text{if}\quad i<n-1.\]
On the other side, by defining $p_1=r_1$ and $p_i=p_{i-1}r_i$, we have that $IS_n$ is presented by generators $s_i$'s and $p_i$'s subject to (\ref{s}) and the relations: 
\[p_i^2=p_i\quad p_ip_j=p_jp_i,\quad p_is_ip_i=p_{i+1},\quad\text{and}\quad p_is_j=s_jp_i=p_i\quad\text{if}\quad j<i.\]
This one is the presentation of the monoid that generates, for $u=1$, the Rook algebra $\mathfrak{R}_n(u)$, see \cite{HaRa2004}.
\end{remark}

The symmetric inverse monoid can be realised as the submonoid $\IS_n$ of $\CC_n$ formed by the set partitions whose blocks are either lines or points, see \cite[(E2) p. 416]{KuMaCEJM2006}. Note that $S_n$ is also the group of units of $\IS_n$. In particular, for each $i\in[n]$, the generator $r_i$ can be realised as the set partition $\{\{i\},\{2n-i+1\}\}\cup\{\{k,2n-k+1\}\mid k\in[n]\setminus\{i\}\}$ in Figure \ref{riGenR}. See Figure \ref{figuresix} for an example.
\begin{figure}[H]
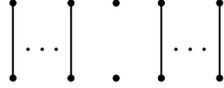

\figureonefou
\caption{Generator $r_i$.}
\label{riGenR}
\end{figure}

\begin{figure}[H]
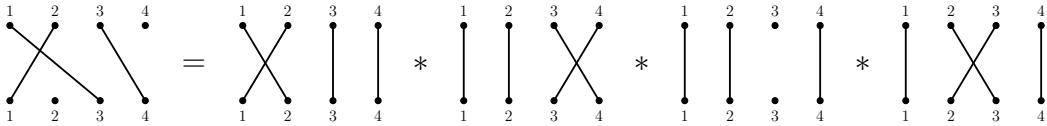

\figuresix
\caption{Decomposition in generators of $(\{1,6\},\{2,8\},\{3,5\},\{4\},\{7\})$.}
\label{figuresix}
\end{figure}

Note that $\I_n$ and $\I_{n+1}^{\bullet}$ are formed by the set partitions whose possible lines are all vertical, that is, lines $\{i,2n-i+1\}$ with $i\in[n]$. Since there are $\binom{n}{k}$ set partitions of $\I_n$ and of $\I_{n+1}^{\bullet}$ with $k$ lines,\[|\I_n|=|\I_{n+1}^{\bullet}|=\sum_{k=0}^n\binom{n}{k}=2^n.\]

\subsubsection{}\label{InvSymMonSubSub}
Elements of $\IS_n$ and $\IS_n^{\bullet}:=(\IS_n)^{\bullet}$ with $k$ lines are uniquely defined by the choice of the upper points and the lower points of the lines together with a permutation of $[k]$. Therefore, the cardinalities of $\IS_n$ \cite[A002720]{OEIS} is the following\[|\IS_n|=\sum_{k=0}^nk!\binom{n}{k}^2,\]while the cardinality of $\IS_n^{\bullet}$ \cite[A000262]{OEIS} is given by\[|\IS_n^{\bullet}|=\sum_{k=0}^{n-1}k!\binom{n}{k}\binom{n-1}{k}.\]See also \cite[Proposition 2.1]{{KuMaCEJM2006}}. Both cardinalities give \cite[A056953]{OEIS}.

\begin{remark}[Normal form]\label{normalFormISn}
For every $g\in\IS_n$ with exactly $m$ ($0\leq m\leq n$) lines, there are $(n-m)!$ permutations $s\in S_n$ containing the lines of $g$. Hence $g=r_{i_1}\cdots r_{i_k}s$, where $\{i_1\},\ldots,\{i_k\}$ with $i_1<\cdots<i_k$ are the points of $g$ contained in $[n]$. We denote by $g_p$ the unique permutation obtained by replacing the blocks $\{i_r\},\{j_r\}$ of $g$ with the line $\{i_r,j_r\}$, for all $r\in[k]$, where $\{j_1\},\ldots,\{j_k\}$ with $j_1<\cdots<j_k$ are the blocks of $g$ contained in $[2n]\setminus[n]$. Thus, the word $r_{i_1}\cdots r_{i_k}g_p$ in $(\{r_1,\ldots,r_n\}\cup\Sym_n)^*$, is a normal form of $g$. See Figure \ref{figuresev} for an example.
\end{remark}
\begin{figure}[H]
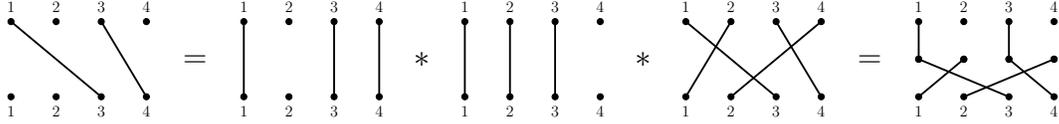

\figuresev
\caption{Normal decomposition of $(\{1,6\},\{2\},\{3,5\},\{4\},\{7\},\{8\})$.}
\label{figuresev}
\end{figure}

For every partial permutation $g\in\IS_n$, the permutation $g_p\in S_n$ defined in Remark \ref{normalFormISn} will be called the {\it completed permutation} of $g$.

\begin{remark}\label{minCrosNormISn}
Let $g$ be a partial permutation, and let $r_{i_1}\cdots r_{i_k}g_p$ be its normal form. By definition, $g_p$ is the permutation with minimal number of crossings which contains the same lines as $g$. Thus, by using (\ref{gr23}), we can write $g=rs$ for some $(r,s)\in\I_n\times S_n$, where $r=r_{i_1}\cdots r_{i_k}$ and $s=s'g_p$, with $s'$ is a permutation satisfying $rs'=r$.
\end{remark}

\section{Ramified and tied monoids}\label{RamTiedMon}

We first recall here the concept of ramified partition and show interpretations in terms of diagrams. Later, we recall the construction of the ramified monoid associated with any submonoid of the partition monoid and show some properties of this construction. The section ends showing that the tied monoid can be obtained as the ramified monoid of the symmetric group, see Remark \ref{tiedRamSn}.

\subsection{General background}

A {\it ramified partition} of a set $A$ is a pair $(I,J)$ of set partitions of $A$ such that $I\preceq J$, see \cite{MaEl2004} for more details. The collection of ramified partitions of $A$ is denoted by $RP(A)$, and for a positive integer $n$ we shall write $RP_n$ instead of $RP([n])$. The collection $RP(A)$ inherits the monoid structure of $P(A)\times P(A)$ because $I\preceq J$ and $H\preceq K$ implies that $IH\preceq JK$. Diagrammatically, we will represent ramified partitions of $[n]$ via {\it tied arc diagrams}, that is, by mixing arc diagrams and diagrams of ties, see \cite{Ai20} for details. Indeed, a ramified partition $(I,J)$ will be represented by connecting by ties the components of the linear graph of $I$ belonging to the same block of $J$. So, blocks of $I$ that are contained in the same block of $J$ must be transitively connected in the diagram of $(I,J)$. See Figure \ref{tiearcd1}
\begin{figure}[H]
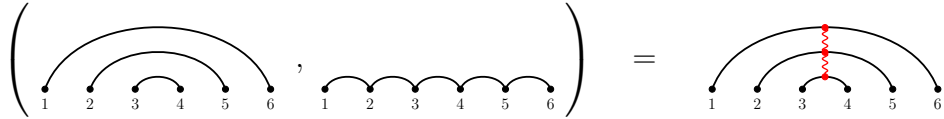

\figureforthr
\caption{Tied arc diagram of the ramified partition $(I,J)$ with $I\preceq J$, where $I=\{\{1,6\},\{2,5\},\{3,4\}\}$ and $J=\{\{1,2,3,4,5,6\}\}$}
\label{tiearcd1}
\end{figure}
Ramified partitions of $[2n]$ will be represented similarly by mixing linear graphs and diagrams of ties. See Figure \ref{figureonesix}. Cf. \cite[Figure 5]{AiArJu2023}.
\begin{figure}[H]
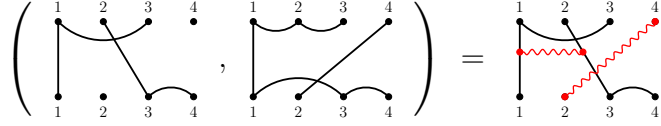

\figureonesix
\caption{Diagram of the ramified partition $(I,J)$ with $I\preceq J$, where $I=(\{1,3,8\},\{2,5,6\},\{4\},\{7\})$ and $J=(\{1,2,3,5,6,8\},\{4,7\})$.}
\label{figureonesix}
\end{figure}

The {\it ramified monoid} of a submonoid $M$ of $\CC_n$, denoted by $\R(M)$, is the monoid formed by the ramified partitions $(I,J)$ of $[2n]$ satisfying $I\in M$, see \cite[Definition 10]{AiArJu2023} for details. As mentioned in \cite{AiArJu2023}, every submonoid of $\CC_n$ embeds into its ramified monoid via the map $I\mapsto(I,I)$. Notice that $\R(\CC_n)$ is formed by all the ramified partitions of $[2n]$.
\begin{proposition}\label{ramId}
We have $\R(\{1\})\simeq P_n$.
\end{proposition}
\begin{proof}
The proof is a direct consequence of Remark \ref{PA=Pn}.
\end{proof}

In virtue of the proposition above, when there is no risk of confusion, we will write $e\in P_n$ instead of $(1,e)\in \R(\{1\})$.

\begin{proposition}\label{facRam}
Assume that $M=EG$ is a monoid with $G$ a group and $E$ the monoid of idempotents of $M$. Then $\R(M)=\R(E)G$.
\end{proposition}
\begin{proof}
As $\R(E)\subseteq\R(M)$ and $G$ embeds into $\R(G)\subseteq\R(M)$ via $g\mapsto(g,g)$, then $\R(E)G\subseteq\R(M)$. Let $(I,J)\in\R(M)$, then $I\preceq J$ with $I=eg$ for some $(e,g)\in E\times G$. We have $(I,J)=(I,I)(1,J)=(eg,eg)(1,J)=(e,e)(g,g)(1,J)=(e,e)(1,gJg^{-1})(g,g)=(e,egJg^{-1})(g,g)$. Since $e\preceq egJg^{-1}$, then $(e,egJg^{-1})\in\R(E)$, hence $(I,J)\in\R(E)G$.
\end{proof}

\begin{corollary}\label{RamiInv}
If $M$ is an inverse monoid as in Proposition \ref{facRam}, then $\R(M)$ is also inverse.
\end{corollary}
\begin{proof}
The idempotents $\R(E)$ of $\R(M)$ is a commutative monoid, so the proof follows by showing that $\R(M)$ is regular. Now, given $a\in\R(M)$, it can be written as $a=eg$ with $e\in\R(E)$, $g\in G$. Define $b=g^{-1}e$, we have $ aba=a$ and $bab=b$, that is, $\R(M)$ is regular. 
\end{proof}

\subsection{Tied symmetric monoid}\label{TiedSymMonSub}
As usual, we denote by $TS_n$ the {\it tied symmetric monoid}, see \cite{AiArJu2023} for details. This monoid is presented by generators $s_1,\ldots,s_{n-1}$, $e_1,\ldots,e_{n-1}$ subject to relation (\ref{s}) together with the following relations:
\begin{align}
e_i^2 =e_i,&\quad e_ie_j=e_je_i\quad\text{for all }i\text{ and }j,\label{TSn1}\\
s_ie_j =e_js_i&\quad\text{if }|i-j|\neq 1,\quad e_is_js_i=s_js_ie_j\quad\text{if }|i-j|=1,\label{TSn2}\\
e_ie_js_i=e_js_ie_j&=s_ie_ie_j\quad\text{if }|i-j|=1.\label{TSn3}
\end{align}

\begin{remark}\label{tiedRamSn}
In \cite[Theorem 3]{AiJu2021} it is proved that $TS_n=P_n\rtimes S_n$, c.f. {\cite[Theorem 4.2]{Ba13}}, where the generators $e_i$'s are realised as the ramified partition represented in Figure \ref{figureonesev}, that is $(1,\{\{i,i+1,2n-i+1,2n-i\}\}\cup\{\{k,2n-k+1\}\mid k\in[n]\setminus\{i\}\})$. On the other hand,  in {\cite[Theorem 19]{AiArJu2023}}, it was  proved that $\R(S_n)=TS_n$. By combining these facts with \cite[Corollary 2]{La1998}, we get that $\R(S_n)$ can be presented with generators $s_1,\ldots,s_{n-1}$ satisfying (\ref{s}), and generators $e_{i,j}$ with $i<j$ satisfying (\ref{ee1})--(\ref{ee3}), subject to the relations $s_ke_{i,j}=e_{s_k(i),s_k(j)}s_k$, where $e_{i,j+1}=s_j\cdots s_{i+1}e_is_{i+1}\cdots s_j=s_i\cdots s_{j-1}e_js_{j-1}\cdots s_i$.
\end{remark}
\begin{figure}[H]
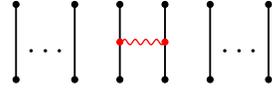

\figureonesev
\caption{Generator $e_i$.}
\label{figureonesev}
\end{figure}

Due to Proposition \ref{ramId} and Remark \ref{tiedRamSn}, we obtain $\R(S_n)=\R(\{1\})\rtimes S_n$. Furthermore, every element of $g\in\R(S_n)$ has a normal form $w=es$ in $(\{e_{i,j}\mid i<j\}\cup S_n)^*$, where $e$ is the normal normal form as in Proposition \ref{nfPn} of set partition in $P_n$, and $s\in S_n$ is the unique permutation defining $g$ obtained from the semi direct product above. See Figure \ref{figureoneeig} for an example.
\begin{figure}[H]
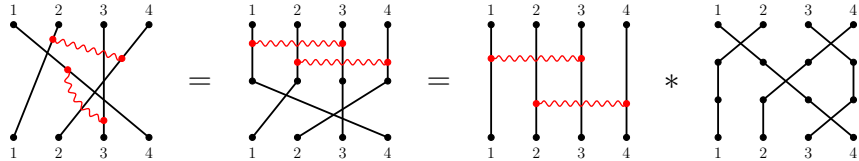

\figureoneeig
\caption{Decomposition of $e_{1,3}e_{2,4}s_1s_3s_2s_4$.}
\label{figureoneeig}
\end{figure}

\begin{remark}
Since $\R(S_n)=P_n\rtimes S_n$, {\cite[Corollary 2]{La1998}} implies that $\R(S_n)$ can also be presented by generators $s_{i,j}$ satisfying (\ref{genSnRels}), and generators $e_{i,j}$ satisfying (\ref{ee1})--(\ref{ee3}), both subject to the following relation given by the action of $S_n$ on set partitions of $[n]$:\begin{align}\label{genActSnPn}s_{i,j}e_{h,k}=e_{\{s_{i,j}(h),s_{i,j}(k)\}}s_{i,j}\quad\text{for all}\quad i<j\quad\text{and}\quad h<k.\end{align}
\end{remark}

\section{Ramified symmetric inverse monoid}\label{RamSymInvMon}

This section realizes one of the main objectives of the article, providing a presentation of the ramified monoid of the symmetric inverse monoid, see Theorem \ref{PresRISn}. The proof of this theorem uses a normal form of $\R(\IS_n)$ and diagrammatic arguments, see Corollary \ref{NFormISn} and Remark \ref{lemNormRISn}.

\subsection{}\label{sec4-1}

Due to Proposition \ref{facRam} and to the fact that $\IS_n=\I_nS_n$, we have $\R(\IS_n)=\R(\I_n)S_n$.

Since the number of blocks of a set partitions corresponding to an element with $k$ lines is $2n-k$, we have\[|\R(\IS_n)|=\sum_{k=0}^n k!\binom{n}{k}^2 b_{2n-k}.\]For $\IS_n^{\bullet}$, the number of lines is at most $n-1$, so, the number of blocks in the case of $k$ lines is evidently $2n-1-k$, hence\[|\R(\IS_n^{\bullet})|=\sum_{k=0}^{n-1}k!\binom{n}{k}\binom{n-1}{k}b_{2n-1-k}.\]

Now, in order to give a set of generators of $\R(\IS_n)$, we need to introduce the following notations. For $g\in\R(IS_n)$ with $g=(I,J)$ and $I\preceq J$, we will denote by $g^*$ the unique element in $IS_n$ in which $\{i,j\}$ is a line of $g^*$ if either it is a line of $I$, or, for some block $B$ of $J$ that contains no lines of $I$, $i=\min(B\cap[n])$ and $j=\max(B\cap[n+1,2n])$. Moreover, denote $g_p^*$ the completed permutation of the partial permutation $g^*$ determined by $g$. Finally, we denote by $q_i$ the element $(r_i,1)\in\R(\IS_n)$, which, according to our tied diagrammatic representation, is represented as in Figure \ref{figureonenin}.
\begin{figure}[H]
\figureonenin
\caption{Generator $q_i$ represented by a vertical tie.}
\label{figureonenin}
\end{figure}

\begin{proposition}\label{gensMonRISn}
$\R (\IS_n)$ is generated by $s_1,\ldots,s_{n-1}$, $r_1,\ldots,r_n$, $e_1,\ldots,e_{n-1}$, $q_1,\ldots,q_n$.
\end{proposition}
\begin{proof}
Let $g=(I,J)$ with $I\in\IS_n$ and $I\preceq J$. Then $g=e(I,g^*)e'$, where $e=(1,J\cap[n])$ and $e'$ is obtained from $(1,J\cap[n+1,2n])$ by removing the generators $e_{i,j}$ such that both $2n-i+1$ and $2n-j+1$ belong to lines of $I$. Consider $\{i_1\},\ldots,\{i_k\}$ the points of $I$ contained in $[n]$, and let $B_{i_j}$ be the block of $J$ containing $i_j$. Now, we define $\hat{r}_{i_j}=q_{i_j}$ if $i_j=\min(B_{i_j})$ and $B_{i_j}$ intersects $[n+1,2n]$ containing no lines of $I$. Otherwise, we set $\hat{r}_{i_j}$ as $(r_{i_j},r_{i_j})$. Thus, we have $g=e\hat{r}_{i_1}\cdots\hat{r}_{i_k}(g^*_p,g^*_p)e'$.
\end{proof}

The proposition above generalizes the normal form of $\IS_n$ described in Remark \ref{normalFormISn}. Thus, we obtain the following corollary.
\begin{corollary}[Normal form]\label{NFormISn}
The word $erg^*_pe'$ with $r=\hat{r}_{i_1}\cdots\hat{r}_{i_k}$ constructed during the proof of Proposition \ref{gensMonRISn} is a normal form in the free monoid $(\{r_i,q_i\mid i\in[n]\}\cup P_n\cup S_n)^*$.
\end{corollary}

For instance, consider $g=(I,J)$ with $I\preceq J$ defined as follows:\figureten Then, we obtain $g=erg^*_pe'$, where\figureele Therefore $g=e_{1,2}e_{4,5}\cdot r_2q_3r_5\cdot s_2s_3\cdot e_{1,5}e_{2,4}$ is represented as follows\figuretwe

\begin{remark}\label{remNFITS_n}
For every $g\in\R(\IS_n)$, the normal form $erg^*_pe'$, with $r=\hat{r}_{i_1}\cdots\hat{r}_{i_k}$, in Corollary \ref{NFormISn}, minimizes the number of vertical ties in $r$ by selecting at most one element $q_i$ for each block containing no lines of $I$, while it maximizes the number of horizontal lines on the ends $e$ and $e'$. As we mentioned above, this normal form generalizes the one for $\IS_n$ in Remark \ref{normalFormISn}, thus, as in Remark \ref{minCrosNormISn}, it minimizes the crossings in $g_p^*$. Note that it also generalizes the normal form of $\R(\Sym_n)$ given by the semi direct product shown in Remark \ref{tiedRamSn}.
\end{remark}

Let $\Omega_n$ be the monoid generated by $\sigma_1,\ldots,\sigma_{n-1}$, $\rho_1,\ldots,\rho_n$, $\e_1,\ldots,\e_{n-1}$, $\pi_1,\ldots,\pi_n$ satisfying the following relations:
\begin{align} 
\sigma_i^2=1,\quad
\sigma_i\sigma_j=\sigma_j\sigma_i&\text{ for }\vert i-j\vert>1,\quad
\sigma_i\sigma_j\sigma_i=\sigma_j\sigma_i\sigma_j\quad\text{for }\vert i-j\vert=1,\label{sgm}\\
\rho_i^2=\rho_i,&\quad
\rho_i\rho_j=\rho_j\rho_i\quad\text{for all }i,j,\label{rh1}\\
\sigma_i\rho_j&=\rho_{s_i(j)}\sigma_i\quad\text{for all }i,j,\label{rh2}\\
\rho_i\sigma_i\rho_i&=\rho_i\rho_{i+1}\quad\text{for }1\leq i\leq n-1,\label{rh3}\\
\e_i^2=\e_i,&\quad\e_i\e_j=\e_j\e_i\quad\text{for all }i\text{ and }j,\label{ep1}\\
\sigma_i\e_j=\e_j\sigma_i&\quad\text{if }|i-j|\neq1,\quad\e_i\sigma_j\sigma_i=\sigma_j\sigma_i\e_j\quad\text{if }|i-j|=1,\label{ep2}\\
\e_i\e_j\sigma_i=\e_j\sigma_i\e_j&=\sigma_i\e_i\e_j\quad\text{if }|i-j|=1,\label{ep3}\\
\pi_i^2&=\pi_i,\quad \pi_i\pi_j=\pi_j\pi_i,\label{qs1}\\
\pi_i\e_j&=\e_j\pi_i,\label{qs2}\\
\sigma_i\pi_j&=\pi_{s_i(j)}\sigma_i,\label{qs3}\\
\e_i\rho_j\e_i=\e_i\pi_j,&\quad\text{if }j=i,i+1,\quad\e_i\rho_j=\rho_j\e_i,\quad \text{if }j\neq i,i+1,\label{qis1}\\
\pi_i\rho_j&=\rho_j\pi_i,\quad\pi_i\rho_i=\rho_i,\label{qis2}\\
\rho_j\e_i\rho_j&=\rho_j,\quad j=i,i+1,\label{qis3}\\
\rho_i\e_i\rho_{i+1}&=\sigma_i\pi_i\rho_{i+1}.\label{qis4}
\end{align}

In what follows of this section, we denote by $\equiv$ the congruence generated by (\ref{sgm})--(\ref{qis4}).

\begin{lemma}\label{epiphi}
The mapping $\sigma_i\mapsto s_i$, $\rho_i\mapsto r_i$, $\e_i\mapsto e_i$, $\pi_i\mapsto q_i$ defines an epimorphism $\mu:\Omega_n\to\R(\IS_n)$.
\end{lemma}
\begin{proof}
The proof follows from the fact that the mapping respects the defining relations of $\Omega_n$ and Proposition \ref{gensMonRISn}.
\end{proof}

\begin{remark}\label{SubISnTSn}
Arguing as in \cite[Lemma 3.3]{KuMaCEJM2006}, we get that $IS_n$ and $TS_n$ are submonoids of $\Omega_n$. By Remark \ref{tiedRamSn}, the elements of $TS_n$ in $\Omega_n$ can be represented by $es$ and $se'$, where $s$ is a word in the letters $\sigma_i$'s, and $e,e'$ are words in the letters $\e_{i,j}:=\sigma_{j-1}\cdots\sigma_{i+1}\e_i\sigma_{i+1}\cdots\sigma_{j-1}$.
\end{remark}

\begin{remark}\label{lemNormRISn}
Let $g=(I,J)$ with $I\in\IS_n$ and $I\preceq J$, satisfying $g=erse'$ for some $e,e'\in P_n$, $r\in\langle r_1,\ldots,r_n,q_1,\ldots,q_n\rangle$ and $s\in\Sym_n$. Since $er$ and $e'$ have neither crossing lines nor crossing ties, $s$ must connect the same blocks of $J\cap[n]$ with the same blocks of $J\cap[n+1,2n]$. So, the unique possibilities that $erse'$ is not the normal form of $g$ are that the crossings of $s$ are not minimal or that, either the number of ties of $r$ is not minimal or the ties of $r$ are not located the most possible to the left. Thus, $erse'$ coincides with the normal form of $g$ if and only if none of the following properties hold:\begin{enumerate}
\item[(a)]$s=s_i\bar{s}$ such that $r$ involves $r_i,r_{i+1}$.
\item[(b)]$s=s_i\bar{s}$ such that $r$ involves $r_i,q_{i+1}$ and $e'$ involves $e_{\bar{s}(i),\bar{s}(i+1)}$.
\item[(c)]$s=s_i\bar{s}$ such that $r$ involves $q_i,r_{i+1}$ and $e'$ involves $e_{\bar{s}(i),\bar{s}(i+1)}$.
\item[(d)]$s=s_i\bar{s}$ such that $r$ involves $q_i,q_{i+1}$ and $e_i$ occurs in $e$ or  $e_{\bar{s}(i),\bar{s}(i+1)}$ occurs in $e'$.
\item[(e)]$r$ involves $q_i$ and $e$ involves $e_{\{i,j\}}$ such that $\{j,s(j)\}$ is a line of $I$.
\item[(f)]$r$ involves $q_i$ and $e'$ involves $e_{\{s(i),j\}}$ such that $\{j,s(j)\}$ is a line of $I$.
\item[(g)]$r$ involves $q_i,q_j$ such that $e_{i,j}$ occurs in $e$ or $e_{s(i),s(j)}$ occurs in $e'$.
\end{enumerate}See Figure \ref{figureeig} for examples of ramified partitions satisfying these properties.
\end{remark}
\begin{figure}[H]
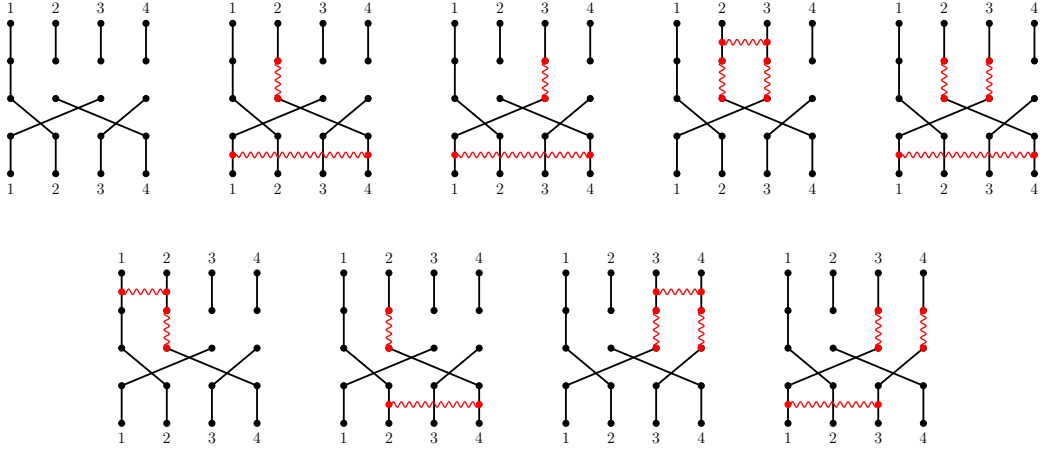

\figureeig
\caption{Ramified partitions satisfying conditions in Remark \ref{lemNormRISn}. The first three correspond to elements of types (a), (b) and (c) respectively, the fourth and fifth correspond to elements of type (d), the sixth and seventh correspond to elements of types (e) and (f) respectively, and the last two correspond to elements of type (g).}
\label{figureeig}
\end{figure}

\begin{lemma}\label{remNormExceptions}
The following relations hold in $\Omega_n$:
\begin{align}
\rho_i\rho_{i+1}\sigma_i&=\rho_i\rho_{i+1},\label{su1}\\
\pi_i\pi_{i+1}\sigma_i\e_i&=\pi_i\pi_{i+1}\e_i,\label{su2}\\
\rho_i\pi_{i+1}\sigma_i\e_i&=\rho_i\pi_{i+1}\e_i,\label{su3}\\
\pi_i\rho_{i+1}\sigma_i\e_i&=\pi_i\rho_{i+1}\e_i,\label{su4}\\
\e_i\pi_i\pi_{i+1}&=\e_i\pi_i\rho_{i+1}\e_i\label{su5}
\end{align}
Moreover, (\ref{su1})--(\ref{su4}) can be generalized for each pair $i,j\in[n]$ with $i<j$, by using $\sigma_{i,j}:=\sigma_i\cdots\sigma_{j-1}\cdots\sigma_i$ instead of $\sigma_i$.
\end{lemma}
\begin{proof}
Due to (\ref{rh2}) and (\ref{rh3}), we have $\rho_i\rho_{i+1}\sigma_i=\rho_is_i\rho_i=\rho_i\rho_{i+1}$, which proves (\ref{su1}). To show (\ref{su2}) note that we can obtain an analogous of (\ref{qis1}) as follows\begin{equation}\label{qis11}\e_i\rho_{i+1}\e_i\stackrel{(\ref{qis2}),(\ref{qs3})}{=}\e_i\rho_{i+1}\sigma_i\pi_i\sigma_i\e_i\stackrel{(\ref{rh2}),(\ref{ep2}),(\ref{qis2})}{=}\sigma_i\e_i\rho_i\e_i\sigma_i\stackrel{(\ref{qis1})}{=}\sigma_i\pi_i\e_i\sigma_i\stackrel{(\ref{rh1}),(\ref{ep2}),(\ref{qs3})}{=}\pi_{i+1}\e_i.\end{equation}Now, by using this relation, we obtain (\ref{su2}), indeed\[\pi_i\pi_{i+1}\sigma_i\e_i\stackrel{(\ref{qs1}),(\ref{qs3})}{=}\pi_{i+1}\sigma_i\pi_{i+1}\e_i\stackrel{(\ref{qis11}),(\ref{qs3}),(\ref{qs2}),(\ref{ep2})}{=}\e_i\sigma_i\pi_i\rho_{i+1}\e_i\stackrel{(\ref{qis4})}{=}\e_i\rho_i\e_i\rho_{i+1}\e_i\stackrel{(\ref{qis1}),(\ref{qis11})}{=}\pi_i\pi_{i+1}\e_i.\]Relation (\ref{qis2}) implies $\rho_i\pi_{i+1}\sigma_i\e_i=\rho_i\pi_i\pi_{i+1}\sigma_i\e_i$ and $\pi_i\rho_{i+1}\sigma_i\e_i=\rho_{i+1}\pi_i\rho_{i+1}\sigma_i\e_i$. Thus, by applying (\ref{su2}) and (\ref{qis2}), we obtain (\ref{su3}) and (\ref{su4}). On the other hand, by using (\ref{ep1}), (\ref{qs2}) and (\ref{qis11}), we get $\e_i\pi_i\pi_{i+1}=\e_i\pi_i\pi_{i+1}\e_i=\e_i\pi_i\e_i\rho_{i+1}\e_i=\e_i\pi_i\rho_{i+1}\e_i$, which proves (\ref{su5}).
\end{proof}

Relations (\ref{rh3}), (\ref{ep2}) and (\ref{ep3}) can be 
generalized, respectively, as follows:\begin{align}
\e_{i,j}\rho_i\e_{i,j}=\pi_i\e_{i,j}=\e_{i,j}\pi_i,&\qquad\e_{i,j}\rho_j\e_{i,j}=\pi_j\e_{i,j}=\e_{i,j}\pi_j,\label{NR1}\\
\rho_i\e_{i,j}\rho_i=\rho_i,&\qquad\rho_j\e_{i,j}\rho_j=\rho_j,\label{NR2}\\
\rho_i\e_{i,j}\rho_j=\pi_j\rho_i\sigma_{i,j}=\sigma_{i,j}\pi_i\rho_j,&\qquad\rho_j\e_{i,j}\rho_i=\pi_i\rho_i\sigma_{i,j}=\sigma_{i,j}\pi_j\rho_i,\label{NR3}
\end{align}

\begin{proposition}\label{normalOmega}
Every element $\gamma\in\Omega_n$ can be represented by a word that coincides with the normal form of $\mu(\gamma)$ when replacing the letters $\sigma_i,\rho_i,\e_i,\pi_i$, respectively, by $s_i,r_i,e_i,q_i$.
\end{proposition}
\begin{proof}
Let $\gamma\in\Omega_n$, and let $w$ be a word representative of it. Note that if none of the $\pi_i$'s occur in $w$, 
%and either none of the $\rho_i$'s or none of the $\e_i$'s occur in $w$, % OLD
and neither one of the $\rho_i$'s nor one of the $\e_i$'s occurs in $w$, % NEW
then $\mu(\gamma)\in IS_n\cup TS_n$, so the result is clear due to Remark \ref{remNFITS_n} and Remark \ref{SubISnTSn}.

Relations (\ref{rh2}) and (\ref{qs2})--(\ref{qis2}) imply that $w\equiv t'u'tv'$, where $t$ and $t'$ represent elements of $TS_n$, and $u'$ is a word in the letters $\rho_i$'s and $\pi_i$'s. Remark \ref{SubISnTSn} implies that $t'\equiv e's''$ and $t\equiv s'e$, where $s',s''\in S_n$ and $e,e'$ are words in the letters $\e_{i,j}$. Furthermore, either $v'=1$ or $v'=\rho_kv$ and $e=e''\e_{\{h,k\}}$ for some $h\in[n]$. Thus $w\equiv e's''u's'ev'$. We treat firstly the case $v'=1$ and after $v'\neq1$.

If $v'=1$, due to (\ref{rh2}) and (\ref{qs3}), we have $w\equiv e'use=:w'$, where $s:=s''s'$ and $u:=s''^{-1}(u')$ is still a word in the $\rho_i$'s and $\pi_i$'s. Thus, the word obtained by replacing the letters in $w'$ satisfies the conditions of Remark \ref{lemNormRISn}. We will construct an equivalent word by repeating the following steps that distinguish the seven cases in Remark \ref{lemNormRISn}. 

If $s=\sigma_{i,j}\bar{s}$ such that $w'$ is of types (a)--(d), $\sigma_{i,j}$ can be removed by using (\ref{su1})--(\ref{su4}) and their generalizations in Lemma \ref{remNormExceptions}. See Figure \ref{figurethisix}.
\begin{figure}[H]
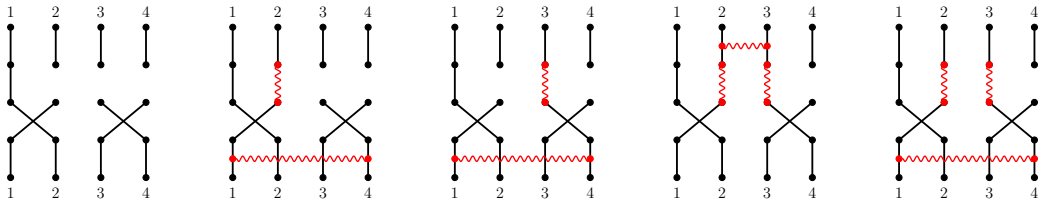

\figurethisix
\caption{Elements obtained when removing $\sigma_2$ from the ones of types (a)--(d) in Figure \ref{figureeig}}
\label{figurethisix}
\end{figure}
If $w'$ is of types (e)--(g) we will use (\ref{qs2}), (\ref{qis1}), and (\ref{su5}) in Lemma \ref{remNormExceptions}. See Figure \ref{figurenin}. Indeed, if $w'$ is of type (e), $\pi_i$ in $u$ is replaced by $\rho_i$ and $e$ is replaced by $e\e_{\{s(i),s(j)\}}$ for some $j$. If $w'$ is of type (f), $\pi_i$ in $u$ is replaced by $\rho_i$ and $e'$ is replaced by $\e_{i,s^{-1}(j)}e'$ for some $j$. If $w'$ is of type (g), $\pi_j$ is replaced by $\rho_j$, $e'$ is replaced by $\e_{i,j}e'$ if $\e_{i,j}$ does not occur in $e'$, and $e$ is replaced by $e\e_{s(i),s(j)}$ if $\e_{s(i),s(j)}$ does not occur in $e$.
\begin{figure}[H]
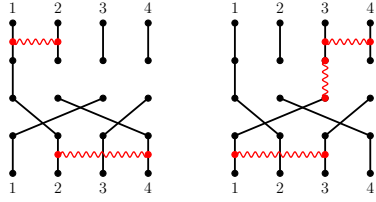

\figurenin
\caption{Elements obtained from the ones of types (e)--(g) in Figure \ref{figureeig}}
\label{figurenin}
\end{figure}
Now, by using (\ref{qis2}), we replace $u$ by a resorted word $\hat{\rho}_1\cdots\hat{\rho}_k$ with $\hat{\rho}_i\in\{\rho_i,\pi_i\}$. Thus, we obtain a normal form that is equivalent to $w$.

If $v'\neq1$, without loss of generality, we will assume that none of the $\e_{p,q}$'s in $e$ can be moved next to $t'$. Thus, if $\e_{p,q}$ occurs in $e$, then $u'$ must contain at least one element of $\{\rho_i,\rho_j\}$, where $i=s'^{-1}(p)$ and $j=s'^{-1}(q)$, otherwise $\e_{p,q}$ can be moved next to $t'$. In particular, the element $\e_{\{h,k\}}$ occurring in $e$ cannot be moved next to $t'$, hence $u'$ contains $\rho_i$ or $\rho_j$ with $s'(i)=h$ and $s'(j)=k$. If $u'\equiv u''\rho_i$, then, due to (\ref{NR3}), we obtain:\[w\equiv t'u''\rho_is'e''\e_{\{h,k\}}\rho_kv\equiv t'u''s'(\rho_h\e_{\{h,k\}}\rho_k)e''v\equiv t'u''s'(\pi_k\rho_h\sigma_{h,k})e''v\equiv t'(u''\pi_j\rho_i)(s'\sigma_{h,k}e'')v.\]Now, if $u'=u''\rho_j$, relation (\ref{NR2}) implies the following:\[w\equiv t'u''\rho_js'e''\e_{\{h,k\}}\rho_kv\equiv t'u''s'(\rho_k\e_{\{h,k\}}\rho_k)e''v\equiv t'u''s'\rho_ke''v\equiv t'(u''\rho_j)(s'e'')v.\]If $v\neq1$, we repeat the process, and so on until we get a word as the previous case.
\end{proof}

\begin{remark}\label{remAboutNFTorsion}
Note that neither the proof of Lemma \ref{remNormExceptions} nor the proof of Proposition \ref{normalOmega} use the torsion of $\Omega_n$ given at the beginning of relation (\ref{sgm}). This implies that the monoid $\Omega_n^+$ obtained by removing the torsion of $\Omega_n$, also satisfies Lemma \ref{epiphi} and Proposition \ref{normalOmega}.
\end{remark}

\begin{theorem}\label{PresRISn}
The monoids $\R(\IS_n)$ and $\Omega_n$ are isomorphic. Thus, $\R(\IS_n)$ is presented by generators $s_1,\ldots,s_{n-1}$, $r_1,\ldots,r_n$, $e_1,\ldots,e_{n-1}$, $q_1,\ldots,q_n$ satisfying {\normalfont(\ref{s})--(\ref{r3})} and {\normalfont(\ref{TSn1})--(\ref{TSn3})}, subject to the following relations:
\begin{align} 
q_i^2=q_i&,\quad q_iq_j=q_jq_i,\label{ris1}\\
q_ie_j&=e_jq_i,\label{ris2}\\
s_iq_j&=q_{s_i(j)}s_i,\label{ris3}\\
e_ir_je_i=e_iq_j,\quad\text{if }j=i,i+1&,\quad e_ir_j=r_je_i,\quad\text{if }j\neq i,i+1,\label{ris4}\\
r_iq_j=q_jr_i&,\quad q_ir_i=r_i,\label{ris5}\\
r_je_ir_j&=r_j,\quad\text{if }j=i,i+1,\label{ris6}\\
r_ie_ir_{i+1}&=s_iq_ir_{i+1}.\label{ris7}
\end{align}
\end{theorem}
\begin{proof}
Proposition \ref{normalOmega} implies that two elements of $\Omega_n$ that are sent by $\mu$ to the same element of $\R(\IS_n)$, are represented by the same word. Thus, $\mu$ is an isomorphism.
\end{proof}

\begin{remark}\label{dualPresentationRISn}
By applying Tietze's transformations, the monoid $\R(\IS_n)$ can be presented by generators $s_{i,j}$ satisfying (\ref{genSnRels}), generators $e_{i,j}$ satisfying (\ref{ee1})--(\ref{ee3}) and (\ref{genActSnPn}), generators $r_1,\ldots,r_n$ satisfying (\ref{r1}) and (\ref{gr23}), and generators $q_1,\ldots,q_n$ satisfying (\ref{ris1}) and (\ref{ris5}), subject to the following relations:\begin{align}
e_{i,j}q_k&=q_ke_{i,j},\\
s_{i,j}q_k&=q_{s_{i,j(k)}}s_{i,j},\\
e_{i,j}r_ie_{i,j}&=q_ie_{i,j},\\
e_{i,j}r_k&=r_ke_{i,j},\quad k\neq i,j,\\
r_ie_{i,j}r_i&=r_i,\\
r_ie_{i,j}r_j&=s_{i,j}q_ir_j.
\end{align}
\end{remark}

\section{Planar monoids}\label{PlanMons}

In this final section we find a presentation of the ramified monoid of $\I_n$ (Theorem \ref{presRIn}) defined in (\ref{InPres}) and introduce the concept of planar ramified monoid. This concept follows that of planar monoid defined in \cite{Jo1983}. In Theorem \ref{presPRIn} we get a presentation of the planar ramified monoid $\PR(\I_n)$ of $\I_n$. Besides the cardinality of $\mathcal{PR}(\I_n)$ is computed recursively, see Proposition 8. 

\subsection{}

A set partition of $\CC_n$ is said to be {\it planar} \cite{Jo1993,HaRa2005} if it can be represented by a diagram with noncrossing generalized lines. As concatenation preserves planarity, for every submonoid $M$ of $\CC_n$, the set of planar set partitions of $M$ forms a submonoid of it, which is called the {\it planar monoid} of $M$ and is denoted by $\P(M)$. In particular, $\P(\CC_n)$ is known as the {\it planar partition monoid} and is usually denoted by $\PP_n$.

A remarkable submonoid of $\PP_n$ is the {\it Jones monoid} $\J_n$ \cite{Jo1983,LaFi2006}, which is formed by the planar set partitions of $\CC_n$ whose blocks have exactly two elements. This monoid is presented by the {\it tangle generators} $t_1,\ldots,t_{n-1}$, subject to the following relations:\begin{equation}\label{JnRels}
t_i^2=t_i,\quad t_it_j=t_jt_i\quad\text{if }|i-j|>1,\quad t_it_jt_i=t_i\quad\text{if }|i-j|=1.
\end{equation}

For each $i\in[n-1]$, the tangle $t_i$ can be realised as the set partition formed by the blocks $\{i,i+1\}$, $\{2n-i+1,2n-i\}$, and $\{k,2n-k+1\}$ for all $k\in[n]\backslash\{i,i+1\}$, see Figure \ref{figuretwenty}.
\begin{figure}[H]
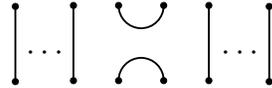

\figuretwenty
\caption{Generator $t_i$.}
\label{figuretwenty}
\end{figure}

The cardinality of $\J_n$ is the $n$th Catalan number $c_n$ \cite[A000108]{OEIS}. Note that the Jones monoid is the planar monoid of the Brauer monoid \cite[p. 416]{KuMaCEJM2006}. Beside, it is well known that $\PP_n$ is isomorphic to $\J_{2n}$, see for example \cite[p. 873]{HaRa2005}, hence $|\PP_n|=c_{2n}$. This isomorphism can be graphically shown as the example in Figure \ref{ExIsomPPnJ2n}, where the dashed lines can be removed.
\begin{figure}[H]
\figuretwotwo
\caption{}
\label{ExIsomPPnJ2n}
\end{figure}

For $i\in[n-1]$, let $h_i$ be the set partition with blocks $\{i,i+1,2n-i+1,2n-i\}$ and $\{k,2n-k+1\}$ for all $k\not\in\{i,i+1\}$. This set partition will be a generator of $\PP_n$ and is represented as in Figure \ref{figuretwothr}.
\begin{figure}[H]
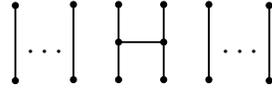

\figuretwothr
\caption{Generator $h_i$.}
\label{figuretwothr}
\end{figure}

The partitions $h_i$ and $r_i$ generate $\PP_n$. In fact, due to \cite[Theorem 1.11.(b)]{HaRa2005}, the isomorphism above from $\J_{2n}$ to $\PP_n$ is made explicit by the mapping\begin{align}
t_{2i}&\mapsto h_i,\\t_{2i-1}&\mapsto r_i,\end{align}so that the monoid $\PP_n$ can be presented by generators $r_1,\ldots,r_n$, $h_1,\ldots,h_{n-1}$ subject to the following relations, coming from relations (\ref{JnRels}):\begin{align}
r_i^2&=r_i,\quad r_ir_j=r_jr_i\quad\text{for all }i,j,\label{BnT1}\\
h_i^2&=h_i,\quad h_ih_j=h_jh_i\quad\text{for all }i,j,\textbf{\label{BnT2}}\\
h_ir_j&=r_jh_i\quad\text{if }|i-j|>1,\label{BnT3}\\
h_ir_ih_i&=h_i=h_ir_{i+1}h_i\quad\text{for all }i,\label{BnT4}\\
r_ih_ir_i&=r_i,\quad r_{i+1}h_ir_{i+1}=r_{i+1}\quad\text{for all }i.\label{BnT5}
\end{align}
See Figure \ref{figuretwofiv} for an example.
\begin{figure}[H]
\figuretwofiv
\caption{}
\label{figuretwofiv}
\end{figure}

\begin{definition}
For a planar submonoid $M$ of $\CC_n$, the {\it planar ramified monoid} $\PR(M)$ of $M$ is the monoid formed by the pairs $(I,J)$ such that $I\in M$ and $J\in\PP_n$ with $I\preceq J$.
\end{definition}

\subsection{Monoids of noncrossing partitions}

Recall that $\I_n$ is the submonoid of $\IS_n$ generated by $r_1,\ldots,r_n$, which is isomorphic to the free idempotent commutative monoid, presented as in (\ref{InPres}). Note that $\I_n$ corresponds to a collection of noncrossing partitions of $[2n]$, and elements of $\I_n^{\bullet}$ are noncrossing partitions of $[2n-1]$. All these elements will be called {\it$r$--partitions}. In what follows we study the ramified and planar ramified monoids of $r$--partitions.

\subsubsection{The monoids $\R(\I_n)$ and $\R(\I_n^{\bullet})$}\label{5.2.1}

The cardinalities of $\R(\I_n)$ and $\R(\I_n^{\bullet})$ are given by the following formulas:\[|\R(\I_n)|=\sum_{k=0}^n\binom{n}{k}b_{2n-k},\qquad|\R(\I_n^{\bullet})|=\sum_{k=0}^{n-1}\binom{n-1}{k}b_{2n-1-k},\]because, if a set partition $I$ of $[m]$, with $m\in\{2n-1,2n\}$, has $k$ lines, the number of blocks of $I$ is $m-k$.

\begin{lemma}\label{fixLinePerm}
Every nontrivial permutation $s\in\Sym_n$ has a word representative written in letters $s_{i,j}$ such that $s(i)\neq i$ and $s(j)\neq j$, that is, vertical lines keep vertical.
\end{lemma}
\begin{proof}
Let $a_1<\cdots<a_m$ be the elements of $[n]$ satisfying $s(a_i)\neq a_i$ for all $i\in[m]$, and let $t\in S_m$ defined by $t(i)=j$ whenever $s(a_i)=a_j$. Thus, every word representative of $t$ defines a word representative of $s$ by replacing each letter $s_i$ of $t$ by $s_{a_i,a_{i+1}}$.
\end{proof}

In particular, Lemma \ref{fixLinePerm} can be applied to the completed permutation of partial permutations. For instance:\figurethifiv{So, $t=s_2s_1$, $s=s_{4,6}s_{1,4}$, involving no indices $3$ and $5$, and $g=r_2\cdot s_{4,6}s_{1,4}$ is represented as follows}

Now, for $i,j$ with $i<j$, denote by $x_{i,j}\in\R(\IS_n)$ the ramified partition $(r_ir_j,s_{i,j})$, which is represented as in Figure \ref{figuretwosev}.
\begin{figure}[H]
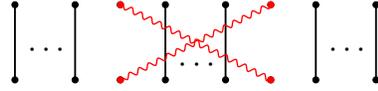

\figuretwosev
\caption{Generator $x_{i,j}$.}
\label{figuretwosev}
\end{figure}
Notice that $(r_ir_j,s_{i,j})=(r_i,1)(r_j,1)(1,s_{i,j})=q_iq_js_{i,j}$, thus, for each $i<j$, we have\begin{equation}\label{remVerTies}q_i x_{i,j}=x_{i,j}=x_{i,j}q_i\quad\text{and}\quad r_ir_jx_{i,j}=r_ir_js_{i,j}=s_{i,j}r_jr_i=x_{i,j}r_jr_i=r_jr_i.\end{equation}

\begin{proposition}\label{gensMonRIn}
The monoid $\R(\I_n)$ is generated by $r_1,\ldots,r_n$ and $q_1,\ldots,q_n$ together with $e_{i,j}$ and $x_{i,j}$ such that $i,j\in[n]$ and $i<j$.
\end{proposition}
\begin{proof}
Corollary \ref{NFormISn} implies that every element $g=(I,J)\in\R(\I_n)$ can be uniquely written in $\R(\IS_n)$ as $g=erg_p^*e'$. Further, Lemma \ref{fixLinePerm} implies that $g_p^*$ has a word representative $u$ in letters $s_{i,j}$ such that $\{i\}$ and $\{j\}$ are points of $I$. Thus, $g=erx_ge'$, where $x_g$ is the element obtained by replacing each letter $s_{i,j}$ by $x_{i,j}$ in the word $u$ representing $g_p^*$. Finally, $g=eqx_ge'$, where $q$ is obtained by removing the ties in $r$ connecting ties in $x_g$.
\end{proof}

\begin{corollary}[Normal form]\label{NFormRInO}
The word $eqx_ge'$ constructed during the proof of Proposition \ref{gensMonRIn} is a normal form in the free monoid $(\{r_i,q_i\mid i\in[n]\}\cup P_n\cup\langle x_{i,j}\mid i<j\rangle)^*$.
\end{corollary}

For instance, consider $g=(I,J)$ with $I\preceq J$ defined as follows:\figurethirty Then, we obtain $g=erx_ge'$, where\figurethione Hence, $g=e_{1,3}\cdot q_2r_3q_4q_5\cdot x_{2,4}\cdot e_{2,3}=e_{1,3}\cdot r_3q_5\cdot x_{2,4}\cdot e_{2,3}$, which is represented as follows\figurethitwo

\begin{remark}
Due to Lemma \ref{fixLinePerm}, the element $x_g$ in Corollary \ref{NFormRInO} can be obtained diagrammatically just by replacing every non vertical line of $g_p^*$ by a tie. Thus, by Remark \ref{minCrosNormISn}, the crossings of $x_g$ are minimal.
\end{remark}

Let $\Upsilon_n$ be the monoid generated by $\rho_1,\ldots,\rho_n$, $\pi_1,\ldots,\pi_n$, and $\e_{i,j}$, $\chi_{i,j}$ with $i,j\in[n]$ and $i<j$, satisfying (\ref{rh1}), (\ref{qs1}), (\ref{qis2}), subject to the following relations:\begin{align}
\e_{i,j}^2=\e_{i,j},\quad\e_{i,j}\e_{r,s}&=\e_{r,s}\e_{i,j},\quad\e_{i,j}\e_{i,k}=\e_{i,j}\e_{j,k}=\e_{i,k}\e_{j,k},\label{ups04}\\
\e_{i,j}\rho_k&=\rho_k\e_{i,j},\quad k\not\in\{i,j\},\label{ups05}\\
\rho_i\e_{i,j}\rho_i&=\rho_i,\quad\rho_j\e_{i,j}\rho_j=\rho_j,\label{ups06}\\
\e_{i,j}\rho_i\e_{i,j}&=\pi_i\e_{i,j},\quad\e_{i,j}\rho_j\e_{i,j}=\pi_j\e_{i,j}\label{ups07}\\
\e_{i,j}\pi_k&=\pi_k\e_{i,j},\label{ups08}\\
\chi_{i,j}^2&=\pi_i\pi_j,\quad\chi_{i,j}\chi_{j,k}=\chi_{j,k}\chi_{i,k}=\chi_{i,k}\chi_{i,j},\label{ups09}\\
\chi_{i,j}\chi_{a,b}&=\chi_{a,b}\chi_{i,j},\quad[i,j],[a,b]\text{ disjoint or nested},\label{ups10}\\
\chi_{i,j}\rho_k&=\rho_{s(k)}\chi_{i,j},\quad s:=(i\,\,j),\label{ups11}\\
\chi_{i,j}\pi_k&=\pi_k\chi_{i,j},\quad\chi_{i,j}\pi_i=\chi_{i,j}=\chi_{i,j}\pi_j,\label{ups12}\\
\chi_{i,j}\e_{p,q}&=\e_{s(p),s(q)}\chi_{i,j},\quad s:=(i\,\,j),\label{ups13}\\
\chi_{i,j}\e_{i,j}&=\e_{i,j}\pi_i\pi_j.\label{ups14}
\end{align}
In what follows of this section, we denote by $\equiv$ the congruence generated by the relations that define $\Upsilon_n$.

\begin{lemma}
The mapping $\rho_i\mapsto r_i$, $\pi_i\mapsto q_i$, $\e_i\mapsto e_i$, $\chi_{i,j}\mapsto x_{i,j}$ defines an epimorphism $\eta:\Upsilon_n\to\R(\I_n)$.
\end{lemma}
\begin{proof}
The proof follows from the fact that the mapping respects the defining relations of $\Omega_n$ and Proposition \ref{gensMonRIn}.
\end{proof}

\begin{lemma}\label{moreRelsRIn}
The relations $\rho_i\chi_{i,j}\rho_i=\rho_i\rho_j$ and $\rho_i\e_{i,j}\rho_i=\chi_{i,j}\pi_i\rho_j$ hold in $\Upsilon_n$.
\end{lemma}
\begin{proof}
Due to (\ref{ups06}), we have $\rho_i\chi_{i,j}\rho_i=\rho_i\chi_{i,j}\rho_i\e_{i,j}\rho_i$, hence\[\rho_i\chi_{i,j}\rho_i\stackrel{(\ref{ups11})}{=}\rho_i\rho_j\chi_{i,j}\e_{i,j}\rho_i\stackrel{(\ref{ups14})}{=}\rho_i\rho_j\e_{i,j}\pi_i\pi_j\rho_i\stackrel{(\ref{rh1}),(\ref{qis2}),(\ref{ups06})}{=}\rho_j\rho_i\pi_i\pi_j\stackrel{(\ref{rh1}),(\ref{qis2})}{=}\rho_i\rho_j.\]Similarly, due to (\ref{qis2}) and (\ref{ups06}), we have $\chi_{i,j}\pi_i\rho_j=\chi_{i,j}\rho_j\pi_i=\chi_{i,j}\rho_j\e_{i,j}\rho_j\pi_i$, hence\[\chi_{i,j}\pi_i\rho_j\stackrel{(\ref{ups11})}{=}\rho_i\chi_{i,j}\e_{i,j}\rho_j\pi_i\stackrel{(\ref{ups14})}{=}\rho_i\e_{i,j}\pi_i\pi_j\rho_j\pi_i\stackrel{(\ref{qs1}),(\ref{qis2})}{=}\rho_i\e_{i,j}\pi_i\rho_j\stackrel{(\ref{ups08}),(\ref{qis2})}{=}\rho_i\e_{i,j}\rho_j.\]
\end{proof}

\begin{proposition}\label{normalUpsilon}
Every element $g\in\Upsilon_n$ can be represented by a word that coincides with the normal form of $\eta(g)$ when replacing the letters $\rho_i,\pi_i,\e_i,\chi_{i,j}$ by $r_i,q_i,e_i,x_{i,j}$ respectively.
\end{proposition}
\begin{proof}
Remark \ref{dualPresentationRISn} and Lemma \ref{moreRelsRIn} imply that the map $\mu(\Omega_n^+)\to\Upsilon_n$ sending $r_i\mapsto\rho_i$, $q_i\mapsto\pi_i$, $e_{i,j}\mapsto\e_{i,j}$ and $s_{i,j}\mapsto\chi_{i,j}$ is an epimorphism because each relation of $\mu(\Omega_n^+)$ holds in $\Upsilon_n$ under it. Thus, $\Upsilon_n$ is a quotient of $\mu(\Omega_n^+)$. So, by Remark \ref{remAboutNFTorsion}, $g=eqxe'$ such that $\mu(e)\mu(u)\mu(x)\mu(e')$ is the normal form of $\mu(g)$ in $\R(\IS_n)$. So, by definition of normal forms for elements of $\R(\I_n)$ in Corollary \ref{NFormRInO}, we have $\eta(x)=x_p$, and the unique possibility that $eqxe'$ is not a normal form when replacing the letters is that, by using (\ref{ups12}), some generators $\pi_i$ can be removed from $q$.
\end{proof}

\begin{theorem}\label{presRIn}
The monoid $\R(\I_n)$ is presented by generators $r_1,\ldots,r_n$ satisfying (\ref{r1}), $q_1,\ldots,q_n$ satisfying (\ref{ris1}) and (\ref{ris5}), $e_{i,j}$ with $i,j\in[n]$ and $i<j$ satisfying (\ref{ee1})--(\ref{ee3}), and $x_{i,j}$ with $i,j\in[n]$ and $i<j$, subject to the following relations:\begin{align}
e_{i,j}r_k&=r_ke_{i,j},\quad k\not\in\{i,j\},\\
r_ie_{i,j}r_i&=r_i,\quad r_je_{i,j}r_j=r_j,\\
e_{i,j}r_ie_{i,j}&=e_{i,j}q_i,\quad e_{i,j}r_je_{i,j}=e_{i,j}q_j\\
e_{i,j}q_k&=q_ke_{i,j},\\
x_{i,j}^2&=q_iq_j,\quad x_{i,j}x_{j,k}=x_{j,k}x_{i,k}=x_{i,k}x_{i,j},\\
x_{i,j}x_{a,b}&=x_{a,b}x_{i,j},\quad[i,j],[a,b]\text{ disjoint or nested},\\
x_{i,j}r_k&=r_{s(k)}x_{i,j},\quad s:=(i\,\,j),\\
x_{i,j}q_k&=q_kx_{i,j},\quad x_{i,j}q_i=x_{i,j}=x_{i,j}q_j,\\
x_{i,j}e_{p,q}&=e_{s(p),s(q)}x_{i,j},\quad s:=(i\,\,j),\\
x_{i,j}e_{i,j}&=e_{i,j}q_iq_j.
\end{align}
\end{theorem}
\begin{proof}
Proposition \ref{normalUpsilon} implies that two elements of $\Upsilon_n$ that are sent by $\eta$ to the same element of $\R(\I_n)$, are represented by the same word. Thus, $\eta$ is an isomorphism.
\end{proof}

\begin{remark}\label{remIe=1}
Note that $\I_n^{\bullet}$ is the submonoid of $\CC_n^{\bullet}$ generated by $r_1,\ldots,r_{n-1}$. Clearly, $\I_n^{\bullet}$ is isomorphic to $\I_{n-1}$, however their ramified monoids do not coincide. By proceeding as in Proposition \ref{gensMonRIn}, it is possible to prove that $\R(\I_n^{\bullet})$ is generated by the same generators as $\R(\I_{n-1})$ together with $e_{i,n}$ for all $i\in[n-1]$. Moreover, elements of $\R(\I_n^{\bullet})$ have normal forms as in Corollary \ref{NFormRInO}, in which the generator $e_{i,n}$ may occur in $e$ or $e'$.
\end{remark}

\subsubsection{The monoids $\PR(\I_n)$ and $\PR(\I_n^{\bullet})$}\label{5.2.2}

For $m\in\N$ and $J\in P_m$, we denote by $\T(J)$ the collection of set partitions of $[m+1]$ obtained either by adding $\{m+1\}$ to $J$ or by adding $m+1$ to some block of $J$. Note that if $J$ is planar, the elements of $\T(J)$ are not necessarily planar. See Figure \ref{examT(J)}.
\begin{figure}[H]
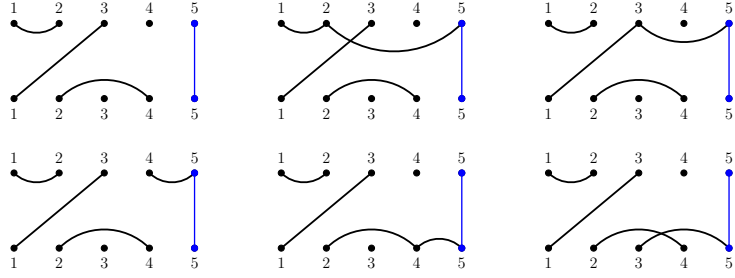

\figureforty
\caption{Elements of $\T(J)$ for $J=(\{1,2\},\{3,8\},\{4\},\{5,7\},\{6\})$ in $P_8$, containing $4$ planar set partitions of $\CC_5^{\bullet}$.}
\label{examT(J)}
\end{figure}

Observe now that if $\T(J)$ has $h$ planar elements, then $J$ has exactly $h-1$ blocks that can be connected to the point ${m+1}$ to get a planar element. Such number of connectable blocks of $J$ is denoted by $c(J)$.

For each $k\in[m]$, we denote by $\T_{m,k}$ the collection of ramified partitions $(I,J)$ such that $I$ is an $r$--partition of $[m]$ and $c(J)=k$. Set $T(m,k)=|\T_{m,k}|$. So,\begin{equation}|\PR(\I_n)|=\sum_{k=1}^{2n}T(2n,k)\quad\text{and}\quad|\PR(\I_n^\bullet)|=\sum_{k=1}^{2n-1}T(2n-1,k).\end{equation}

\begin{proposition}\label{numT(n,m)}
$T(1,1)=1$, and for integers $m>1$ and $k\in[m]$, we have\[T(m,k)=\sum_{j=1}^{m-1}\tau_{j,k}T(m-1,j),\]where\[\tau_{j,k}=\left\{\begin{array}{ll}1&\text{if }k>1\text{ and }j\in[k-1,m-1],\\0&\text{if }k>1\text{ and }j\not\in[k-1,m-1],\\1+e&\text{if }k=1,\text{ with }e=(m-1)\kern-0.7em\mod2\text{ and }j\in[m-1].\end{array}\right.\]

\end{proposition}
\begin{proof}
If $m=2n-1$, for every element $(I,J)$ of $\T_{m-1,j}$ we obtain exactly one element $(I',J')$ of $\T_{m,k}$ for all $k\in[j+1]$, see Figure \ref{figNumT(n,m)A}. Indeed, observe that $I'$ is the same as $I$ with an extra point, $\{n\}$, at right. Since $c(J)=j$, this point can be connected by a tie, keeping planarity, to $j$ blocks of $J$, obtaining $c(J')=k$ for all $k\in [j]$; however, it can be also non connected, so that $c(J')=j+1$. Similarly, if $m=2n$ and $I'$ contains $\{n\}$ and $\{n+1\}$, i.e., it has two points at right, for every element $(I,J)$ of $\T_{m-1,j}$ we obtain exactly one element $(I',J')$ of $\T_{m,k}$ for all $k\in[j+1]$, see Figure \ref{figNumT(n,m)B}. But another element must be counted, corresponding to the case in which $I'$ contains $\{n,n+1\}$, i.e., it has a line at right. In this case  evidently $c(J')=1$, therefore $\tau_{j,1}=2$.
\end{proof}

\begin{figure}[H]
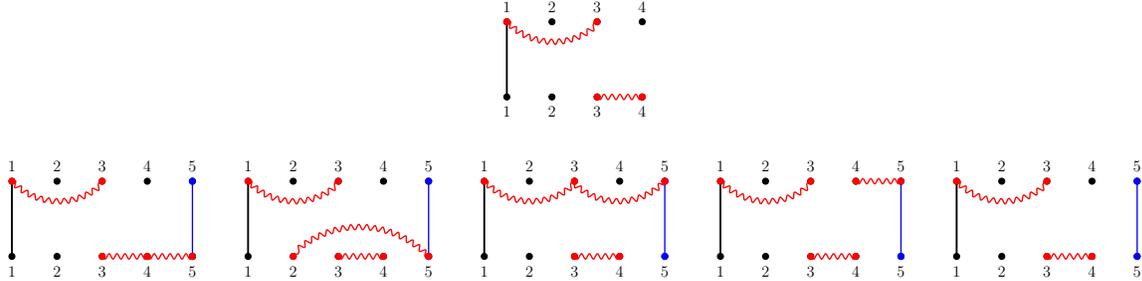

\figurethisev
\caption{Elements of $\T_{9,k}$ with $k\in[5]$ obtained from an element of $\T_{8,4}$.}
\label{figNumT(n,m)A}
\end{figure}
\begin{figure}[H]
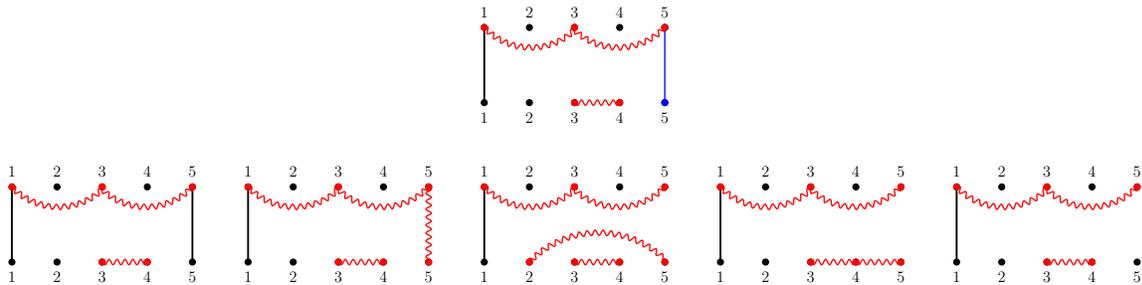

\figureforone
\caption{Elements of $\T_{10,k}$ with $k\in[4]$ obtained from an element of $\T_{9,3}$.}
\label{figNumT(n,m)B}
\end{figure}

\begin{table}[H]
\begin{tabular}{|c|ccccccc|}\hline
\backslashbox{$\kern-0.2em m\kern-0.5em$\\[-0.8mm]}{$\kern-0.5em k\kern-0.2em$\\[-1.5mm]}&$1$&$2$&$3$&$4$&$5$&$6$&$7$\\\hline
$1$&$1$&&&&&&\\
$2$&$2$&$1$&&&&&\\
$3$&$3$&$3$&$1$&&&&\\
$4$&$14$&$7$&$4$&$1$&&&\\
$5$&$26$&$26$&$12$&$5$&$1$&&\\
$6$&$140$&$70$&$44$&$18$&$6$&$1$&\\	
$7$&$279$&$279$&$139$&$69$&$25$&$7$&$1$\\
\hline\end{tabular}
\caption{First lines of the triangle $T(m,k)$.}
\end{table}

For each $i,j\in[n]$, define $\overline{x}_{i,j}\in\R(\I_n)$ as follows:
\[\overline{x}_{i,j}=\left\{\begin{array}{ll}
r_i\cdots r_{j-1}e_i\cdots e_{j-1}r_{i+1}\cdots r_j&\text{if }i<j,\\
q_i&\text{if }i=j,\\
r_{i+1}\cdots r_je_i\cdots e_{j-1}r_i\cdots r_{j-1}&\text{if }i>j,
\end{array}\right.\]and, if $i<j$, we define $\overline{e}_{i,j}=r_{i+1}\cdots r_{j-1}e_i\cdots e_{j-1}r_{i+1}\cdots r_{j-1}$ in $\R(\I_n)$. See Figure \ref{figGensXE}.
\begin{figure}[H]
\figurethinin
\caption{Elements $\overline{x}_{1,4}$, $\overline{x}_{4,1}$ and $\overline{e}_{1,4}$ respectively.}
\label{figGensXE}
\end{figure}

\begin{proposition}\label{GensPRIn}
The monoid $\PR(\I_n)$ is generated by $r_1,\ldots,r_n$, $q_1,\ldots,q_n$, $e_1,\ldots,e_{n-1}$.
\end{proposition}
\begin{proof}
Let $(I,J)=eqx_ge'$ in $\PR(\I_n)$ written in normal form as in Corollary \ref{NFormRInO}. Let $r$ be a product of $r_i$'s and let $q'$ be a product of $q_i$'s such that $q=rq'$, that is, $(I,J)=erye'$ with $y=q'x_g$. As $J$ is planar, $e$ can be written as a product of $\overline{e}_{i,j}$'s such that each $i$ and $j$ lie in a block of $J$ that contains no $k$ satisfying $i<k<j$. Similarly, $e'$ can be written as a product of $\overline{e}_{i,j}$'s such that $2n+1-i$ and $2n+1-j$ lie in a block of $J$ that contains no $k$ satisfying $2n+1-j<k<2n+1-i$. Note that $r_i$ occurs in $r$ whenever $\{i\}$ is a point of $I$ and $i$ is not the minimum of a block of $J$ containing a line but containing no lines of $I$. Finally, by planarity and definition of $x_g$, the element $y$ can be written as a product of $\overline{x}_{i,j}$'s, where, for every block $B$ of $J$ non containing lines of $I$, $i=\min(B\cap[n])$ and $2n+1-j=\max(B\cap[n+1,2n])$. This proves that neither $x_{i,j}$ nor $e_{i,j}$ is needed.
\end{proof}

\begin{corollary}[Normal form]\label{NFormPRIn}
The word $erye'$ obtained during the proof of Proposition \ref{GensPRIn} induces a normal form in the free monoid $(\{r_i,q_i\mid i\in[n]\}\cup\{e_i\mid i\in[n-1]\})^*$.
\end{corollary}
\begin{proof}
By definition, $\overline{e}_{i,j}$ commutes with $\overline{e}_{h,k}$ if $[i,j]\cap[h,k]\subset\{i,j,h,k\}$ and contains at most one element. This implies that $\overline{e}_{i,j}$'s can be partially ordered, indeed, if $i<h<k<j$, then $\overline{e}_{h,k}$ precedes $\overline{e}_{i,j}$ in $e$, while $e_{h,k}$ follows $e_{i,j}$ in $e'$. On the other hand, $\overline{x}_{i,j}$ commutes with $\overline{x}_{h,k}$ only if $[\min(i,j),\max(i,j)]$ and $[\min(h,k),\max(h,k)]$ are disjoint. Thus, $\overline{x}_{i,j}$'s can be partially ordered as well, indeed, if $i<h\leq j<k$, then $\overline{x}_{i,j}$ precedes $\overline{x}_{h,k}$ and $\overline{x}_{j,i}$ precedes $\overline{x}_{k,h}$. We get a normal form by sorting the $\overline{e}_{i,j}$'s and $\overline{x}_{i,j}$'s occurring in $erye'$.
\end{proof}
\begin{figure}[H]
\figurethieig
\caption{Normal form of $(I,J)$ in $\R(\I_5)$, where $I=r_2r_3r_4r_5$ and $J=(\{1,3,10\},\{2\},\{4,9\},\{5,7\},\{6\},\{8\})$.}
\label{XtoRER}\end{figure}

Let $\Gamma_n$ be the monoid presented by generators $\rho_1,\ldots,\rho_n$, $\pi_1,\ldots,\pi_n$ and $\e_1,\ldots,\e_n$, subject to the relations (\ref{rh1}), (\ref{ep1}), (\ref{qs1}), (\ref{qs2}), (\ref{qis1}), (\ref{qis2}) and (\ref{qis3}).

\begin{lemma}\label{PiGammaLem}
Let $a,b\in\PP_n$, and let $a',b'\in\Gamma_n$ obtained from $a,b$ respectively, by replacing each $h_i$ with $\e_i$ and each $r_i$ with $\rho_i$. If $a=b$, then $\pi a'=\pi b'$, where $\pi$ is the product of all $\pi_i$ such that $r_i$ occurs in $a$ or in $b$.
\end{lemma}
\begin{proof}
Observe that when replacing $h_i$ with $\e_i$ and $r_i$ with $\rho_i$, relations (\ref{BnT1}), (\ref{BnT2}), (\ref{BnT3}) and (\ref{BnT5}) become part of (\ref{rh1}), (\ref{ep1}), (\ref{qis1}) and (\ref{qis3}). This does not occur with (\ref{BnT4}), however (\ref{qs1}), (\ref{qs2}) and (\ref{qis1}) imply $\pi_i(\epsilon_j\rho_i\epsilon_j)=\pi_i(\pi_i\epsilon_j)=\pi_i(\epsilon_j)$ for each $j\in\{i,i+1\}$. Thus, if $a=b$ in $\PP_n$, then $\pi a'=\pi b'$ in $\Gamma_n$.
\end{proof}

\begin{theorem}\label{presPRIn}
The monoid $\PR(\I_n)$ is presented by generators $r_1,\ldots,r_n$, $q_1,\ldots,q_n$ and $e_1,\ldots,e_{n-1}$, satisfying (\ref{r1}), (\ref{TSn1}), (\ref{ris1}), (\ref{ris2}), (\ref{ris4}), (\ref{ris5}) and (\ref{ris6}).
\end{theorem}
\begin{proof}
Let $\lambda:\Gamma_n\to\R(\IS_n)$ sending $\rho_i\mapsto r_i$, $\e_i\mapsto e_i$ and $\pi_i\mapsto q_i$. Note that each relation of $\Gamma_n$ holds in $\PR(\I_n)$ under $\lambda$, so Proposition \ref{GensPRIn} implies that $\lambda(\Gamma_n)=\PR(\I_n)$. Also, let $\phi:\Gamma_n\to\PP_n$ sending $\rho_i\mapsto r_i$, $\e_i\mapsto h_i$ and $\pi_i\mapsto 1$. As above, each relation of $\Gamma_n$ holds in $\PP_n$ under $\phi$, so it is an epimorphism as well.

Observe that, if $x\in\Gamma_n$ with $\lambda(x)=(I,J)$, then $\phi(x)=J$. Further, in virtue of (\ref{qs2}) and (\ref{qis2}), if $\rho_{i_1},\ldots,\rho_{i_k}$ and  $\pi_{j_1},\ldots,\pi_{j_t}$ are the $\rho_i$'s and $\pi_i$'s generators occurring in $x$, then each $\pi_j$ with the same index of a $\rho_i$ can be removed. So, we can assume that $\{i_1,\ldots,i_k\}$ and $\{j_1,\ldots,j_t\}$ are disjoint. Note also that $x$ has such generators if and only if $I=r_{i_1}\cdots r_{i_k}r_{j_1}\cdots r_{j_t}$. Again, because of (\ref{qs2}), (\ref{qis1}) and (\ref{qis2}), we have $x=\pi x'$, where $\pi=\pi_{i_1}\cdots\pi_{i_k}\pi_{j_1}\cdots\pi_{j_t}$ and $x'$ is obtained by removing the $\pi_i$'s from $x$.

Now, let $x,y\in\Gamma_n$ such that $\lambda(x)=(I,J)=\lambda(y)$, so, as mentioned above, we have $\phi(x)=J=\phi(y)$ with $x=\pi x'$ and $y=\pi y'$. So, Lemma \ref{PiGammaLem} implies $x=\pi x'=\pi y'=y$. Therefore $\lambda$ is an isomorphism.
\end{proof}

\subsection{The planar ramified Jones monoid}\label{5.3}

Here we prove that the planar ramified monoid $\PR(\J_n)$ is isomorphic to the monoid $\tTL_n$ obtained as a specialization of the {\it tied Temperley--Lieb algebra}, defined in \cite{AiJuPa2023}. See Theorem \ref{PRJnisomTTL}. The monoid $\tTL_n$ is presented by generators $\e_1,\ldots,\e_{n-1}$ satisfying (\ref{ep1}), and $\tau_1,\ldots,\tau_{n-1}$, $\phi_1,\ldots,\phi_{n-1}$, subject to the following relations:\begin{align}
\tau_i^2=\tau_i,&\quad\tau_i\tau_j=\tau_j\tau_i,\quad|i-j|>1,\\
\tau_i\tau_j\tau_i=\tau_i,&\quad|i-j|=1,\\
\phi_i^2=\phi_i,&\quad\phi_i\phi_j=\phi_j\phi_i,\quad|i-j|>1,\\
\e_i\tau_i=\tau_i\e_i=\tau_i,&\quad\e_i\phi_j=\phi_j\e_i,\quad\phi_i\e_i=\phi_i\quad\tau_i\phi_i=\phi_i\tau_i,\\
\e_i\tau_j=\tau_j\e_i,&\quad\phi_i\tau_j=\tau_j\phi_i,\quad|i-j|>1,\\
\tau_i\e_j\tau_i=\tau_i,&\quad\phi_i\e_j=\e_j\phi_i=\e_j\tau_i\e_j,\quad|i-j|=1.
\end{align}
Note that the submonoid of $\tTL_n$ generated by the $\tau_i$'s is isomorphic to the Jones monoid $\J_n$.

For each $i\in[n-1]$ we will denote by $f_i$ the ramified partition $(t_i,h_i)$, which is represented as in Figure \ref{efeGen}.
\begin{figure}[H]
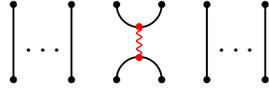

\figureforfiv
\caption{Generator $f_i$.}
\label{efeGen}
\end{figure}

\begin{theorem}\label{PRJnisomTTL}
The monoids $\PR(\J_n)$ and $\tTL_n$ are isomorphic. In consequence, $\PR(\J_n)$ is presented by generators $t_1,\ldots,t_{n-1}$ satisfying (\ref{JnRels}), $e_1,\ldots,e_{n-1}$ satisfying (\ref{TSn1}), and $f_1,\ldots,f_{n-1}$, subject to the following relations:\begin{align}
f_i^2=f_i,&\quad f_if_j=f_jf_i,\quad|i-j|>1,\\
e_it_i=t_ie_i=t_i,&\quad e_if_j=f_je_i,\quad f_ie_i=f_i,\quad t_if_i=f_it_i,\\
e_it_j=t_je_i,&\quad f_it_j=t_jf_i,\quad|i-j|>1,\\
t_ie_jt_i=t_i,&\quad f_ie_j=e_jf_i=e_jh_ie_j,\quad|i-j|=1.
\end{align}
\end{theorem}

To prove Theorem \ref{PRJnisomTTL}, we need the following proposition and some technical lemmas below.

\begin{proposition}
The cardinality of $\PR(\J_n)$ is the Fuss--Catalan number $C(4,n)$.
\end{proposition}
\begin{proof}
Let $M$ be the set of ramified partitions $(I,J)$ of $[2n]$ such that $I$ is planar and each of its blocks contains exactly two elements. It was shown in \cite{Ai20} that, for every even integer $n$, the number of elements in $M$ is the Fuss--Catalan number $C(4,n)=\frac{1}{3n+1}\binom{4n}{n}$ \cite[A002293]{OEIS}. To show the result, for every ramified partition $(I,J)\in M$, it is enough to consider a topological open disc containing the tied arc diagram of $(I,J)$ excluding the middle point of the interval in which the endpoints of the arcs are located. So, we get an element of $\PR(\J_n)$ by a continuous deformation of this disc. See Figure \ref{openDiscDef}. 
\begin{figure}[H]
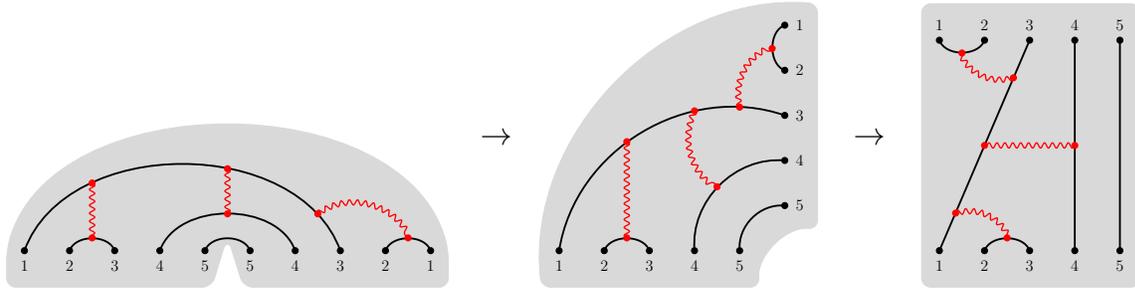

\figureforfor
\caption{Continuous transformation of the open disc.}
\label{openDiscDef}
\end{figure}
Since no crossings are introduced, the transformation above defines a bijection between $\PR(\J_n)$ and $M$. Thus, $|\PR(\J_n)|=C(4,n)$.
\end{proof}

\begin{lemma}
The mapping sending $\tau_i\mapsto t_i$, $\phi_i\mapsto f_i$ and $\e_i\mapsto e_i$ defines a homomorphism $\theta:\tTL_n\to\PR(\J_n)$.
\end{lemma}
\begin{proof}
The proof follows from the fact that the mapping sends each defining relation of $\tTL_n$ to an identity of $\PR(\J_n)$.
\end{proof}

To show that $\theta$ is an epimorphism (Proposition \ref{PRJgens}), we first need to prove Lemma \ref{PRJgensL1}, Lemma \ref{PRJgensL2} and Lemma \ref{PRJgensL3}.

For each $I\in\J_n$, we will denote by $a(I)$ the set of indexes $i\in[n-1]$ such that either $\{j<i+1\}$ is a block of $I$ or $\{i+1<j\}$ with $2n-j+1\leq i$ is a block of $I$. Similarly, we will denote by $b(I)$ the set of indexes $i\in[n-1]$ such that $\{j<2n-i+1\}$ is a block of $I$, where either $n<j$ or $i<j\leq n$. See Figure \ref{a(I)b(I)} for an instance.
\begin{figure}[H]
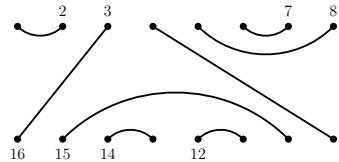

\figureforsix
\caption{For $I=\{\{1,2\},\{3,16\},\{4,9\},\{5,8\},\{6,7\},\{10,15\},\{11,12\},\allowbreak\{13,14\}\}$, we have $a(I)=\{1,2,6,7\}$ and $b(I)=\{1,2,3,5\}$.}
\label{a(I)b(I)}
\end{figure}

\begin{remark}\label{remProc}
Note that $j\in a(I)$ if and only if $j+1$ is the upper right endpoint of an arc of $I$, and $i\in b(I)$ if and only if $2n-i+1$ is the lower left endpoint of an arc of $I$. Observe that in both cases if the arc is a line then its upper endpoint is to the right of the lower endpoint. This implies that $|a(I)|=|b(I)|$ because the number of arcs with both endpoints $\le n$, and the number of arcs with both endpoints $>n$ coincide. Also, note that, because of planarity, the sets $a(I)$ and $b(I)$ determine uniquely the set partition $I$.
\end{remark}

In what follows, for $(I,J)\in\{t_i,f_i,e_i\}$, the subset of $I$ formed by the two arcs that are contained in $\{i,i+1,2n-i+1,2n-i\}$ will be called the {\it core} of $(I,J)$.

For each $i,j\in[n-1]$ with $i\leq j$, define $H(i,j)=\tau_j\tau_{j-1}\cdots\tau_i$.

\begin{lemma}\label{PRJgensL1}
For every $I\in\J_n$, we have $\theta(H(i_1,j_1)\cdots H(i_k,j_k))=(I,I)$, where $a(I)=\{j_1<\cdots<j_k\}$ and $b(I)=\{i_1<\ldots<i_k\}$.
\end{lemma}
\begin{proof}
({Cf. \cite[Aside 4.1.4.]{Jo1983}}) Let $h=H(i_1,j_1)\cdots H(i_k,j_k)$. Since $i_1<\ldots<i_k$ and $j_1<\cdots<j_k$, then for each $j\in a(I)$, the first occurrence of $\tau_j$ in $h$ is encountered in $H(i,j)$ for some $i\in b(I)$. Moreover, if a generator $\tau_q$ precedes $\tau_j$ in $h$, then $q<j$, and so the core of $\theta(\tau_q)$ appears on the left of the one of $\theta(\tau_j)$ in $\theta(h)$. Therefore, the upper endpoint $j+1$ of the core of $t_j=\theta(\tau_j)$ cannot be the lower endpoint of another core, that is, the letter $\tau_{j+1}$ does not occur on the left of $\tau_j$ in $h$. This implies that $j+1$ is an endpoint of an arc of $I$. Similarly, for each $i\in b(I)$, the last occurrence of $\tau_i$ in $h$ is encountered in $H(i,j)$ for some $j\in a(I)$, so the lower endpoint $2n-i+1$ of $t_i=\theta(\tau_i)$ cannot be the upper endpoint of another core, which implies that it is an endpoint of an arc of $I$ as well. Thus, the lemma follows from Remark \ref{remProc}.
\end{proof}
See Figure \ref{(I,I)exampPRJ} for an example.
\begin{figure}[H]
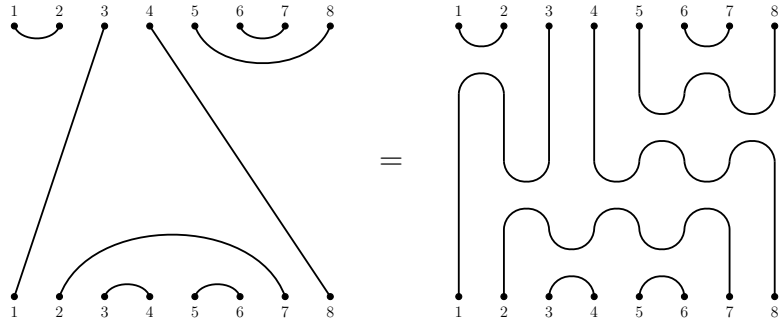

\figureforeig
\caption{The element $(I,I)$ with $I=\{\{1,2\},\{3,16\},\{4,9\},\{5,8\},\allowbreak\{6,7\},\{10,15\},\{11,12\},\{13,14\}\}$ is obtained from $\tau_1\cdot\tau_2\cdot\tau_6\tau_5\tau_4\tau_3\cdot\tau_7\tau_6\tau_5$.}
\label{(I,I)exampPRJ}
\end{figure}

For each $i\in[2n]$, we set $i'=2n-i+1$ if $i>n$ and $i'=i$ otherwise. In the sequel we will denote arcs by $(i,j)$ whenever $i'<j'$. Two arcs $(i,j)$ and $(h,k)$ of a set partition of $\J_n$ are said to be {\it vertically adjacent} if $i'<k'$, $h'<j'$ and a tie connecting them does not cross other arcs. A tie is called {\it vertical} if it connects two vertically adjacent arcs, otherwise it is called {\it horizontal}. A ramified partition $(I,J)\in\R(\J_n)$ is called {\it flat} if it can be represented with no horizontal ties. See Figure \ref{HorizRPart}.
\begin{figure}[H]
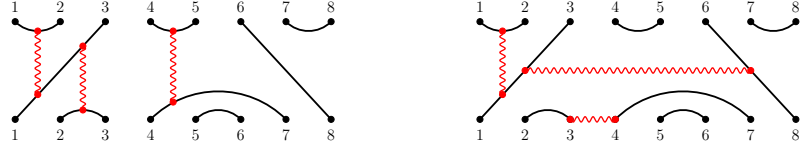

\figureforsev
\caption{A flat ramified partition together with a non flat one.}
\label{HorizRPart}
\end{figure}

\begin{lemma}\label{PRJgensL2}
For every flat $(I,R)\in\PR(\J_n)$, we have $\theta(\hat{H}(i_1,j_1)\cdots\hat{H}(i_k,j_k))=(I,R)$, where $a(I)=\{j_1<\cdots<j_k\}$, $b(I)=\{i_1<\ldots<i_k\}$ and $\hat{H}(i,j)=\hat{\tau}_j\hat{\tau}_{j-1}\cdots\hat{\tau}_i$ for all $i\leq j$, with $\hat{\tau}_p=\phi_p$ if $p$ is the minimum index such that the core of $\theta(\hat{\tau}_p)$ is contained in a pair of arcs that belong to the same block of $R$, and $\hat{\tau}_p=\tau_p$ otherwise.
\end{lemma}
\begin{proof}
Let $\hat{h}:=\hat{H}(i_1,j_1)\cdots\hat{H}(i_k,j_k)$. Lemma \ref{PRJgensL1} implies that $I=\theta(H(i_1,j_1)\cdots H(i_k,j_k))$, so, for each $\hat{\tau}_i$ occurring in $\hat{h}$, the arcs of $I$ involved in the core of $\theta(\hat{\tau}_i)$ are well defined. Thus, choosing $\phi_i$ instead of $\tau_i$ for $\hat{\tau}_i$ means putting a tie between the arcs of the core of $\theta(\hat{\tau}_i)$, which guarantees that these arcs belong to the same block of $R$. Now, it remains to prove that if two vertically adjacent arcs are in the same block, then they will take part in the core of $\theta(\hat{\tau}_i)$ for some $\hat{\tau}_i\in\{\tau_i,\phi_i\}$. For this, it is enough to note that an arc $(q,r)$ is realized by means of $r'-q'-1$ elements $\theta(\hat{\tau}_i)$ with $i\in[q',r'-1]$ of such types. Since $(q,r)$ and $(s,v)$ are vertically adjacent, then $q'<v'$ and $s'<r'$. Moreover, since these arcs are connectible by a tie, their realizations share the nonempty set of elements $\theta(\hat{\tau}_i)$ with $i\in [q',r'-1]\cap[s',v'-1]$. We set $\hat{\tau}_p=\phi_p$, where $p$ is the minimum of such indices.
\end{proof}
See Figure \ref{(I,R)FexampPRJ} for an example.
\begin{figure}[H]
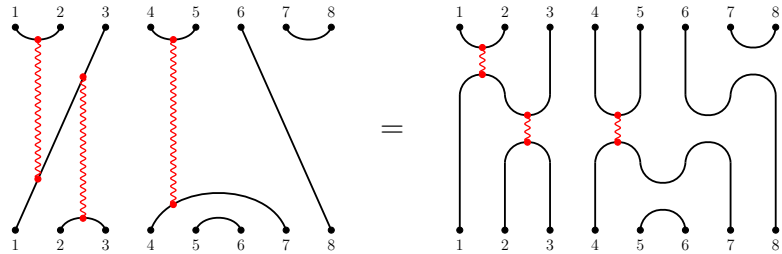

\figurefornin
\caption{The element $(I,R)$ with $I=\{\{1,2\},\{3,16\},\{4,5\},\{6,9\},\allowbreak\{7,8\},\{10,13\},\{11,12\}\}$ and $R=\{\{1,2,3,14,15,16\},\{4,5,10,13\},\{6,9\},\allowbreak\{7,8\},\{11,12\},\{14,15\}\}$ is obtained from $\phi_1\cdot\phi_2\cdot\phi_4\cdot\tau_7\tau_6\tau_5$.}
\label{(I,R)FexampPRJ}
\end{figure}

For $(I,R)\in\PR(\J_n)$, we will denote by $\tilde{I}$ the planar set partition obtained from $I$ by replacing each horizontally tied arcs $(i,j),(h,k)$ with $j'<h'$, by the vertically adjacent arcs $(i,k),(j,h)$. Note that the ramified partition $(\tilde{I},R)$ is flat.

\begin{lemma}\label{PRJgensL3}
For every non flat $(I,R)\in\PR(\J_n)$, we have $\theta(\tilde{H}(i_1,j_1)\cdots\tilde{H}(i_k,j_k))=(I,R)$, where $\tilde{H}(i_1,j_1)\cdots\tilde{H}(i_k,j_k)$ is obtained by applying Lemma \ref{PRJgensL2} to $(\tilde{I},R)$ and replacing $\phi_p$ with $\epsilon_p$ if the corresponding arcs of $\theta(\tau_p)$ belong to modified arcs of $\tilde{I}$.
\end{lemma}
\begin{proof}
Observe that replacing $I$ by $\tilde{I}$ in $(I,R)$ is allowed because the presence of the horizontal ties guarantees that the new pieces of arcs introduced do not cross other arcs nor other ties. Moreover, note that when $\tilde{\tau}_p$ is replaced by $\e_p$, the original arcs are recovered. The sole difference is that possible new elements $\tau_i$ are introduced, such that the core of $\theta(\tau_i)$ contains arcs belonging to the same arc.
\end{proof}
See Figure \ref{(I,R)NFexampPRJ1} and Figure \ref{(I,R)NFexampPRJ} for examples.
\begin{figure}[H]
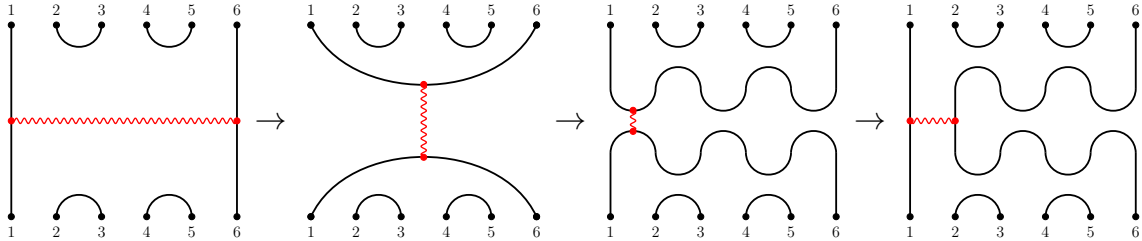

\figurefifone
\caption{The element $(I,R)$ with $I=\{\{1,12\},\{2,3\},\{4,5\},\{6,7\},\allowbreak\{8,9\},\{10,11\}\}$ and $R=\{\{1,6,7,12\},\{2,3\},\{4,5\},\{8,9\},\{10,11\}\}$ is obtained from $\tau_2\epsilon_1\cdot\tau_4\tau_3\tau_2\cdot\tau_5\tau_4$.}
\label{(I,R)NFexampPRJ1}
\end{figure}

\begin{proposition}\label{PRJgens}
The monoid $\PR(\J_n)$ is generated by $t_1,\ldots,t_n$, $f_1,\ldots,f_n$, $e_1,\ldots,e_{n-1}$. In consequence, the map $\theta:\tTL_n\to\PR(\J_n)$ is an epimorphism.
\end{proposition}
\begin{proof}
It is a consequence of Lemma \ref{PRJgensL1}, Lemma \ref{PRJgensL2} and Lemma \ref{PRJgensL3}.
\end{proof}

Note that Lemma \ref{PRJgensL1}, Lemma \ref{PRJgensL2} and Lemma \ref{PRJgensL3} together with Proposition \ref{PRJgens} give a normal form for elements of $\PR(\J_n)$. See Figure \ref{(I,R)NFexampPRJ}.
\begin{figure}[H]
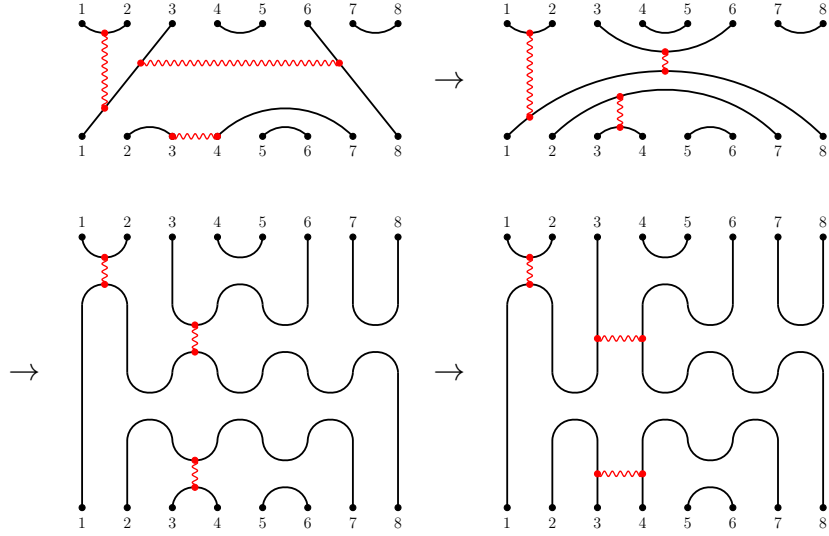

\figurefifty
\caption{The element $(I,R)$ with $I=\{\{1,2\},\{3,16\},\{4,5\},\{6,9\},\allowbreak\{7,8\},\{10,13\},\{11,12\},\{14,15\}\}$ and $R=\{\{1,2,3,6,9,16\},\{4,5\},\{7,8\},\allowbreak\{10,13,14,15\},\{11,12\}\}$ can be written as $f_1\cdot t_4e_3t_2\cdot t_5t_4e_3\cdot t_7t_6t_5$.}
\label{(I,R)NFexampPRJ}
\end{figure}

\begin{proof}[Proof of Theorem \ref{PRJnisomTTL}]
For every $u\in\tTL_n$, denote by $\bar{u}$ the element obtained from $u$ by removing the $\epsilon_i$'s and replacing each $\phi_i$ by $\tau_i$. Let $u,v\in\tTL_n$ such that $\theta(u)=\theta(v)=:(I,R)$. Thus $\theta(\bar{u})=\theta(\bar{v})=(I,I)$. As $(I,I)\in\J_n$, then $\bar{u}=\bar{v}$ in $\tTL_n$. By proceeding as \cite[Lemma 41]{AiArJu2023} we get that $u=v$ in $\tTL_n$. Therefore $\theta$ is injective, so Proposition \ref{PRJgens} implies that $\theta$ is an isomorphism.
\end{proof}

\section*{Appendix}

In Table \ref{tableCard} we collect the cardinalities, for $n\leq6$, of the monoids studied in this paper:
\begin{table}[H]
\begin{tabular}{|c|c|c|c|c|c|c|c|c|}\hline
Monoid&Section&$n=1$&$2$&$3$&$4$&$5$&$6$&OEIS \cite{OEIS}\\\hline
$\R(\IS_n)$&\ref{sec4-1}&$3$&$39$&$971$&$38140$&$2126890$&$157874467$&\\\hline
$\R(\IS_n^\bullet)$&\ref{sec4-1}&$1$&$9$&$172$&$5545$&$264147$&$17194680$&\\\hline
$\R(\I_n)$&\ref{5.2.1}&$3$&$27$&$409$&$9089$&$272947$&$10515147$&A216078$(2n)$\\\hline
$\R(\I^{\bullet}_n)$&\ref{5.2.1}&$1$&$7$&$87$&$1657$&$43833$&$1515903$&A216078$(2n-1)$\\\hline
$\PR(\I_n)$&\ref{5.2.2}&$3$&$26$&$279$&$3302$&$41390$&$538996$&A002293$(2n)$\\\hline
$\PR(\I^\bullet_n)$&\ref{5.2.2}&$1$&$7$&$70$&$799$&$9794$&$125606$&A002293$(2n-1)$\\\hline
$\PR(\J_n)$&\ref{5.3}&$1$&$4$&$22$&$140$&$969$&$7084$&A001764\\\hline
\end{tabular}
\caption{Cardinalities}
\label{tableCard}
\end{table}

\subsubsection*{Acknowledgements}

The third author was supported partially by the grant FONDECYT Regular Nro. 1210011.

\end{document}